\documentclass{article}
\usepackage{latexsym, amscd, amsfonts, eucal, mathrsfs, amsmath, amssymb, amsthm, xypic,xr, makeidx, stmaryrd}
\def\Z{\mathbf{Z}}
\def\C{\mathbf{C}}

\DeclareMathOperator{\Mul}{Mul}

\DeclareMathOperator{\Spectra}{Sp}

\DeclareMathOperator{\Proj}{Proj}
\DeclareMathOperator{\CDer}{{\mathcal D}er}

\DeclareMathOperator{\LPrest}{L\mathcal{P}r}
\DeclareMathOperator{\RPrest}{R\mathcal{P}r}
\DeclareMathOperator{\LPress}{\widehat{\Cat}_{\infty}^{\LPrest}}
\DeclareMathOperator{\RPress}{\widehat{\Cat}_{\infty}^{\RPrest}}

\DeclareMathOperator{\Ring}{Ring}
\DeclareMathOperator{\MRing}{Ring^{+}}
\DeclareMathOperator{\Stabb}{\sigma}

\DeclareMathOperator{\aug}{aug}
\DeclareMathOperator{\Post}{Post}

\newcommand{\Andre}{Andr\'{e}}

\DeclareMathOperator{\RPressStab}{\widehat{\Cat}_{\infty}^{\RPrest, \Stabb}}

\newcommand{\EInfty}{ {\mathfrak E}_{\infty}}

\newcommand{\Kahler}{K\"{a}hler\,}
\DeclareMathOperator{\Der}{Der}

\DeclareMathOperator{\conn}{conn}

\DeclareMathOperator{\Exc}{Exc}

\newcommand{\seg}[1]{{\langle #1 \rangle}}

\DeclareMathOperator{\Alg}{Alg}
\DeclareMathOperator{\CAlg}{CAlg}

\DeclareMathOperator{\LFun}{\Fun^{L}}

\newcommand{\toposref}[1]{T.\ref{HTT-#1}}
\newcommand{\stableref}[1]{S.\ref{STA-#1}}
\newcommand{\monoidref}[1]{M.\ref{MON-#1}}
\newcommand{\symmetricref}[1]{C.\ref{SYM-#1}}

\newcommand{\bicatref}[1]{B.\ref{BIC-#1}}

\newcommand{\degree}{\text{o}}

\DeclareMathOperator{\sCoNerve}{\mathfrak{C}}

\DeclareMathOperator{\Cat}{\mathcal{C}at}
\newcommand{\h}[1]{\rm{h} \! #1}

\newcommand{\Adjoint}[4]{\xymatrix@1{#2 \ar@<.4ex>[r]^-{#1} & #3 \ar@<.4ex>[l]^-{#4}}}

\DeclareMathOperator{\coker}{coker}
 
\newcommand{\etale}{{\'{e}tale}\,\,}
\newcommand{\etalenospace}{{\'{e}tale}}
\newcommand{\et}{\'{e}t}
\newcommand{\mathet}{\text{\et}}
\newcommand{\Etale}{\'{E}tale\,}
\newcommand{\bigdot}{\bullet}

\DeclareMathOperator{\Stab}{Stab}

\DeclareMathOperator{\bHom}{Map}

\DeclareMathOperator{\Sym}{Sym}

\DeclareMathOperator{\Mod}{\mathcal{M}od}
\DeclareMathOperator{\CMod}{c\mathcal{M}od}

\DeclareMathOperator{\calM}{\mathcal{M}}

 \DeclareMathOperator{\G}{\mathbf{G}}

\DeclareMathOperator{\Ext}{Ext} 
 
\DeclareMathOperator{\DerRing}{\mathcal{SCR}}

\DeclareMathOperator{\Q}{\mathbf{Q}}
\DeclareMathOperator{\Tor}{Tor} 
\DeclareMathOperator{\colim}{colim}

\DeclareMathOperator{\bd}{\partial}

\newcommand{\Sphere}{S}

\DeclareMathOperator{\calE}{\mathcal{E}}

\DeclareMathOperator{\calO}{\mathcal{O}}

\DeclareMathOperator{\FinSpace}{\mathcal{S}_{\ast}^{fin}}

\DeclareMathOperator{\Spec}{{\bf Spec}}

\DeclareMathOperator{\Hom}{Hom} 
\DeclareMathOperator{\HH}{H} 
\DeclareMathOperator{\id}{id} \DeclareMathOperator{\Fun}{Fun}
\DeclareMathOperator{\calC}{\mathcal{C}}

\DeclareMathOperator{\SSet}{\mathcal{S}}

\DeclareMathOperator{\calX}{\mathcal{X}}

\DeclareMathOperator{\calD}{\mathcal{D}}
\DeclareMathOperator{\Ind}{Ind} 
\newtheorem{theorem}{Theorem}[subsection]
\newtheorem{lemma}[theorem]{Lemma}
\newtheorem{proposition}[theorem]{Proposition}

\newtheorem{corollary}[theorem]{Corollary}

\theoremstyle{definition}
\newtheorem{definition}[theorem]{Definition}

\newtheorem{example}[theorem]{Example}

\newtheorem{notation}[theorem]{Notation}

\newtheorem{warning}[theorem]{Warning}
\newtheorem{remark}[theorem]{Remark}

\makeindex

\begin{document}

\title{Derived Algebraic Geometry IV: Deformation Theory}

\maketitle
\tableofcontents

\section*{Introduction}

Let $X = \Spec A$ be a smooth affine algebraic variety, defined over the field $\C$ of complex numbers. The holomorphic cotangent bundle $T^{\ast}_{X}$ has the structure of an algebraic vector bundle over $X$, associated to a projective $A$-module $\Omega_{A/\C}$ of finite rank. The $A$-module
$\Omega_{A/\C}$ is called the {\it module of \Kahler differentials} of $A$ (relative to $\C$). In order to describe it, we first recall a bit of terminology. 
Let $M$ be an arbitrary $A$-module. A {\it $\C$-linear derivation} from $A$ into $M$ is a map of complex vector spaces $d: A \rightarrow M$ which satisfies the Leibniz rule $d(ab) = a (db) + b (da)$. Let $\Der_{\C}(A,M)$ denote the set of all $\C$-linear derivations
of $A$ into $M$. The module $\Omega_{A/ \C}$ is the receptacle for the {\em universal} derivation; in other words, there exists a derivation $d: A \rightarrow \Omega_{A/\C}$ such that, for every $A$-module $M$, composition with $d$ induces an isomorphism $$ \Hom_{A}( \Omega_{A/\C}, M) \rightarrow \Der_{\C}(A, M).$$
More concretely, $\Omega_{A/\C}$ is generated (as an $A$-module) by symbols $\{ df \}_{f \in A}$, which are subject to the usual rules of calculus:
$$ d(f+g) = df + dg$$
$$ d(fg) = f (dg) + g (df)$$
$$ d \lambda = 0, \lambda \in \C.$$
The module of \Kahler differentials $\Omega_{A/\C}$ is a fundamental invariant of the algebraic variety $X$, and plays an important role in the study of deformations of $X$. The goal of this paper is to describe the analogue of $\Omega_{A/\C}$ in the case where the $\C$-algebra $A$ is replaced by an $E_{\infty}$-ring.

The first step is to reformulate the theory of derivations. Let $A$ be a $\C$-algebra and let $M$ be an $A$-module. Complex-linear derivations of $A$ into $M$ can be identified with
$\C$-algebra maps $A \rightarrow A \oplus M$, which are sections of the canonical projection
$A \oplus M \rightarrow A$. Here $A \oplus M$ is endowed with the ``trivial square-zero''
algebra structure, described by the formula $$ (a,m)(a',m') = (aa', am' + a'm).$$

The advantage of the above formulation is that it generalizes easily to other contexts. For example, suppose now that $A$ is an $E_{\infty}$-ring, and that $M$ is an $A$-module. The direct sum
$A \oplus M$ (formed in the $\infty$-category of spectra) admits a canonical $E_{\infty}$-ring structure, generalizing the trivial square-zero structure described above. Moreover, there is a canonical projection
$p: A \oplus M \rightarrow A$. We will refer to sections of $p$ (in the $\infty$-category of $E_{\infty}$-rings) as {\it derivations} of $A$ into $M$. As in the classical case, there is a universal example of
an $A$-module $M$ equipped with a derivation from $A$; this $A$-module is called the
{\it absolute cotangent complex of $A$} (or the {\it topological \Andre-Quillen homology} of $A$), and
will be denoted by $L_{A}$. 

The theory of the cotangent complex plays a fundamental role in the foundations of derived algebraic geometry. Let $A$ be a connective $E_{\infty}$-ring, and let $\pi_0 A$ denote the underlying ordinary commutative ring of connected components of $A$. The canonical map $\phi: A \rightarrow \pi_0 A$
should be viewed as an {\em infinitesimal} extension, whose kernel consists of ``nilpotents'' in
$A$. Consequently, the difference between the theory of (connective) $E_{\infty}$-rings and that of ordinary commutative rings can be reduced to problems in deformation theory, which are often conveniently phrased in terms of the cotangent complex. As a simple illustration of this principle, we offer the following example (Corollary \ref{twain}): if $f: A \rightarrow B$ is a morphism of connective $E_{\infty}$-rings
which induces an isomorphism of ordinary commutative rings $\pi_0 A \rightarrow \pi_0 B$, then
$f$ is an equivalence if and only if the {\em relative} cotangent complex $L_{B/A}$ vanishes.

Our goal in this paper is to define the cotangent complex $L_{A}$ of an $E_{\infty}$-ring $A$, and to 
study the associated deformation theory. We have divided this paper into four parts, whose contents we will now sketch; a more detailed summary can be found at the beginning of each part.

We will begin in \S \ref{gentheory} with a very general formalism. To every presentable $\infty$-category $\calC$, we will define a {\it tangent bundle} $T_{\calC}$. Roughly speaking, $T_{\calC}$ is an $\infty$-category whose objects can be viewed as pairs $(A, M)$, where $A \in \calC$ and
$M$ is an infinite loop object in the $\infty$-category $\calC^{/A}$ of objects of $\calC$ lying
over $A$. In this case, we can take the ``0th space'' of $M$, to obtain an object
of $\calC$ which we will denote by $A \oplus M$. The functor
$(A, M) \mapsto A \oplus M$ admits a left adjoint, given by
$A \mapsto (A, L_A)$. We will refer to $L$ as the {\it absolute cotangent complex functor} of $\calC$. 
Many of the basic formal properties of the cotangent complex can be 
established in the general setting. For example, to every morphism
$A \rightarrow B$ in $\calC$, we can define a {\em relative} cotangent complex
$L_{B/A}$. Moreover, to every commutative triangle
$$ \xymatrix{ & B \ar[dr]^{f} & \\
A \ar[ur] \ar[rr] & & C }$$
in $\calC$, we can associate a distinguished triangle
$$ f_{!} L_{B/A} \rightarrow L_{C/A} \rightarrow L_{C/B} 
\rightarrow
f_{!} L_{B/A}[1] $$
in the homotopy category of $\Stab( \calC^{/C} )$  (Corollary \ref{longtooth}). Here
$f_{!}: \Stab( \calC^{/B}) \rightarrow \Stab( \calC^{/C})$ denotes the base change functor
associated to $f$. 

Of course, we are primarily interested in the situation where $\calC$ is the $\infty$-category
of $E_{\infty}$-rings. In \S \ref{sec2}, we will identify the $\infty$-category
$T_{\calC}$ in this case. The main result, Theorem \ref{scummer2}, asserts that
$T_{\calC}$ can be identified with the $\infty$-category of pairs $(A,M)$, where
$A$ is an $E_{\infty}$-ring and $M$ is an $A$-module. In this case, the
functor $(A, M) \mapsto A \oplus M$ implements the idea sketched above: as a spectrum,
$A \oplus M$ can be identified with the coproduct of $A$ with $M$, and the multiplication
on $A \oplus M$ is trivial on $M$ (Remark \ref{toaster}). 

In \S \ref{sec3}, we will study the cotangent complex functor $A \mapsto L_{A}$ in the
setting of $E_{\infty}$-rings. Here our objectives are more quantitative. Our main result, Theorem \ref{tulbas}, asserts that the connectivity properties of a map $f: A \rightarrow B$ of connective $E_{\infty}$-rings are closely related to the connectivity properties of the relative cotangent complex $L_{B/A}$. 
This result has many consequences; for example, it implies that the cotangent complex $L_{B/A}$ can be used to test whether or not $f$ has good finiteness properties (Theorem \ref{sucker}). 
Our other objective in \S \ref{sec3} is to introduce the theory of \etale maps between $E_{\infty}$-rings, and to show that the relative cotangent complex of an \etale map vanishes (Proposition \ref{etrel}).

Our goal in the final section of this paper (\S \ref{sec4}) is to study square-zero extensions of $E_{\infty}$-rings. The idea is very general: given an $E_{\infty}$-ring $A$
and a map of $A$-modules $\eta: L_{A} \rightarrow M[1]$, we can build an associated
``square-zero'' extension $$ M \rightarrow A^{\eta} \rightarrow A.$$
The main results in this context are existence theorems, which assert that a large class
of maps $\widetilde{A} \rightarrow A$ arise via this construction (Theorems \ref{h2h2} and
\ref{exsqze}). We will apply these results to prove a crucial comparison result: for every $E_{\infty}$-ring $A$, the $\infty$-category of \etale $A$-algebras is equivalent to the ordinary category of
\etale $\pi_0 A$-algebras (Theorem \ref{turncoat}). This result will play an essential role in developing the foundations of derived algebraic geometry; see \cite{structured}.


\begin{remark}
The theory of the cotangent complex presented here is not new. For expositions in a similar spirit, we refer the reader to \cite{schwede} and \cite{basterra}.
\end{remark}

\begin{warning}
In this paper, we will generally be concerned with the {\em topological} version of \Andre-Quillen homology. This theory is closely related to the classical \Andre-Quillen theory, but generally yields different answers even for discrete commutative rings. More precisely, suppose that $R$ is a discrete commutative ring. In this case, the relative
topological \Andre-Quillen homology $L_{R/ \Z}$ admits the structure of an $(R \otimes_{\Sphere} \Z)$-module, where the tensor product is taken over the sphere spectrum $\Sphere$. The relative tensor product $$ L_{R/ \Z} \otimes_{R \otimes_{\Sphere} \Z} R$$
can be identified (as an $R$-module spectrum) with the classical cotangent complex $L^{\degree}_{R}$ constructed by Quillen as the nonabelian left derived functor of the \Kahler differentials. In particular, for each $i \geq 0$ we obtain an induced map
$$\phi_{i}: \pi_{i} L_{R} \rightarrow \pi_{i} L^{\degree}_{R}.$$ 
This map is an isomorphism for $i \leq 1$ and
a surjection when $i = 2$. Moreover, if $R$ is a $\Q$-algebra, then $\phi_{i}$ is an isomorphism
for all $i \in \Z$.
\end{warning}

\subsection*{Notation and Terminology}

Throughout this paper, we will freely use the theory of $\infty$-categories developed in
\cite{topoi}. We will also use \cite{DAGStable} as a reference for the theory of stable $\infty$-categories, \cite{monoidal} as reference for the theory of monoidal $\infty$-categories, and
\cite{symmetric} as a reference for the theory of symmetric monoidal $\infty$-categories and
$E_{\infty}$-rings. We will also use \cite{derivative} as a reference for some rudiments of the Goodwille calculus.

References to \cite{topoi} will be indicated by use of the letter T, references to \cite{DAGStable} will be indicated by use of the letter S, references to \cite{monoidal} will be indicated by use of the letter M,
references to \cite{symmetric} will be indicated by use of the letter C, and references to
\cite{derivative} will be indicated by use of the letter B.
For example, Theorem \toposref{mainchar} refers to Theorem \ref{HTT-mainchar} of \cite{topoi}.

If $p: X \rightarrow S$ is a map of simplicial sets and $s$ is a vertex of $S$, we will typically write
$X_{s}$ to denote the fiber $X \times_{S} \{s\}$. 
 
Let $n$ be an integer. We say that a spectrum $X$ is {\it $n$-connective} if
$\pi_{i} X$ vanishes for $i < n$. We say that $X$ is {\it connective} if it is $0$-connective.
We will say that a map $f: X \rightarrow Y$ of spectra is {\it $n$-connective} if 
the fiber $\ker(f)$ is $n$-connective.

We let $\Spectra$ denote the $\infty$-category of spectra. The $\infty$-category
$\CAlg( \Spectra)$ of $E_{\infty}$-rings will be denoted by $\EInfty$. We let
$\EInfty^{\conn}$ denote the full subcategory of $\EInfty$ spanned by 
the {\em connective} $E_{\infty}$-rings.

Let $\calC \rightarrow \calD$ be an inner fibration of $\infty$-categories. For each object
$D \in \calD$, we let $\calC_{D}$ denote the fiber $\calC \times_{\calD} \{ D \}$. We will apply a similar notation for functors: given a commutative diagram
$$ \xymatrix{ \calC \ar[rr]^{F} \ar[dr]^{p} & & \calC' \ar[dl]^{q} \\
& \calD & }$$
in which $p$ and $q$ are inner fibrations, 
we let $F_{D}$ denote the induced map of fibers $\calC_{D} \rightarrow \calC'_{D}$.
In this situation, we also let $\Fun_{\calD}( \calC, \calC')$ denote the 
fiber product $\Fun( \calC, \calC') \times_{ \Fun( \calC, \calD) } \{ p \}$. 

\section{The Cotangent Complex: General Theory}\label{gentheory}

Our goal in this section is to introduce the basic formalism underlying the theory of cotangent complex. 
Let us begin by reviewing the classical theory of \Kahler differentials.
Given a commutative ring $A$ and an $A$-module $M$, we define a {\em derivation} of $A$ into $M$ to be a section of the projection map $A \oplus M \rightarrow A$. This definition depends on our ability to endow the direct sum $A \oplus M$ with the structure of a commutative ring. To describe the situation a little bit more systematically, let $\Ring$ denote the category of commutative rings, and
$\MRing$ the category of pairs $(A, M)$, where $A$ is a commutative ring and $M$ is an
$A$-module. A morphism in the category $\MRing$ is a pair of maps $(f,f'): (A, M) \rightarrow (B, N)$,
where $f: A \rightarrow B$ is a ring homomorphism and $f': M \rightarrow N$ is a map of $A$-modules,
(here we regard $N$ as an $A$-module via transport of structure along $f$). Let
$G: \MRing \rightarrow \Ring$ be the square-zero extension functor given by the formula
$(A, M) \mapsto A \oplus M$. Then the functor $G$ admits a left adjoint $F$, which is described by the formula $F(A) = (A, \Omega_{A})$. Here $\Omega_{A}$ is the $A$-module of {\it absolute \Kahler differentials}: it is generated by symbols $\{ da \}_{a \in A}$, subject to the relations
$$ d(a+a') = da + da'$$
$$ d(aa') = a da' + a' da.$$
For every commutative ring $A$, the unit map $u_{A}: A \rightarrow (G \circ F)(A) = A \oplus \Omega_{A}$
is given by the formula $a \mapsto a + da$.

We now make two fundamental observations concerning the above situation:
\begin{itemize}
\item[$(1)$] In addition to the functor $G$, there is a forgetful functor
$G': \MRing \rightarrow \Ring$, given by $(A, M) \mapsto A$. 
Moreover, there is a natural transformation of functors from $G$ to $G'$, which
can itself be viewed as a functor from $\MRing$ into the category
$\Fun( [1], \Ring)$ of arrows in $\Ring$.

\item[$(2)$] For every commutative ring $A$, the fiber
${G'}^{-1} \{A \}$ is an abelian category (namely, the category of $A$-modules).
\end{itemize}
We wish to produce an analogous theory of derivations in the case where the category
$\Ring$ is replaced by an arbitrary presentable $\infty$-category $\calC$. 
What is the proper analogue of $\MRing$ in this general situation?
Observation $(1)$ suggests that we should choose another $\infty$-category
$\calC^{+}$ equipped with a functor $\calC^{+} \rightarrow \Fun( \Delta^1, \calC)$.
Observation $(2)$ suggests that the fibers of composite map
$$ \phi: \calC^{+} \rightarrow \Fun( \Delta^1, \calC) \rightarrow \Fun( \{1\}, \calC) \simeq \calC$$
should be ``abelian'' in some sense. There is a good $\infty$-categorical analogue of the theory of abelian categories: the theory of {\em stable} $\infty$-categories, as presented in \cite{DAGStable}.
It is therefore natural to require that the the fibers of $\phi$ be stable. It turns out that there
is a canonical choice for the $\infty$-category $\calC^{+}$ with these properties. We will refer to this canonical choice as the {\it tangent bundle} to $\calC$ and denote by $T_{\calC}$. Roughly speaking, an object of $T_{\calC}$ consists of a pair
$(A, M)$, where $A \in \calC$ and $M \in \Stab( \calC^{/A})$; here $\Stab$ denotes the {\em stabilization}
construction introduced in \S \stableref{stable9.1}. In \S \ref{sec2}, we will see that this really
is a good analogue of the algebraic situation considered above: if
$\calC$ is the $\infty$-category of $E_{\infty}$-rings, then $T_{\calC}$ can be identified
with the $\infty$-category of pairs $(A, M)$ where $A$ is an $E_{\infty}$-ring and $M$ is an $A$-module.

Once we have established the theory of tangent bundles, we can proceed to define
the analogue of the \Kahler differentials functor. Namely, for any presentable $\infty$-category $\calC$, we will define the {\it cotangent complex functor} $L: \calC \rightarrow T_{\calC}$ to be a left
adjoint to the forgetful functor 
$$T_{\calC} \rightarrow \Fun( \Delta^1, \calC) \rightarrow \Fun( \{0\}, \calC) \simeq \calC.$$
However, it is important to exercise some care here: in the algebraic situation, we want to
make sure that the cotangent complex $L_{A}$ of an $E_{\infty}$-ring produces an $A$-module.
In other words, we want to ensure that the composition
$$ \calC \stackrel{L}{\rightarrow} T_{\calC} \rightarrow \Fun( \Delta^1, \calC)
\rightarrow \Fun( \{1\}, \calC) \simeq \calC$$
is the identity functor. We will construct a functor $L$ with this property using the theory of
{\em relative adjunctions}, which we present in \S \ref{reladj}.

Given an object $A \in \calC$ and $M \in T_{\calC} \times_{\calC} \{A\}$, we can define the notion
of a {\em derivation} of $A$ into $M$. This can be described either as map from
$L_{A}$ into $M$ in the $\infty$-category $T_{\calC} \times_{\calC} \{A\}$, or as
a section of the canonical map $G(M) \rightarrow A$ in $\calC$. For many purposes, it is convenient to work in an $\infty$-category containing both $\calC$ and $T_{\calC}$, in which the morphisms are given by derivations. Such an $\infty$-category is readily available: namely, the correspondence associated to the pair of adjoint functors
$\Adjoint{L}{\calC}{T_{\calC}}{G}$, where $G$ and $L$ are defined as above.
We will call this $\infty$-category the {\it tangent correspondence} to $\calC$; an explicit construction will be given in \S \ref{gen1}.

In the classical theory of \Kahler differentials, it is convenient to consider the absolute \Kahler differentials $\Omega_{A}$ of a commutative ring $A$, but also the module of relative \Kahler differentials $\Omega_{B/A}$ associated to a ring homomorphism $A \rightarrow B$. In \S \ref{relcot} we will introduce an analogous {\em relative} version of the cotangent complex $L$. We will then establish some of the basic formal properties of the relative cotangent complex. For example, given a sequence of commutative ring homomorphisms $A \rightarrow B \rightarrow C$, there
is an associated short exact sequence
$$ \Omega_{B/A} \otimes_{B} C \rightarrow \Omega_{C/A} \rightarrow \Omega_{C/B} \rightarrow 0.$$
Corollary \ref{longtooth} provides an $\infty$-categorical analogue of this statement:
for every commutative diagram
$$ \xymatrix{ & B \ar[dr]^{f} & \\
A \ar[ur] \ar[rr] & & C}$$
in a presentable $\infty$-category $\calC$, there is an associated distinguished triangle
$$ f_{!} L_{B/A} \rightarrow L_{C/A} \rightarrow L_{C/B} \rightarrow f_{!} L_{B/A}[1]$$
in the triangulated category $\h \Stab( \calC^{/C} )$.

\subsection{Stable Envelopes and Tangent Bundles}\label{cotangent1}

In the last section, we introduced the definition of an extension structure on an $\infty$-category $\calC$.
In this section, we will show that every presentable category $\calC$ admits a natural
extension structure, which we will call the {\it tangent bundle} to $\calC$. We begin with some generalities on the stable envelope of $\infty$-categories.

\begin{definition}\label{defstabbb}
Let $\calC$ be a presentable $\infty$-category. A {\it stable envelope} ({\it pointed envelope}) of $\calC$ is a categorical fibration $u: \calC' \rightarrow \calC$ with the following properties:
\begin{itemize}
\item[$(i)$] The $\infty$-category $\calC'$ is stable (pointed) and presentable.
\item[$(ii)$] The functor $u$ admits a left adjoint.
\item[$(iii)$] For every presentable stable (pointed) $\infty$-category $\calE$, composition with
$u$ induces an equivalence of $\infty$-categories
$\Fun^{R}( \calE, \calC' ) \rightarrow \Fun^{R}( \calE, \calC).$
Here $\Fun^{R}( \calE, \calC' )$ denotes the full subcategory of $\Fun( \calE, \calC' )$ spanned by those functors which admit left adjoints, and $\Fun^{R}( \calE, \calC)$ is defined similarly.
\end{itemize}

More generally, suppose that $p: \calD \rightarrow \calC$ is a presentable fibration.
A {\it stable envelope} ({\it pointed envelope}) of $p$ is a categorical fibration
$u: \calC' \rightarrow \calC$ with the following properties:
\begin{itemize}
\item[$(1)$] The composition $p \circ u$ is a presentable fibration.
\item[$(2)$] The functor $u$ carries $(p \circ u)$-Cartesian morphisms of $\calC'$ to
$p$-Cartesian morphisms of $\calC$.
\item[$(3)$] For every object $D \in \calD$, the induced map
$\calC'_{D} \rightarrow \calC_{D}$ is a stable envelope (pointed envelope) of $\calC'_{D}$.
\end{itemize}
\end{definition}

\begin{remark}
Let $\calC$ be a presentable $\infty$-category, so that the projection
$p: \calC \rightarrow \Delta^0$ is a presentable fibration. It follows immediately from the definitions that
a map $u: \calC' \rightarrow \calC$ is a stable envelope (pointed envelope) of $\calC$ if and only if $u$ is a stable envelope (pointed envelope) of $p$.
\end{remark}

Let $p: \calC \rightarrow \calD$ be a presentable fibration, and let $u: \calC' \rightarrow \calC$ be a stable envelope (pointed envelope) of $u$. We will often abuse terminology by saying that $\calC'$ is a a {\it stable envelope} ({\it pointed envelope}) of $p$, or that {\it $u$ exhibits $\calC'$ as a stable envelope $(${\it pointed envelope}$)$ of $p$}. In the case where $\calD \simeq \Delta^0$, we will say instead that {\it $\calC'$ is a stable envelope $(${\it pointed envelope}$)$ of $\calC$}, or that {\it $u$ exhibits $\calC'$ as a stable envelope $(${\it pointed envelope}$)$ of $\calC$}.

\begin{remark}\label{soon}
Suppose given a pullback diagram of simplicial sets
$$ \xymatrix{ \calC_0 \ar[r] \ar[d]^{p_0} & \calC \ar[d]^{p} \\
\calD_0 \ar[r] & \calD }$$
where $p$ (and therefore also $p_0$) is a presentable fibration. If
$u: \calC' \rightarrow \calC$ is a stable envelope (pointed envelope) of the presentable fibration $p$, then
the induced map $\calC' \times_{\calC} \calC_0 \rightarrow \calC_0$ is a stable envelope
(pointed envelope) of the presentable fibration $p_0$.
\end{remark}

\begin{example}
Let $\calC$ be a presentable $\infty$-category and $\Stab(\calC)$ its stabilization. 
Then the composite map $\Stab(\calC) \stackrel{ \Omega^{\infty}_{\calC_\ast}}{\rightarrow} \calC_{\ast} \rightarrow \calC$ exhibits $\Stab(\calC)$ as a stable envelope of $\calC$. This follows immediately from
Corollary \stableref{mapprop}.
\end{example}

\begin{example}
Let $p: \calC \rightarrow \calD$ be a presentable fibration, and suppose that
each fiber $\calC_{D}$ of $p$ is pointed. Then the $\infty$-category
$\Spectra(p)$ of Definition \bicatref{linfib} is a stable envelope of $\calC$;
this can be deduced easily from Proposition \bicatref{jagger}.
\end{example}

\begin{example}\label{pinah}
Let $p: \calC \rightarrow \calD$ be a presentable fibration. We can explicitly construct a stable envelope of $\calC$ as follows. Let $\FinSpace$ denote the $\infty$-category of finite pointed spaces (Notation \stableref{finner}). We let $\calC'$ denote the full subcategory of the fiber product
$$ \Fun( \FinSpace, \calC) \times_{ \Fun( \FinSpace, \calD ) } \calD$$
spanned by those maps which correspond to {\em excisive} functors
$\FinSpace \rightarrow \calC_{D}$ for some object $D \in \calD$. Evaluation
on the zero sphere $S^0 \in \FinSpace$ induces a forgetful functor
$u: \calC' \rightarrow \calC$. The functor $u$ identifies $\calC'$ with a stable envelope of the presentable fibration $p$. The proof is easily reduced to the case where $\calD$ consists of a single point, in which case the result follows from Corollary \stableref{surritt}.
\end{example}

\begin{example}
Let $\calC$ be a presentable $\infty$-category. Let $\calC_{\ast} \subseteq \Fun( \Delta^1, \calC)$
denote the full subcategory of $\calC$ spanned by those morphisms $f: X \rightarrow Y$ such that
$X$ is a final object of $\calC$. Let $u: \calC_{\ast} \rightarrow \calC$ be given by evaluation
at the vertex $\{1\} \subseteq \Delta^1$. We claim that $u$ is a pointed envelope of $\calC$.

It is clear that $\calC_{\ast}$ is pointed (Lemma \toposref{pointerprime}) and that $u$ is a categorical fibration. Using Lemma \toposref{pointer}, we can identify $\calC_{\ast}$ with $\calC^{1/}$, where $1 \in \calC$ is a final object. It follows that
the forgetful functor $u: \calC_{\ast} \rightarrow \calC$ preserves limits (Proposition \toposref{needed17}) and filtered colimits (Proposition \toposref{goeselse}), and therefore admits a left adjoint (Corollary \toposref{adjointfunctor}). To complete the proof, it will suffice to show that if $\calD$ is a pointed presentable category, then composition with $u$ induces an equivalence
$$ \psi: \Fun^{R}( \calD, \calC_{\ast}) \rightarrow \Fun^{R}( \calD, \calC).$$
We now observe that $\Fun^{R}( \calD, \calC_{\ast})$ is isomorphic to the $\infty$-category
of pointed objects $\Fun^{R}( \calD, \calC)_{\ast}$. In view of Lemma \toposref{pointer}, the functor $\psi$ is an equivalence if and only if the $\infty$-category $\Fun^{R}( \calD, \calC)$ is pointed. We now
observe that $\Fun^{R}( \calD, \calC)^{op}$ is canonically equivalent to the full subcategory
$\LFun( \calC, \calD) \subseteq \Fun( \calC, \calD)$ spanned by the colimit-preserving functors.
Since $\calD$ has a zero object $0 \in \calD$, the $\infty$-category $\Fun( \calC, \calD)$ also has a zero object, given by the constant functor taking the value $0$. This functor preserves colimits, and is
therefore a zero object of $\LFun( \calC, \calD)$ as well. 
\end{example}

\begin{remark}\label{huggt}
Let $\calC$ be a presentable $\infty$-category. A stable envelope of $\calC$ is determined uniquely up to equivalence by the universal property given in Definition \ref{defstabbb}, and is therefore equivalent to $\Stab(\calC)$. More precisely, suppose we are given a commutative diagram
$$ \xymatrix{ \calC' \ar[rr]^{w} \ar[dr]^{u} & & \calC'' \ar[dl]^{v} \\
& \calC & }$$
in which $u$ and $v$ are stable envelopes of $\calC$. Then the functor $w$ is an equivalence of
$\infty$-categories (observe that in this situation, the functor $w$ automatically admits a left adjoint by virtue of Proposition \stableref{urtusk22}). Similar reasoning shows that pointed envelopes of $\calC$ are unique up to equivalence.
\end{remark}

Our next goal is to establish a relative version of Remark \ref{huggt}.
First, we need to introduce a bit of notation. Suppose we are given a diagram
$$ \xymatrix{ \calC \ar[dr]^{p} & & \calD \ar[dl]^{q} \\
& \calE & }$$
of $\infty$-categories, where $p$ and $q$ are presentable fibrations.
We let $\Fun^{R}_{\calE}( \calC, \calD)$ denote the full subcategory of
$\Fun_{\calE}(\calC, \calD)$ spanned by those functors $G: \calC \rightarrow \calD$
with the following properties:
\begin{itemize}
\item[$(i)$] The functor $G$ carries $p$-Cartesian edges of $\calC$ to $q$-Cartesian edges of $\calD$.
\item[$(ii)$] For each object $E \in \calE$, the induced functor $G_{E}: \calC_{E} \rightarrow \calD_{E}$
admits a left adjoint.
\end{itemize}
We let $\Fun^{R,\sim}_{\calE}( \calC, \calD)$ denote the largest Kan complex contained
in $\Fun^{R}_{\calE}(\calC, \calD)$.

\begin{proposition}\label{psycher}
Let $p: \calC \rightarrow \calD$ be a presentable fibration of $\infty$-categories.
Then there exists a functor $u: \calC' \rightarrow \calC$ with the following properties:
\begin{itemize}

\item[$(1)$] The functor $u$ is a stable envelope (pointed envelope) of the presentable fibration $p$.

\item[$(2)$] Let $q: \calE \rightarrow \calD$ be a presentable fibration, and assume that
each fiber of $q$ is a stable (pointed) $\infty$-category. Then composition with $u$ induces a trivial
Kan fibration
$$ \Fun^{R, \sim}_{\calD}( \calE, \calC') \rightarrow \Fun^{R, \sim}_{\calD}( \calE, \calC).$$

\item[$(3)$] Let $v: \calE \rightarrow \calC$ be any stable envelope (pointed envelope) of $p$. Then
$v$ factors as a composition
$$ \calE \stackrel{ \overline{v} }{\rightarrow} \calC' \stackrel{u}{\rightarrow} \calC,$$
where $\overline{v}$ is an equivalence of $\infty$-categories.
\end{itemize}
\end{proposition}

\begin{remark}
Assertion $(2)$ of Proposition \ref{psycher} implies the stronger property that the map
$$ \Fun^{R}_{\calD}( \calE, \calC') \rightarrow \Fun^{R}_{\calD}( \calE, \calC)$$
is a trivial Kan fibration, but we will not need this fact.
\end{remark}

\begin{proof}
We will give the proof in the case of stable envelopes; the case of pointed envelopes can be handled using similar arguments. Let $\RPress$ denote the $\infty$-category whose objects are presentable $\infty$-categories
and whose morphisms are functors which admit left adjoints (see \S \toposref{colpres}), and let
$\RPressStab$ be the full subcategory of $\RPress$ spanned by those presentable $\infty$-categories which are stable. It follows from Corollary \stableref{mapprop} that the inclusion
$\RPressStab \subseteq \RPress$ admits a right adjoint, given by the construction
$\calX \mapsto \Stab( \calX ).$
Let us denote this right adjoint by $G$.

The presentable fibration $p$ is classified by a functor
$\chi: \calD^{op} \rightarrow \RPress$. Let $\alpha$ denote the counit transformation
$G \circ \chi \rightarrow \chi.$
Then $\alpha$ is classified by a map $u: \calC' \rightarrow \calC$ of presentable fibrations
over $\calD$. Making a fibrant replacement if necessary, we may suppose that $u$ is a categorical fibration (see Proposition \monoidref{umpertein}). Assertion $(1)$ now follows immediately from the construction. 

To prove $(2)$, let us suppose that the presentable fibration $q$ is classified by a functor
$\chi': \calD^{op} \rightarrow \RPress$. Using Theorem \toposref{straightthm} and
Proposition \toposref{gumby444},
we deduce the existence of a commutative diagram
$$ \xymatrix{ \bHom_{ \Fun( \calD^{op}, \RPress)}( \chi', G \circ \chi) \ar[r] \ar[d] &
\bHom_{ \Fun( \calD^{op}, \RPress)}( \chi', \chi ) \ar[d] \\
\Fun^{R, \sim}_{\calD}( \calE, \calC') \ar[r] & \Fun^{R, \sim}_{\calD}(\calE, \calC) }$$
in the homotopy category of spaces, where the vertical arrows are homotopy equivalences.
Since the fibers of $q$ are stable, $\chi'$ factors through $\RPressStab \subseteq \RPress$, so the
upper horizontal arrow is a homotopy equivalence. It follows that the lower horizontal arrow is a homotopy equivalence as well. Since $u$ is a categorical fibration, the lower horizontal arrow is also a Kan fibration, and therefore a trivial Kan fibration.

We now prove assertion $(3)$. The existence of $\overline{v}$ (and its uniqueness up to homotopy)
follows immediately from $(2)$. To prove that $\overline{v}$ is an equivalence, we first invoke Corollary \toposref{usefir} to reduce to the case where $\calD$ consists of a single vertex. In this case, the result follows from Remark \ref{huggt}.
\end{proof}

\begin{remark}\label{spukk}
Let $p: \calC \rightarrow \calD$ be a presentable fibration. Let
$u: \calC' \rightarrow \calC$ be a stable envelope of $u$, and let
$v: \calC'' \rightarrow \calC$ be a pointed envelope of $u$. Since every stable $\infty$-category
is pointed, Proposition \ref{psycher} implies the existence of a commutative diagram
$$ \xymatrix{ \calC' \ar[dr]^{u} \ar[rr]^{w} & & \calC'' \ar[dl]^{v} \\
& \calC, & }$$
where $w$ admits a left adjoint relative to $\calD$. Moreover, $w$ is uniquely determined up to homotopy. In the case where $\calD$ consists of a single vertex, we can identify $w$
with the usual infinite loop functor $\Omega^{\infty}_{\ast}: \Stab(\calC) \rightarrow \calC_{\ast}$.
\end{remark}

\begin{definition}\label{uwe}
Let $\calC$ be a presentable $\infty$-category. A {\it tangent bundle to $\calC$}
is a functor $T_{\calC} \rightarrow \Fun( \Delta^1, \calC)$ which exhibits
$T_{\calC}$ as the stable envelope of the presentable fibration
$\Fun( \Delta^1, \calC) \rightarrow \Fun( \{1\}, \calC) \simeq \calC$.
\end{definition}

In the situation of Definition \ref{uwe}, we will often abuse terminology by referring to
$T_{\calC}$ as the {\it tangent bundle} to $\calC$. We note that $T_{\calC}$ is determined
up equivalence by $\calC$. Roughly speaking, we may think of an object
of $T_{\calC}$ as a pair $(A, M)$, where $A$ is an object of $\calC$ and
$M$ is an infinite loop object of $\calC_{/A}$. In the case where
$\calC$ is the $\infty$-category of $E_{\infty}$-rings, we can identify
$M$ with an $A$-module (Theorem \ref{subbe}).
In this case, the functor $T_{\calC} \rightarrow \Fun( \Delta^1, \calC)$ associates
to $(A,M)$ the projection morphism $A \oplus M \rightarrow A$. Our terminology is justified as follows:
we think of this morphism as a ``tangent vector'' in the $\infty$-category $\calC$, relating the object
$A$ to the ``infinitesimally near'' object $A \oplus M$.

We conclude this section with a few remarks about limits and colimits in the tangent bundle to a presentable $\infty$-category $\calC$.

\begin{proposition}\label{kinder}
Let $\calC$ be a presentable $\infty$-category. Then the tangent bundle
$T_{\calC}$ is also presentable.
\end{proposition}

\begin{proof}
We will give an explicit construction of a tangent bundle to $\calC$. 
Let $\FinSpace$ denote the $\infty$-category of finite pointed spaces (see Notation \stableref{finner}).
Let $\calE = \Fun( \FinSpace, \Fun( \Delta^1, \calC) ) \simeq \Fun( \FinSpace \times \Delta^1, \calC)$.
Since $\FinSpace$ is essentially small, Proposition \toposref{presexp} implies that
$\calE$ is presentable. We let $\calE_0$ denote the full subcategory of $\calE$
spanned by those functors $F: \FinSpace \rightarrow \Fun(\Delta^1,\calC)$ with the following
properties:
\begin{itemize}
\item[$(i)$] The value $F(\ast)$ is an equivalence in $\calC$.
\item[$(ii)$] For every pushout diagram
$$ \xymatrix{ X \ar[r] \ar[d] & Y \ar[d] \\
X' \ar[r] & Y' }$$
in $\FinSpace$, the induced diagram
$$ \xymatrix{ F(X) \ar[r] \ar[d] & F(Y) \ar[d] \\
F(X') \ar[r] & F(Y') }$$
is a pullback diagram in $\Fun( \Delta^1, \calC)$.
\end{itemize}
Using Lemmas \toposref{stur1}, \toposref{stur2}, and \toposref{stur3}, we conclude
that $\calE_0$ is a strongly reflective subcategory of $\calE$, and therefore presentable.
Form a pullback diagram
$$ \xymatrix{ T_{\calC} \ar[r] \ar[d]^{p} & \calE_0 \ar[d] \\
\calC \ar[r] & \Fun( \FinSpace \times \{1\}, \calC ). }$$
It follows from Example \ref{pinah} that we can identify $T_{\calC}$ with a tangent
bundle to $\calC$. Theorem \toposref{surbus} implies that $T_{\calC}$ is presentable.
\end{proof}

It follows from Proposition \ref{kinder} that if $\calC$ is a presentable $\infty$-category, then the tangent bundle $T_{\calC}$ admits small limits and colimits. The following result
describes these limits and colimits in more detail:

\begin{proposition}\label{tanlim}
Let $\calC$ be a presentable $\infty$-category, let $T_{\calC}$ be a tangent bundle to
$\calC$, and let $p$ denote the composition
$$T_{\calC} \rightarrow \Fun( \Delta^1, \calC) \rightarrow \Fun( \{1\}, \calC) \simeq \calC.$$ 
Then:
\begin{itemize}
\item[$(1)$] A small diagram $\overline{q}: K^{\triangleright} \rightarrow T_{\calC}$ is a
colimit diagram if and only if $\overline{q}$ is a $p$-colimit diagram and
$p \circ \overline{q}$ is a colimit diagram in $\calC$.

\item[$(2)$] A small diagram $\overline{q}: K^{\triangleleft} \rightarrow T_{\calC}$ is a
limit diagram if and only if $\overline{q}$ is a $p$-limit diagram and
$p \circ \overline{q}$ is a limit diagram in $\calC$.
\end{itemize}
\end{proposition}

\begin{proof}
We will prove $(1)$; assertion $(2)$ will follow from the same argument.
The ``if'' direction follows from Proposition \toposref{basrel}. The converse then follows
from the uniqueness of colimit diagrams and the following assertion:
\begin{itemize}
\item[$(\ast)$] Let $K$ be a small simplicial set, and let $q: K \rightarrow T_{\calC}$ be a diagram.
Then $q$ admits an extension $\overline{q}: K^{\triangleright} \rightarrow T_{\calC}$ such
that $\overline{q}$ is a $p$-colimit diagram, and $p \circ \overline{q}$ is a colimit diagram in $\calC$.
\end{itemize}
To prove $(\ast)$, we first invoke the assumption that $\calC$ is presentable to deduce
the existence of a colimit diagram $\overline{q}_0: K^{\triangleright} \rightarrow \calC$ extending
$p \circ q$. It then suffices to show that we can lift $\overline{q}_0$ to a $p$-colimit diagram
in $T_{\calC}$; this follows from that fact that $p$ is a presentable fibration.
\end{proof}

\subsection{Relative Adjunctions}\label{reladj}

Let $\calC$ be a presentable $\infty$-category. Our goal in this section is to produce a left adjoint to the composite functor
$$ T_{\calC} \rightarrow \Fun( \Delta^1, \calC) \rightarrow \Fun( \{0\}, \calC) \simeq \calC.$$
The existence of the desired left adjoint can be deduced easily from the adjoint functor theorem
(Corollary \toposref{adjointfunctor}). However, we will later need more detailed information about
$L$, which can be deduced from the following {\em relative} version of the adjoint functor theorem: 

\begin{proposition}\label{reladjprop}
Suppose given a commutative diagram
$$ \xymatrix{ \calC \ar[dr]^{q} & & \calD \ar[dl]^{p} \ar[ll]^{G} \\
& \calE & }$$
of $\infty$-categories with the following properties:
\begin{itemize}
\item[$(i)$] The maps $p$ and $q$ are Cartesian fibrations.
\item[$(ii)$] The functor $G$ carries $p$-Cartesian morphisms of $\calD$
to $q$-Cartesian morphisms of $\calC$.
\item[$(iii)$] For each object $E \in \calE$, the induced map
$G_{E}: \calD_{E} \rightarrow \calC_{E}$ admits a left adjoint.
\end{itemize}
Then there exists a functor $F: \calC \rightarrow \calD$ such that
$pF=q$ and a natural transformation $u: \id_{\calC} \rightarrow G \circ F$
with the following properties:
\begin{itemize}
\item[$(1)$] The image of $u$ under the map $q$ is the identity transformation
from $q$ to itself.
\item[$(2)$] For each object $E \in \calE$, the induced transformation
$u_{E}: \id_{\calC_E} \rightarrow G_{E} \circ F_{E}$ is the unit of an adjunction
between $F_E$ and $G_E$.
\item[$(3)$] The map $u$ is the unit of an adjunction between $G$ and $F$.
\end{itemize}
\end{proposition}

In the situation of Proposition \ref{reladjprop}, we will say that $F$ is a {\it left adjoint
to $G$ relative to $\calE$}.

\begin{proof}
We first construct a correspondence associated to the functor $G$.
Let $X = \calC \coprod_{ \calD \times \{0\} } ( \calD \times \Delta^1)$.
Using the small object argument, we can construct a factorization
$$ X \stackrel{i}{\rightarrow} \calM \stackrel{r}{\rightarrow} \calE \times \Delta^1$$
where $i$ is inner anodyne and $r$ is an inner fibration. Moreover, we may assume that the maps
$$ \calC \rightarrow \calM \times_{\Delta^1} \{0\}$$
$$ \calD \rightarrow \calM \times_{\Delta^1} \{1\}$$
are isomorphisms of simplicial sets. We will henceforth identify $\calC$ and
$\calD$ with full subcategories of $\calM$ via these isomorphisms.

We first claim the following:
\begin{itemize}
\item[$(a)$] Let $g: D \rightarrow D'$ be a $p$-Cartesian morphism in
$\calD$. Then $g$ is an $r$-Cartesian morphism in $\calM$.
\end{itemize}

Fix an object $M \in \calM$. According to Proposition \toposref{charCart}, we get a diagram of spaces
$$ \xymatrix{ \bHom_{\calM}( M, D) \ar[r]^{\circ g} \ar[d] & \bHom_{\calM}( M, D') \ar[d] \\
\bHom_{\calE}( \overline{M}, E) \ar[r] & \bHom_{\calE}( \overline{M}, E' )}$$
which commutes up to specified homotopy, and we need to show that this diagram is homotopy Cartesian. If $M \in \calD$, then this follows from our assumption that $g$ is $p$-Cartesian.
Otherwise, $M \in \calC$, and we have a homotopy commutative diagram
$$ \xymatrix{ \bHom_{\calC}( M, G(D)) \ar[r]^{\circ G(g)} \ar[d] & \bHom_{\calC}( M,  G(D') ) \ar[d] \\
\bHom_{\calM}(M, D) \ar[r] & \bHom_{\calM}( M, D') }$$
where the vertical arrows are homotopy equivalences. Splicing these diagrams, we
deduce the desired result from Proposition \toposref{charCart}, since assumption $(ii)$ guarantees
that $G(g)$ is a $q$-Cartesian morphism in $\calC$. This completes the proof of $(a)$.

Note that any $r$-Cartesian morphism in $\calM$ is also $r'$-Cartesian, where
$r'$ denotes the composite map
$$ \calM \stackrel{r}{\rightarrow} \calE \times \Delta^1 \rightarrow \calE.$$

We next claim:
\begin{itemize}
\item[$(b)$] Let $C \in \calC$ and $D_0 \in \calD$ be objects having the same image
$E \in \calE$, and let $f: C \rightarrow D_0$ be a morphism in $\calM$ which projects
to $\id_{E}$ in $\calE$. Suppose that $f$ is an $r_{E}$-coCartesian morphism
of $\calM \times_{\calE} \{E\}$. Then $f$ is an $r$-coCartesian morphism of $\calM$.
\end{itemize}

To prove $(b)$, we must show that for every object $D' \in \calD$, composition
with $f$ induces a homotopy equivalence $\bHom_{\calD}(D_0, D') \rightarrow
\bHom_{\calM}(C, D')$. Let $E'$ denote the image of $D'$ in $\calE$, so that we have a commutative diagram
$$ \xymatrix{ \bHom_{\calD}(D_0,D') \ar[r] \ar[d]^{\psi'} & \bHom_{\calM}(C, D') \ar[d]^{\psi} \\
\bHom_{\calE}(E,E') \ar[r]^{\sim} & \bHom_{\calE}(E, E'). }$$
To show that the upper horizontal map is a homotopy equivalence, it will suffice to show that
it induces a homotopy equivalence after passing to the homotopy fiber over
any point $\overline{g} \in \bHom_{\calE}(E,E')$. Since $p$ is a Cartesian fibration, we can
lift $\overline{g}$ to a $p$-Cartesian morphism $g: D \rightarrow D'$ in $\calD$. 
Using Proposition \toposref{compspaces}, we can identify the homotopy fiber of $\psi'$
over the point $\overline{g}$ with the mapping space $\bHom_{ \calD_{E} }( D_0, D)$.
Similarly, since assertion $(a)$ implies that $g$ is $r$-Cartesian (and therefore $r'$-Cartesian),
Proposition \toposref{compspaces} allows us to identify the homotopy fiber of $\psi$ with the mapping space $\bHom_{ \calM_{E} }(C, D)$. We are therefore reduced to showing that composition with
$f$ induces a homotopy equivalence
$\bHom_{ \calD_{E} }( D_0, D) \rightarrow \bHom_{ \calM_{E} }( C, D)$, which is simply
a reformulation of the condition that $f$ is $r_{E}$-coCartesian. This completes the proof of $(b)$.

By construction, there is a natural transformation $\alpha: G \rightarrow \id_{\calD}$
of functors from $\calD$ to $\calM$. We next claim:
\begin{itemize}
\item[$(c)$] For every object 
$D \in \calD$, the map $\alpha_{D}: G(D) \rightarrow D$ is an $r$-Cartesian morphism
in $\calM$. 
\end{itemize}
To prove $(c)$, we let $E$ denote the image of $D \in \calE$. Unwinding the definitions, we
must show that the canonical map
$$ \psi: \calM_{/\alpha_D} \rightarrow \calE_{/ \id_{E} } \times_{ \calE_{/E} } \calM_{/D} \times_{\calM} \calC.$$
is a trivial Kan fibration. Since $\psi$ is automatically a right fibration, it suffices to show that
$\psi$ is a categorical equivalence. Since the projection $\calE_{/ \id_{E}} \rightarrow \calE_{/E}$
is a trivial Kan fibration, the induced map
$\calE_{/ \id_{E} } \times_{ \calE_{/E} } \calM_{/D} \times_{\calM} \calC \rightarrow \calM_{/D} \times_{\calM} \calC$ is also a trivial Kan fibration. By the two-out-of-three property, we are reduced to proving that the map
$\calM_{/ \alpha_{D}} \rightarrow \calM_{/D} \times_{\calM} \calC$ is a categorical equivalence.
This follows from the fact that $\alpha_{D}$ is $\overline{r}$-Cartesian, where
$\overline{r}$ denotes the composition
$$ \calM \stackrel{r}{\rightarrow} \calE \times \Delta^1 \rightarrow \Delta^1.$$

We are now ready to proceed with the main step. We will construct a commutative diagram
$$ \xymatrix{ \calC \times \{0\} \ar@{^{(}->}[r] \ar@{^{(}->}[d] & \calM \ar[d]^{r} \\
\calC \times \Delta^1 \ar[r] \ar@{-->}[ur]^{\beta} & \calE \times \Delta^1. }$$
with the following property: for every object $C \in \calC$, the functor
$\beta$ carries $\{C \} \times \Delta^1$ to an $r$-coCartesian morphism of $\calM$.
To construct $\beta$, we work simplex-by-simplex on $\calC$. Let us first consider
the case of zero-dimensional simplices. Fix an object $C \in \calC$, having image
$E \in \calM$. Invoking assumption $(iii)$, we see that the correspondence
$\calM_{E}$ is an adjunction, so there exists an $r_{E}$-coCartesian morphism
$\beta_{C}: C \rightarrow D$ in $\calM_{E}$. Assertion $(b)$ above now implies that
$\beta_{C}$ is $r$-Cartesian as desired.

To handle simplices of larger dimension, we need to solve mapping problems of the form
$$ \xymatrix{ ( \Delta^n \times \{0\} ) \coprod_{ \bd \Delta^n \times \{0\} } ( \bd \Delta^n \times \Delta^1) \ar[r]^-{j} \ar@{^{(}->}[d] & \calM \ar[d]^{r} \\
\Delta^n \times \Delta^1 \ar@{-->}[ur] \ar[r] & \calE \times \Delta^1, }$$
where $n > 0$ and the map $j$ carries $\{0\} \times \Delta^1$ to an $r$-coCartesian morphism
in $\calM$. The existence of the required extension follows from Proposition \toposref{goouse}.

We now define $F: \calC \rightarrow \calD$ to be the restriction of $\beta$ to
$\calC \times \{1\}$. The maps $\alpha$ and $\beta$ together define a diagram
$$ \xymatrix{ & GF \ar[dr]^{\alpha} & \\
\id_{\calC} \ar@{-->}[ur]^{u} \ar[rr]^{\beta} & & F }$$
in the $\infty$-category $\Fun_{\calE}( \calC, \calM)$. Using $(c)$, we can construct the dotted arrow
$u$ indicated in the diagram. It is easy to see that $u$ has the required properties.
\end{proof}

\begin{definition}\label{urtime}
Let $\calC$ be a presentable $\infty$-category, and consider the associated diagram
$$ \xymatrix{ T_{\calC} \ar[rr]^{G} \ar[dr]^{p} & & \Fun( \Delta^1, \calC) \ar[dl]^{q} \\
& \calC & }$$
where $q$ is given by evaluation at $\{1\} \subseteq \Delta^1$. The functor
$G$ carries $p$-Cartesian morphisms to $q$-Cartesian morphisms, and for each
object $A \in \calC$ the induced map $G_{A}: \Stab( \calC^{/A} ) \rightarrow \calC^{/A}$
admits a left adjoint $\Sigma^{\infty}$. Applying Proposition \ref{reladjprop}, we conclude
that $G$ admits a left adjoint relative to $\calC$, which we will denote by $F$.
The {\it absolute cotangent complex functor} $L: \calC \rightarrow T_{\calC}$ is defined
to be the composition
$$ \calC \rightarrow \Fun( \Delta^1, \calC) \stackrel{F}{\rightarrow} T_{\calC},$$
where the first map is given by the diagonal embedding.
We will denote the value of $L$ on an object $A \in \calC$
by $L_{A} \in \Stab( \calC^{/A})$, and will refer to $L_A$ as the {\it cotangent complex of $A$}.
\end{definition}

\begin{remark} Let $\calC$ be a presentable $\infty$-category.
Since the diagonal embedding $\calC \rightarrow \Fun( \Delta^1, \calC)$ is a left adjoint
to the evaluation map $\Fun( \Delta^1, \calC) \rightarrow \Fun( \{0\}, \calC) \simeq \calC$,
we deduce that the absolute cotangent complex functor $L: \calC \rightarrow T_{\calC}$
is left adjoint to the composition
$$ T_{\calC} \rightarrow \Fun( \Delta^1, \calC) \rightarrow \Fun( \{0\}, \calC) \simeq \calC.$$
\end{remark}

\begin{remark}
The terminology of Definition \ref{urtime} is slightly abusive, since the tangent bundle
$T_{\calC}$ and the functor $L$ are only well-defined up to equivalence.
It would perhaps be more accurate to refer to $L: \calC \rightarrow T_{\calC}$ as {\em an} absolute cotangent functor. However, $L$ and $T_{\calC}$ are well-defined up to a contractible space of choices, so we will tolerate the ambiguity.
\end{remark}

\begin{remark}
Let $\calC$ be a presentable $\infty$-category containing an object $A$. We observe that
the fiber of the tangent bundle $T_{\calC}$ over $A \in \calC$ can be identified with the stable envelope
$\Stab( \calC_{/A} )$. Under this identification, the object $L_A \in \Stab( \calC_{/A} )$ corresponds
to the image of $\id_{A} \in \calC_{/A}$ under the suspension spectrum functor
$$\Sigma^{\infty}: \calC_{/A} \rightarrow \Stab( \calC_{/A} ).$$
\end{remark}

\begin{remark}\label{stuck}
Let $\calC$ be a presentable $\infty$-category.
Since the cotangent complex functor $L$ is a left adjoint, it carries colimit diagrams
in $\calC$ to colimit diagrams in $T_{\calC}$. In view of Proposition \ref{tanlim}, 
we see that $L$ also carries small colimit diagrams in $\calC$ to $p$-colimit diagrams
in $T_{\calC}$, where $p$ denotes the composition
$$ T_{\calC} \rightarrow \Fun( \Delta^1, \calC) \rightarrow \Fun( \{1\}, \calC) \simeq \calC.$$
\end{remark}

\begin{remark}\label{urn}
Let $\calC$ be a presentable $\infty$-category, and let $A$ be an initial object of $\calC$.
Using Remark \ref{stuck}, we deduce that $L_A$ is an initial object of the tangent bundle $T_{\calC}$. 
Equivalently, $L_A$ is a zero object of the stable $\infty$-category $\Stab( \calC^{/A})$.
\end{remark}

\subsection{The Tangent Correspondence}\label{gen1}

Let $\calC$ be an $\infty$-category, $T_{\calC}$ a tangent bundle to $\calC$, and
$L: \calC \rightarrow T_{\calC}$ the associated cotangent complex functor.
Then there exists a coCartesian fibration $p: \calM \rightarrow \Delta^1$
with $\calM \times_{ \Delta^1 } \{0\} \simeq \calC$, 
$\calM \times_{ \Delta^1 } \{1\} \simeq T_{\calC}$, such that the associated functor
$\calC \rightarrow T_{\calC}$ can be identified with $L$ (see \S \toposref{afunc1}). 
The $\infty$-category $\calM$ is called a {\it tangent correspondence} to $\calC$.
Our goal in this section is to give an explicit construction of a tangent correspondence to $\calC$, which we will refer to as {\em the} tangent correspondence to $\calC$ and denote by $\calM^{T}(\calC)$.

\begin{remark}
Since the cotangent complex functor $L$ admits a right adjoint, the
coCartesian fibration $p: \calM \rightarrow \Delta^1$ considered above is also a Cartesian fibration, associated to the composite functor
$$ T_{\calC} \rightarrow \Fun( \Delta^1, \calC) \rightarrow \Fun( \{0\}, \calC) \simeq \calC.$$
\end{remark}

Recall that a {\em correspondence} between a pair of $\infty$-categories $\calC$ and $\calD$
is an $\infty$-category $\calM$ equipped with a functor $p: \calM \rightarrow \Delta^1$ and
isomorphisms $\calC \simeq \calM \times_{ \Delta^1 } \{0\}$ and
$\calD \simeq \calM \times_{ \Delta^1 } \{1\}$. If $p$ is a Cartesian fibration, then a correspondence
determines a functor $\calD \rightarrow \calC$, which is well-defined up to homotopy. 
It is therefore reasonable to think of a correspondence as a ``generalized functor''. Our first result describes how to compose these ``generalized functors'' with ordinary functors.

\begin{lemma}\label{susin}
Suppose given sequence of maps $A \stackrel{f}{\rightarrow} B \rightarrow \Delta^1$
in the category of simplicial sets. Let $A_1$ denote the fiber product
$A \times_{ \Delta^1} \{1\}$, and define $B_1$ similarly. If $f$ is a categorical equivalence, then the induced map $A_1 \rightarrow B_1$ is a categorical equivalence.
\end{lemma}

\begin{proof}
This follows immediately from the definition, since $\sCoNerve(A_1)$ and $\sCoNerve(B_1)$
can be identified with the full simplicial subcategories of $\sCoNerve(A)$ and $\sCoNerve(B)$
lying over the object $\{1\} \in \sCoNerve( \Delta^1 )$.
\end{proof}

\begin{proposition}\label{suskin}
Let $\calC$ and $\calD$ be $\infty$-categories, and let $p: \calM \rightarrow \Delta^1$ be
a correspondence from $\calC$ to $\calD$. Let $G: \calD' \rightarrow \calD$ be a categorical fibration
of simplicial sets. We define a new simplicial set $\calM'$ equipped with a map $p': \calM' \rightarrow \calM$, so that the following universal property is satisfied: for every map of simplicial sets
$A \rightarrow \Delta^1$, we have a pullback diagram of sets
$$ \xymatrix{ \Hom_{ \Delta^1}(A, \calM') \ar[d] \ar[r] & \Hom( A \times_{ \Delta^1} \{1\}, \calD') \ar[d] \\
\Hom_{\Delta^1}(A, \calM) \ar[r] & \Hom( A \times_{ \Delta^1} \{1\}, \calD ). }$$
Then:
\begin{itemize}
\item[$(1)$] The map $\calM' \rightarrow \calM$ is an inner fibration of simplicial sets.

\item[$(2)$] The simplicial set $\calM'$ is an $\infty$-category.

\item[$(3)$] Let $f:C \rightarrow D'$ be a morphism in $\calM'$ from an object of
$\calC$ to an object of $\calD'$. Then $f$ is a $(p \circ p')$-Cartesian morphism of $\calM'$ if and only
if $p'(f)$ is a $p$-Cartesian morphism of $\calM$.

\item[$(4)$] Assume that the map $\calM \rightarrow \Delta^1$ is a Cartesian fibration, associated to a functor $G': \calD \rightarrow \calC$. Then the composite map $\calM' \rightarrow \calM \rightarrow \Delta^1$ is a Cartesian fibration, associated to the functor $G' \circ G$.
\end{itemize}
\end{proposition}

\begin{proof}
We first prove $(1)$. We wish to show that the projection $\calM' \rightarrow \calM$ has the right lifting property with respect to every inclusion $A \rightarrow B$ which is a categorical equivalence of simplicial sets. Fix a map $\alpha: B \rightarrow \Delta^1$; we must show that it is possible to solve any mapping problem of the form
$$ \xymatrix{ A \times_{ \Delta^1} \{1\} \ar[r] \ar@{^{(}->}[d]^{i} & \calD' \ar[d]^{G} \\
B \times_{ \Delta^1} \{1\} \ar[r] & \calD. }$$
Since $G$ is assumed to be a categorical fibration, it will suffice to show that $i$ is a categorical equivalence, which follows from Lemma \ref{susin}.
This completes the proof of $(1)$. Assertion $(2)$ follows immediately.

We now prove $(3)$. Let $\overline{f}$ denote the image of $f$ in $\calM$. 
We have a commutative diagram of simplicial sets
$$ \xymatrix{ & \calM_{/\overline{f} } \ar[dr]^{\psi} & \\
\calM'_{/f} \ar[ur]^{\phi} \ar[rr] & & \calC_{/C}. }$$
We observe that $f$ is $(p \circ p')$-Cartesian if and only if $(\psi \circ \phi)$ is a trivial Kan fibration,
and that $\overline{f}$ is $p$-Cartesian if and only if $\psi$ is a trivial Kan fibration. The desired equivalence now follows from the observation that $\phi$ is an isomorphism.

To prove $(4)$, let us suppose that we are given a map $h: \calD \times \Delta^1 \rightarrow \calM$
which is a $p$-Cartesian natural transformation from $G'$ to $\id_{\calD}$. Using the definition of
$\calM'$, we see that the composition
$$ \calD' \times \Delta^1 \rightarrow \calD \times \Delta^1 \stackrel{h}{\rightarrow} \calM$$
can be lifted uniquely to a map $h': \calD' \times \Delta^1 \rightarrow \calM'$ which
is a natural transformation from $G' \circ G$ to $\id_{\calD'}$. It follows from $(3)$ that
$h'$ is a $(p \circ p')$-Cartesian transformation, so that $(p \circ p')$ is a Cartesian fibration associated to the functor $G' \circ G$.
\end{proof}

We now describe an important example of a correspondence.

\begin{notation}\label{iffle}
Let $K \subseteq \Delta^1 \times \Delta^1$ denote the full subcategory spanned by
the vertices $\{i\} \times \{j\}$ where $i \leq j$ (so that $K$ is isomorphic to a $2$-simplex
$\Delta^2$). For every simplicial set $A$ equipped with a map $f: A \rightarrow \Delta^1$, we
let $\overline{A}$ denote the inverse image of $K$ under the induced map
$$ \Delta^1 \times A \rightarrow \Delta^1 \times \Delta^1.$$
Note that the map $A \stackrel{(f, \id)}{\rightarrow} \Delta^1 \times A$ factors through
$\overline{A}$; we will denote the resulting inclusion by $\psi_{A}: A \rightarrow \overline{A}$.

Let $\calC$ be an $\infty$-category. The {\it fundamental correspondence of $\calC$} is
a simplicial set $\calM^{0}(\calC)$ equipped with a map $p: \calM^{0}(\calC) \rightarrow \Delta^1$, characterized by the following universal property: for every map of simplicial sets $A \rightarrow \Delta^1$, we have a canonical bijection of sets
$$ \Hom_{\Delta^1}( A, \calM^{0}(\calC) ) \simeq \Hom( \overline{A}, \calC). $$
The inclusions $\psi_{A}: A \rightarrow \overline{A}$ determine a map
$q: \calM^{0}(\calC) \rightarrow \calC$. Together $p$ and $q$ determine a map
$\calM^{0}(\calC) \rightarrow \calC \times \Delta^1$, which we will call the
{\it fundamental projection}.
\end{notation}

\begin{remark}
Let $\calC$ be an $\infty$-category, and let $\calM^{0}(\calC)$ be its fundamental correspondence. Then the fiber $\calM^{0}(\calC) \times_{ \Delta^1} \{0\}$ is canonically isomorphic to $\calC$, and
the fiber $\calM^{0}(\calC) \times_{ \Delta^1} \{1\}$ is canonically isomorphic to $\Fun(\Delta^1, \calC)$.
We will generally abuse terminology, and use these isomorphisms identify $\calC$ and $\Fun( \Delta^1, \calC)$ with subsets of $\calM^{0}(\calC)$. The map $q: \calM^{0}(\calC) \rightarrow \calC$ is given by the identity on $\calC$, and by evaluation at
$\{1\} \subseteq \Delta^1$ on $\Fun( \Delta^1, \calC)$.
\end{remark}

\begin{proposition}\label{sss}
Let $\calC$ be an $\infty$-category, let $\calM^{0}(\calC)$ be the fundamental correspondence of $\calC$, and let $\pi: \calM^{0}(\calC) \rightarrow \calC \times \Delta^1$ denote the fundamental projection,
and $p: \calM^{0}(\calC) \rightarrow \Delta^1$ the composition of $\pi$ with projection onto the second factor. Then:
\begin{itemize}
\item[$(1)$] The fundamental projection $\pi$ is a categorical fibration. In particular,
$\calM^{0}(\calC)$ is an $\infty$-category.

\item[$(2)$] The map $p$ is a Cartesian fibration.

\item[$(3)$] Let $A \in \calC \subseteq \calM^{0}(\calC)$, and let
$(f: B \rightarrow C) \in \Fun(\Delta^1, \calC) \subseteq \calM^{0}(\calC)$. Let
$\alpha: A \rightarrow f$ be a morphism in $\calM^{0}(\calC)$, corresponding to a commutative
diagram
$$ \xymatrix{ A \ar[r]^{\overline{\alpha}} \ar[dr] & B \ar[d]^{f} \\
& C }$$
in $\calC$. Then $\alpha$ is $p$-Cartesian if and only if $\overline{\alpha}$ is an equivalence in $\calC$.

\item[$(4)$] The Cartesian fibration $p$ is associated to the functor
$\Fun( \Delta^1, \calC) \rightarrow \calC$ given by evaluation at the vertex
$\{0\} \in \Delta^1$.

\item[$(5)$] The map $p$ is also a coCartesian fibration, associated to the diagonal
inclusion $\calC \rightarrow \Fun( \Delta^1, \calC)$.

\end{itemize}
\end{proposition}

The proof will require a few lemmas. In what follows, we will employ the conventions of
Notation \ref{iffle}. 

\begin{lemma}\label{toofle}
Let $A$ be a simplicial set equipped with a map $A \rightarrow \Delta^1$, and let
$$\widetilde{A} = (A \times \{0\}) \coprod_{ A_1 \times \{0\} } (A_1 \times \Delta^1) \subseteq \overline{A}.$$ 
Then the inclusion $\widetilde{A} \subseteq \overline{A}$ is a categorical equivalence.
\end{lemma}

\begin{proof}
The functors $A \mapsto \widetilde{A}$ and $A \mapsto \overline{A}$ both commute with colimits. Since the class of categorical equivalences is stable under filtered colimits, we may reduce to the case where $A$ has only finitely many simplices. We now work by induction on the dimension $n$ of $A$, and the number of nondegenerate simplices of dimension $n$. If $A$ is empty there is nothing to prove; otherwise there exists a pushout diagram
$$ \xymatrix{ \bd \Delta^n \ar[r] \ar[d] & \Delta^n \ar[d] \\
A' \ar[r] & A. }$$
This induces homotopy pushout diagrams
$$ \xymatrix{ \overline{ \bd \Delta^n } \ar[r] \ar[d] & \overline{ \Delta^n} \ar[d] & \widetilde{ \bd \Delta^n} \ar[r] \ar[d] & \widetilde{ \Delta^n} \ar[d] \\
\overline{A}' \ar[r] & \overline{A} & \widetilde{A}' \ar[r] & \widetilde{A}. }$$
It will therefore suffice to prove the lemma after replacing $A$ by $A'$, $\bd \Delta^n$, or $\Delta^n$.
In the first two cases this follows from the inductive hypothesis. We may therefore assume that $A = \Delta^n$. In particular, $A$ is an $\infty$-category. The composite map
$$ \overline{A} \subseteq A \times \Delta^1 \rightarrow \Delta^1$$
is a Cartesian fibration associated to the inclusion $i: A_1 \rightarrow A$, and $\widetilde{A}$
can be identified with the mapping cylinder of $i$. The desired result now follows from
Proposition \toposref{qequiv}.
\end{proof}

\begin{lemma}\label{pyle}
Suppose given maps of simplicial sets $A \stackrel{f}{\rightarrow} B \rightarrow \Delta^1$. If
$f$ is a categorical equivalence, then the induced map $\overline{A} \rightarrow \overline{B}$
is a categorical equivalence.
\end{lemma}

\begin{proof}
Let $\widetilde{A}$ and $\widetilde{B}$ be defined as in Lemma \ref{toofle}. We have a commutative diagram
$$ \xymatrix{ \widetilde{A} \ar@{^{(}->}[d] \ar[r]^{\widetilde{f} } & \widetilde{B} \ar@{^{(}->}[d] \\
\overline{A} \ar[r]^{ \overline{f} } & \overline{B}, }$$
where the vertical maps are categorical equivalences by Lemma \ref{toofle}. It will therefore
suffice to show that $\widetilde{f}$ is a categorical equivalence. The map $\widetilde{f}$
determines a map of homotopy pushout diagrams
$$ \xymatrix{ A_1 \times \{0\} \ar[r] \ar[d] & A \times \{0\} \ar[d] & B_1 \times \{0\} \ar[r] \ar[d] & B \times \{0\} \ar[d] \\
A_1 \times \Delta^1 \ar[r] & \widetilde{A} & B_1 \times \Delta^1 \ar[r] & \widetilde{B}. }$$
It therefore suffices to show that the map $A_1 \rightarrow B_1$ is a categorical equivalence, which follows from Lemma \ref{susin}.
\end{proof}

\begin{proof}[Proof of Proposition \ref{sss}]
We first prove $(1)$. Consider a lifting problem
$$ \xymatrix{ A \ar@{^{(}->}[d]^{i} \ar[r] & \calM^{0}(\calC) \ar[d]^{\pi} \\
B \ar@{-->}[ur] \ar[r] & \calC \times \Delta^1, }$$
where $i$ is a monomorphism of simplicial sets. We must show that this lifting problem has a solution if $i$ is a categorical equivalence. Unwinding the definitions (and using the conventions of Notation \ref{iffle}, we are reduced to showing that $\calC$ has the extension property with respect to the inclusion
$j: \overline{A} \coprod_{A} B \rightarrow \overline{B}$. For this, it suffices to show that
$j$ is a categorical equivalence. Since the Joyal model structure is left proper, it will suffice to show that the inclusion $\overline{A} \rightarrow \overline{B}$ is a categorical equivalence, which follows
from Lemma \ref{pyle}.

We next prove $(3)$. Let us identify $\alpha$ with a $2$-simplex in $\calC$. Unwinding the definitions, we see that $\alpha$ is $p$-Cartesian if and only if the map
$\phi: \calC_{/ \alpha} \rightarrow \calC_{/f}$ is a trivial Kan fibration. In view of Proposition
\toposref{greenlem}, this is equivalent to the requirement that the map $A \rightarrow B$
be an equivalence in $\calC_{/C}$, which is equivalent to the requirement that
$\overline{\alpha}$ be an equivalence in $\calC$ (Proposition \toposref{needed17}).

We now prove $(2)$. Since $p$ is the composition of $\pi$ with the projection map $\calC \times \Delta^1 \rightarrow \Delta^1$, we deduce immediately that $p$ is an inner fibration. To show that $p$ is a Cartesian fibration, it will suffice to show that for every object $X \in \calM^{0}(\calC)$ and every morphism
$\overline{\alpha}: y \rightarrow p(x)$ in $\Delta^1$, there exists a $p$-Cartesian morphism
$\alpha: Y \rightarrow X$ lifting $\overline{\alpha}$. If $\overline{\alpha}$ is degenerate, we can choose $\alpha$ to be degenerate. We may therefore assume that $X \in \Fun( \Delta^1, \calC)$ classifies
a map $B \rightarrow C$ in $\calC$. We can then choose $\alpha$ to classify the diagram
$$ \xymatrix{ B \ar[r]^{\id}  \ar[dr] & B \ar[d] \\
& C. }$$
It follows from $(3)$ that $\alpha$ is $p$-Cartesian.

Let $G: \Fun( \Delta^1, \calC) \rightarrow \calC$ denote the functor given by evaluation at the vertex $\{0\}$. To prove $(4)$, we must exhibit a $p$-Cartesian natural transformation
$h: \Delta^1 \times \Fun( \Delta^1, \calC) \rightarrow \calM^{0}(\calC)$ from
$G$ to $\id_{\Fun( \Delta^1, \calC) }$. We now choose $h$ to classify the composite map
$$ K \times \Fun( \Delta^1, \calC) \stackrel{(h_0, \id)}{\rightarrow} \Delta^1 \times \Fun( \Delta^1, \calC) \rightarrow \calC$$
where $K$ is defined as in Notation \ref{iffle}, and $h_0: K \simeq \Delta^2 \rightarrow \Delta^1$
is the map which collapses the edge $\Delta^{ \{0,1\} } \subseteq \Delta^2$. It follows from
$(3)$ that $h$ is a Cartesian transformation with the desired properties.

We now prove $(5)$.
Let $F: \calC \rightarrow \Fun( \Delta^1, \calC)$ denote the diagonal embedding. 
The $G \circ F = \id_{\calC}$. The identity map $\id_{\calC} \rightarrow G \circ F$ is the unit for
an adjunction between $G$ and $F$. Thus $p$ is also a coCartesian fibration, associated to the functor $F$, as desired.
\end{proof}

\begin{definition}\label{extcor}
Let $\calC$ be a presentable $\infty$-category and let $G: T_{\calC} \rightarrow \Fun( \Delta^1, \calC)$ be a tangent bundle to $\calC$. We define the {\it tangent correspondence}
$\calM^{T}(\calC)$ to be the result of applying the construction of Proposition \ref{suskin} using
the fundamental correspondence $\calM^{0}(\calC)$ and the functor 
$G$. 
By construction, $\calM^{T}(\calC)$ is equipped with a projection map $\pi: \calM^{T}(\calC) \rightarrow \Delta^1 \times \calC$.
\end{definition}

\begin{remark}
The terminology of Definition \ref{extcor} is slightly abusive: the tangent correspondence
$\calM^{T}(\calC)$ depends on a choice of tangent bundle $T_{\calC} \rightarrow \Fun( \Delta^1, \calC)$.
However, it is easy to eliminate this ambiguity: for example, we can use an explicit construction
of $T_{\calC}$ such as the one which appears in the proof of Proposition \ref{kinder}.
\end{remark}

The following result is an immediate consequence of Propositions \ref{sss}, Proposition \ref{suskin}, and the definition of the cotangent complex functor $L$:

\begin{proposition}
Let $\calC$ be a presentable $\infty$-category. Then:
\begin{itemize}
\item[$(1)$] The projection $\calM^{T}(\calC) \rightarrow \Delta^1 \times \calC$ is a categorical fibration.
\item[$(2)$] The composite map $p: \calM^{T}(\calC) \rightarrow \Delta^1 \times \calC \rightarrow \Delta^1$ is a Cartesian fibration, associated to the functor
$$ T_{\calC} \rightarrow \Fun( \Delta^1, \calC) \rightarrow \Fun( \{0\}, \calC) \simeq \calC.$$

\item[$(3)$] The map $p$ is also a coCartesian fibration, associated to the cotangent complex
functor $L: \calC \rightarrow T_{\calC}$.
\end{itemize}
\end{proposition}

\subsection{The Relative Cotangent Complex}\label{relcot}

In applications, it will be convenient to have also a {\em relative} cotangent complex defined for a morphism $f: A \rightarrow B$ in a presentable $\infty$-category $\calC$. In this section, we will define the relative cotangent complex $L_{B/A}$ and establish some of its basic properties.

\begin{definition}\label{uppa}
Let $\calC$ be a presentable $\infty$-category and let $p: T_{\calC} \rightarrow \calC$ be a tangent bundle to $\calC$. A {\it relative cofiber sequence} in $T_{\calC}$ is a diagram $\sigma$:
$$ \xymatrix{ X \ar[r] \ar[d] & Y \ar[d] \\
0 \ar[r] & Z}$$
in $T_{\calC}$ with the following properties:
\begin{itemize}
\item[$(1)$] The map $p \circ \sigma$ factors through the projection
$\Delta^1 \times \Delta^1 \rightarrow \Delta^1$, so that the vertical arrows above
become degenerate in $\calC$.
\item[$(2)$] The diagram $\sigma$ is a pushout square. (Since condition $(1)$ implies that
$p \circ \sigma$ is a pushout square, this is equivalent to the requirement that $\sigma$
be a $p$-colimit diagram; see Proposition \ref{tanlim}).
\end{itemize}

Let $\calE$ denote the full subcategory of $$\Fun( \Delta^1 \times \Delta^1, T_{\calC} )
\times_{ \Fun( \Delta^1 \times \Delta^1, \calC) } \Fun( \Delta^1, \calC)$$ spanned by
the relative cofiber sequences. There is an evident forgetful functor
$\psi: \calE \rightarrow \Fun( \Delta^1, T_{\calC} )$, given by restriction to the upper half of the diagram.
Invoking Proposition \toposref{lklk} twice, we deduce that $\psi$ is a trivial Kan fibration.

The {\it relative cotangent complex functor} is defined to be the composition
$$ \Fun( \Delta^1, \calC) \stackrel{L}{\rightarrow}
\Fun( \Delta^1, T_{\calC} ) \stackrel{s}{\rightarrow} \calE \stackrel{s'}{\rightarrow} T_{\calC},$$
where $s$ is a section of $\psi$ and $s'$ is given by evaluation at the vertex
$\{1\} \times \{1\} \subseteq \Delta^1 \times \Delta^1$.

We will denote the image a morphism $f: A \rightarrow B$ under the relative cotangent complex
functor by $L_{B/A} \in T_{\calC} \times_{ \calC} \{ B \} \simeq \Stab( \calC^{B/} )$.
\end{definition}

\begin{remark}\label{irn}
Let $\calC$ and $p: T_{\calC} \rightarrow \calC$ be as in Notation \ref{uppa}.
By definition, the relative cotangent complex of a morphism $f: A \rightarrow B$
fits into a relative cofiber sequence
$$ \xymatrix{ L_{A} \ar[r] \ar[d] & L_{B} \ar[d] \\
0 \ar[r] & L_{B/A} }$$
in the $\infty$-category $T_{\calC}$.
Using Proposition \toposref{chocolatelast}, we deduce the existence of a distinguished triangle
$$ f_{!} L_A \rightarrow L_B \rightarrow L_{B/A} \rightarrow f_{!} L_A[1] $$
in the stable $\infty$-category $\Stab(\calC^{/B} ) \simeq T_{\calC} \times_{\calC} \{B\}$;
here $f_{!}: \Stab( \calC^{/A} ) \rightarrow \Stab( \calC^{/B} )$ denotes the functor induced by the coCartesian fibration $p$.
\end{remark}

\begin{remark}\label{irk}
Let $\calC$ be a presentable $\infty$-category containing a morphism $f: A \rightarrow B$.
If $A$ is an initial object of $\calC$, then the canonical map
$L_B \rightarrow L_{B/A}$ is an equivalence. This follows immediately from Remark \ref{irn}, since
the absolute cotangent complex $L_{A}$ vanishes (Remark \ref{urn}). We will sometimes invoke this equivalence implicitly, and ignore the distinction between the relative cotangent complex $L_{B/A}$ and the absolute cotangent complex $L_{B}$.
\end{remark}

\begin{remark}\label{blurn}
Let $\calC$ be a presentable $\infty$-category containing a morphism $f: A \rightarrow B$.
If $f$ is an equivalence, then the relative cotangent complex $L_{B/A}$ is a zero object
of $\Stab( \calC^{/B})$. This follows immediately from Remark \ref{irn}. 
\end{remark}

We next study the distinguished triangle of cotangent complexes associated to a triple of morphisms $A \rightarrow B \rightarrow C$.

\begin{proposition}
Let $\calC$ be a presentable $\infty$-category, let $T_{\calC}$ be a tangent bundle to $\calC$. Suppose given a commutative diagram
$$ \xymatrix{ & B \ar[dr] & \\
A \ar[ur] \ar[rr] & & C }$$
in $\calC$. The resulting square
$$ \xymatrix{ L_{B/A} \ar[r]^{f} \ar[d] & L_{C/A} \ar[d] \\
L_{B/B} \ar[r] & L_{C/B} }$$
is a pushout diagram in $T_{\calC}$ $($and therefore a relative cofiber sequence, in view of Remark \ref{blurn}$)$.
\end{proposition}

\begin{proof}
We have a commutative diagram
$$ \xymatrix{ L_A \ar[r] \ar[d] & L_{B} \ar[r] \ar[d] & L_{C} \ar[d] \\
L_{A/A} \ar[r] & L_{B/A} \ar[r] \ar[d] & L_{C/A} \ar[d] \\
& L_{B/B} \ar[r] & L_{C/B} }$$ 
in the $\infty$-category $T_{\calC}$. Here $L_{A/A}$ and $L_{B/B}$ are zero
objects in the fibers $\Stab( \calC^{/A} )$ and $\Stab( \calC^{/B})$, respectively (Remark \ref{blurn}).
By construction, the upper left square and both large rectangles in this diagram
are coCartesian. It follows first that the upper right square is coCartesian, and then that the lower right square is coCartesian as desired.
\end{proof}

\begin{corollary}\label{longtooth}
Let $\calC$ be a presentable $\infty$-category containing a commutative triangle
$$ \xymatrix{ & B \ar[dr]^{f} & \\
A \ar[ur] \ar[rr] & & C, }$$
and let $f_{!}: \Stab( \calC^{/B} ) \rightarrow \Stab( \calC^{/C} )$ denote the induced map.
Then we have a canonical distinguished triangle
$$ f_! L_{B/A} \rightarrow L_{C/A} \rightarrow L_{C/B} \rightarrow f_{!} L_{B/A}[1]$$
in the homotopy category $\h{ \Stab( \calC^{/C} ) }$. 
\end{corollary}

Our next result records the behavior of the relative cotangent complex under base change.


\begin{proposition}\label{basechunge}
Let $\calC$ be a presentable $\infty$-category, $T_{\calC}$ a tangent bundle to $\calC$, and
$p$ the composite map
$$ T_{\calC} \rightarrow \Fun( \Delta^1, \calC) \rightarrow \Fun( \{1\}, \calC) \simeq \calC.$$
Suppose given a pushout diagram
$$ \xymatrix{ A \ar[r] \ar[d] & B \ar[d]^{f} \\
A' \ar[r] & B' }$$
in $\calC$. Then the induced map
$\beta: L_{B/A} \rightarrow L_{B'/A'}$ is a $p$-coCartesian morphism in $T_{\calC}$.
\end{proposition}

\begin{proof}
Using Definition \ref{uppa}, we deduce the existence of a map between relative
cofiber sequences in $T_{\calC}$, which we can depict as a cubical diagram $\tau$:
$$ \xymatrix{ L_A \ar[rr] \ar[dd] \ar[dr] & & L_B \ar[dd] \ar[dr] & \\
& 0_{A} \ar[dd] \ar[rr] & & L_{B/A} \ar[dd] \\
L_{A'} \ar[rr] \ar[dr] & & L_{B'} \ar[dr] & \\
& 0_{A'} \ar[rr] & & L_{B'/A'}. }$$
Let $K \subseteq \Delta^1 \times \Delta^1 \times \Delta^1$ denote the full simplicial subset
obtained by omitting the final vertex. Let $K_0 \subseteq K$ be obtained by
omitting the vertex $v = \{1\} \times \{1\} \times \{0\}$ such that $\tau(v) = L_{B'}$, and let
$K_1 \subseteq K$ be obtained by omitting the vertex $w = \{1\} \times \{0\} \times \{1\}$ such
that $\tau(w) = L_{B/A}$. By construction, $\tau$ is a $p$-left Kan extension of $\tau|K_1$.
Using Proposition \toposref{acekan}, we conclude that $\tau$ is a $p$-colimit diagram.

Remark \ref{stuck} implies that the square
$$ \xymatrix{ L_A \ar[r] \ar[d] & L_{B} \ar[d] \\
L_{A'} \ar[r] & L_{B'} }$$
is a $p$-colimit diagram, so that $\tau| K$ is a $p$-left Kan extension of $\tau|K_0$.
Invoking Proposition \toposref{acekan} again, we deduce that $\tau$ is a $p$-left Kan
extension of $\tau| K_0$. It follows that $\tau$ restricts to a $p$-colimit square:
$$ \xymatrix{ 0_{A} \ar[r] \ar[d] & L_{B/A} \ar[d] \\
0_{A'} \ar[r] & L_{B'/A'}. }$$
Proposition \toposref{chocolatelast} implies that the induced square
$$ \xymatrix{ 0 \ar[r] \ar[d] & f_! L_{B/A} \ar[d]^{\alpha} \\
0 \ar[r] & L_{B'/A'} }$$
is a pushout square in $\Stab( \calC^{/B'} )$; in other words, the map $\alpha$ is an equivalence.
This is simply a reformulation of the assertion that $\beta$ is $p$-coCartesian.
\end{proof}

There is another way to view the relative cotangent complex: if we fix an object
$A \in \calC$, then the functor $B \mapsto L_{B/A}$ can be identified with the
{\em absolute} cotangent complex for the $\infty$-category $\calC_{A/}$. The rest of this section
will be devoted to justifying this assertion. These results will not be needed elsewhere in this paper, and may be safely omitted by the reader.
We begin by describing the tangent bundle to an $\infty$-category of the form $\calC_{A/}$.

\begin{proposition}\label{sametan}
Let $\calC$ be a presentable $\infty$-category containing an object $A$, and let
$\calD = \calC_{A/}$. Let $T_{\calC}$ and $T_{\calD}$ denote tangent bundles to $\calC$ and $\calD$, respectively. Then there is a canonical equivalence
$$ T_{\calD} \simeq T_{\calC} \times_{ \calC } \calD$$
of presentable fibrations over $\calD$.
\end{proposition}

Proposition \ref{sametan} a relative version of the following more elementary observation:

\begin{lemma}\label{spstud}
Let $\calC$ be an $\infty$-category which admits finite limits and let $A$ be an object of $\calC$.
The forgetful functor
$\calC_{A/} \rightarrow \calC$
induces equivalences of $\infty$-categories
$$f: ( \calC_{A/} )_{\ast} \rightarrow \calC_{\ast} \quad \quad g: \Stab( \calC_{A/} ) \rightarrow \Stab( \calC ).$$
\end{lemma}

\begin{proof}
We will prove that $f$ is an equivalence; the assertion that $g$ is an equivalence is an obvious consequence.
Let $1$ denote a final object of $\calC$. Using Proposition \toposref{needed17}, 
we deduce that $\calC_{A/}$ admits a final object, given by a morphism $u: A \rightarrow 1$.
Using Lemma \toposref{pointer}, we deduce the existence of a commutative diagram
$$ \xymatrix{ \calC_{u/} \ar[r]^{f'} \ar[d] & \calC_{1/} \ar[d] \\
( \calC_{A/})_{\ast} \ar[r]^{f} & \calC_{\ast}, }$$
where the vertical arrows are equivalences. It follows that $f$ is an equivalence
if and only if $f'$ is an equivalence. But $f$ is a trivial Kan fibration, since the inclusion
$\{1\} \subseteq \Delta^1$ is right anodyne.
\end{proof}

\begin{proof}[Proof of Proposition \ref{sametan}]
Let $\calE = \Fun( \Delta^1, \calC) \times_{ \Fun( \{1\}, \calC) } \calD$, so that we have a commutative diagram
$$ \xymatrix{ \Fun( \Delta^1, \calD) \ar[dr]^{q} \ar[rr]^{f} & & \calE \ar[dl]^{q'} \\
& \calD, & }$$
where $q$ and $q'$ are presentable fibrations. We first claim that $f$ carries $q$-limit diagrams
to $q'$-limit diagrams. In view of Propositions \toposref{chocolatelast} and \toposref{relcolfibtest}, it will suffice to verify the following pair of assertions:

\begin{itemize}

\item[$(i)$] For each object $\overline{B} \in \calD$, corresponding to a morphism
$A \rightarrow B$ in $\calC$, the induced map of fibers
$$ f_{\overline{B}}: \calD^{/\overline{B}} \rightarrow \calC^{/B}$$
preserves limits.

\item[$(ii)$] The map $f$ carries $q$-Cartesian morphisms to $q'$-Cartesian morphisms.
\end{itemize}

To prove $(i)$, we observe that $f_{\overline{B}}$ is equivalent to the forgetful functor
$(\calC_{/B})_{A/} \rightarrow \calC_{/B},$
which preserves limits by Proposition \toposref{needed17}.
Assertion $(ii)$ is equivalent to the requirement that the forgetful functor
$\calD \rightarrow \calC$ preserves pullback diagrams, which follows again from Proposition \toposref{needed17}. 

Using Remark \ref{soon}, we can identify $T_{\calC} \times_{ \calC} \calD$ with the
stable envelope of the presentable fibration $q'$. It follows from the universal property
of Proposition \ref{psycher} that the map $f$ fits into a commutative diagram
$$ \xymatrix{ T_{\calD} \ar[r]^{\overline{f} } \ar[d] & T_{\calC} \times_{\calC} \calD \ar[d] \\
\Fun( \Delta^1, \calD) \ar[r]^{f} & \calE. }$$
To complete the proof, we will show that $\overline{f}$ is an equivalence.
In view of Corollary \toposref{usefir}, it will suffice to show that for each $\overline{B} \in \calD$ classifying
a map $A \rightarrow B$ in $\calC$, the induced map
$\Stab( \calD^{/\overline{B}} ) \rightarrow \Stab( \calC^{/B} )$
is an equivalence of $\infty$-categories. This follows immediately from Lemma \ref{spstud}.
\end{proof}

We now wish to study the relationship between the cotangent complex functors
of $\calC$ and $\calC_{A/}$, where $A$ is an object of $\calC$. For this, it is convenient to introduce a bit of terminology.

\begin{definition}
Let $F, F': \calC \rightarrow \calD$ be a functors from an $\infty$-category $\calC$ to an $\infty$-category $\calD$, and let $\alpha: F \rightarrow F'$ be a natural transformation. We will say that $\alpha$
is {\it coCartesian} if, for every morphism $C \rightarrow C'$ in $\calC$, the induced diagram
$$ \xymatrix{ F(C) \ar[r] \ar[d]^{\alpha_{C}} & F(C') \ar[d]^{\alpha_{C'}} \\
F'(C) \ar[r] & F'(C') }$$
is a pushout square in $\calD$.
\end{definition}

The basic properties of the class of coCartesian natural transformations are summarized in the following lemma:

\begin{lemma}\label{kudz}
\begin{itemize}
\item[$(1)$] Let $F, F', F'': \calC \rightarrow \calD$ be functors between $\infty$-categories,
and let $\alpha: F \rightarrow F'$ and $\beta: F' \rightarrow F''$ be natural transformations.
If $\alpha$ is coCartesian, then $\beta$ is coCartesian if and only if $\beta \circ \alpha$
is coCartesian.

\item[$(2)$] Let $F: \calC \rightarrow \calD$ be a functor between $\infty$-categories,
let $G,G': \calD \rightarrow \calE$ be a pair of functors, and let $\alpha: G \rightarrow G'$
be a natural transformation. If $\alpha$ is coCartesian, then so is the induced transformation
$GF \rightarrow G'F$.

\item[$(3)$]  Let $F,F': \calC \rightarrow \calD$ be a pair of functors between $\infty$-categories,
let $G: \calD \rightarrow \calE$ another functor, and let $\alpha: F \rightarrow F'$
be a natural transformation. If $\alpha$ is coCartesian and $G$ preserves all pushout squares which exist in $\calD$, then the induced transformation $GF \rightarrow GF'$ is coCartesian.

\end{itemize}
\end{lemma}

\begin{definition}
We will say that a commutative diagram of $\infty$-categories 
$$ \xymatrix{ \calD \ar[r]^{H} \ar[d]^{G} & \calC \ar[d]^{G'} \\
\calD' \ar[r]^{H'} & \calC' }$$
is {\it rectilinear} if the following conditions are satisfied:
\begin{itemize}
\item[$(1)$] The functors $G$ and $G'$ admit left adjoints, which we will denote by
$F$ and $F'$ respectively.
\item[$(2)$] The identity map $H' G \simeq G' H$ induces a coCartesian natural transformation
$F' H' \rightarrow HF$. 
\end{itemize}
\end{definition}

\begin{remark}
The condition of being rectilinear is closely related to the condition of being {\it left adjointable},
as defined in \S \toposref{propertopoi}.
\end{remark}

\begin{proposition}\label{spazzz}
Let $\calC$ be a presentable $\infty$-category containing an object $A$ and
let $\calD = \calC_{A/}$. Let $G: T_{\calC} \rightarrow \calC$ denote the composite map
$$T_{\calC} \rightarrow \Fun( \Delta^1, \calC) \rightarrow \Fun( \{0\}, \calC) \simeq \calC,$$
and let $G': T_{\calD} \rightarrow \calD$ be defined similarly, so that we have a commutative diagram
$$ \xymatrix{ T_{\calD} \ar[r] \ar[d] & T_{\calC} \ar[d] \\
\calD \ar[r] & \calC }$$
$($see the proof of Proposition \ref{sametan}$)$. Then the above diagram is rectilinear.
\end{proposition}

\begin{corollary}
Let $\calC$ and $\calD = \calC_{A/}$ be as in Proposition \ref{spazzz}, and let
$L^{\calC}: \calC \rightarrow T_{\calC}$ and $L^{\calD}: \calD \rightarrow T_{\calD}$
be cotangent complex functors for $\calC$ and $\calD$, respectively.
Then:
\begin{itemize}
\item[$(1)$] Let $p: \calD \rightarrow \calC$ be the projection, and let
$q: T_{\calD} \rightarrow T_{\calC}$ be the induced map. Then
there is a coCartesian natural transformation $L^{\calC} \circ p \rightarrow
q \circ L^{\calD}$.

\item[$(2)$] There is a pushout diagram of functors
$$ \xymatrix{ L^{\calC}_{A} \ar[r] \ar[d] & L^{\calC} \circ p \ar[d] \\
0 \ar[r] & q \circ L^{\calD}. }$$
Here the terms in the left hand column indicate the constant functors
taking the values $L^{\calC}_{A}, 0 \in \Stab(\calC^{/A}) \subseteq T_{\calC}$.

\item[$(3)$] The functor $q \circ L^{\calD}: \calD \rightarrow T_{\calC}$ can be identified with the functor
$B \mapsto L_{B/A}$. 
\end{itemize}
\end{corollary}

\begin{proof}
Assertion $(1)$ is merely a reformulation of Proposition \ref{spazzz}.
To prove $(2)$, we let $e: \calD \rightarrow \calD$ denote the constant functor
taking the value $\id_{A} \in \calD$, so that we have a natural transformation
$\alpha: e \rightarrow \id_{\calD}$. Applying the coCartesian transformation 
of $(1)$ to $\alpha$ yields the desired diagram, since
$L^{\calD} \circ e$ vanishes by Remark \ref{urn}.
Assertion $(3)$ follows immediately from $(2)$ and the definition
of the relative cotangent complex.
\end{proof}

To prove Proposition \ref{spazzz}, we observe that the square in question fits into a commutative diagram
$$ \xymatrix{ T_{\calD} \ar[r] \ar[d] & T_{\calC} \ar[d] \\
\Fun( \Delta^1, \calD) \ar[r] \ar[d] & \Fun( \Delta^1, \calC) \ar[d] \\
\Fun( \{0\}, \calD) \ar[r] & \Fun( \{0\}, \calC). }$$
It will therefore suffice to prove the following three results:

\begin{lemma}\label{lem1}
Suppose given a commutative diagram of $\infty$-categories
$$ \xymatrix{ \calD \ar[r]^{H} \ar[d]^{G_0} & \calC \ar[d]^{G'_0} \\
\calD' \ar[r]^{H'} \ar[d]^{G'_0} & \calC' \ar[d]^{G'_1} \\
\calD'' \ar[r]^{H''} & \calC'' }
$$
If the upper and lower squares are rectilinear, then the outer square is rectilinear.
\end{lemma}

\begin{lemma}\label{lem2}
Let $p: \calD \rightarrow \calC$ be a functor between $\infty$-categories. Then the commutative diagram
$$ \xymatrix{ \Fun( \Delta^1, \calD) \ar[r] \ar[d]^{G} & \Fun( \Delta^1, \calC) \ar[d]^{G'} \\
\Fun( \{0\}, \calD) \ar[r] & \Fun( \{0\}, \calC) }$$
is rectilinear.
\end{lemma}

\begin{lemma}\label{lem3}
Let $\calC$ be a presentable $\infty$-category containing an object $A$, and let
$\calD = \calC_{A/}$. Then the diagram
$$ \xymatrix{ T_{\calD} \ar[r] \ar[d] & T_{\calC} \ar[d] \\
\Fun( \Delta^1, \calD) \ar[r] & \Fun( \Delta^1, \calC) }$$
$($see the proof of Proposition \ref{sametan}$)$ is rectilinear.
\end{lemma}

\begin{proof}[Proof of Lemma \ref{lem1}]
We observe that $G_1 G_0$ admits a left adjoint $L_0 L_1$, where
$L_0$ and $L_1$ are left adjoints to $G_0$ and $G_1$, respectively. Similarly,
$G'_1 G'_0$ admits a left adjoint $L'_0 L'_1$. It remains only to show that the composite transformation
$$ L_0 L_1 H'' \rightarrow L_0 H' L'_1 \rightarrow H L'_0 L'_1$$
is coCartesian, which follows from Lemma \ref{kudz}.
\end{proof}

\begin{proof}[Proof of Lemma \ref{lem2}]
For any $\infty$-category $\calC$, the evaluation functor
$\Fun( \Delta^1, \calC) \rightarrow \Fun( \{0\}, \calC) \simeq \calC$ has a left adjoint given
by the diagonal embedding $\delta_{\calC}: \calC \rightarrow \Fun( \Delta^1, \calC)$. In the situation
of Lemma \ref{lem2}, we obtain a {\em strictly commutative} diagram
of adjoint functors
$$ \xymatrix{ \Fun( \Delta^1, \calD) \ar[r] & \Fun( \Delta^1, \calC) \\
\calD \ar[r] \ar[u]^{\delta_{\calD}} & \calC \ar[u]^{\delta_{\calC}}. }$$
It now suffices to observe that that any invertible natural transformation is automatically
coCartesian.
\end{proof}

To prove Lemma \ref{lem3}, we once again break the work down into two steps.
First, we need a bit of terminology:

\begin{notation}
For every $\infty$-category $\calC$, we let
$P_{\ast}(\calC)$ denote the full subcategory of
$\Fun( \Delta^2, \calC)$ spanned by those diagrams
$$ \xymatrix{ & B \ar[dr] & \\
A \ar[rr]^{f} \ar[ur] & & C }$$
such that $f$ is an equivalence.
If $\calC$ is presentable, then the evaluation map
$$ P_{\ast}(\calC) \rightarrow \Fun( \Delta^{ \{1,2\} }, \calC) \simeq \Fun( \Delta^1, \calC)$$
exhibits $P_{\ast}(\calC)$ as a pointed envelope of the presentable fibration
$\Fun( \Delta^1, \calC) \rightarrow \Fun( \{1\}, \calC) \simeq \calC$.
\end{notation}

Now let $p: \calD \rightarrow \calC$ be as in Lemma \ref{lem3}. The proof of Proposition \ref{sametan} gives a commutative diagram
$$ \xymatrix{ T_{\calD} \ar[r] \ar[d] & T_{\calC} \ar[d] \\
P_{\ast}(\calD) \ar[r] \ar[d] & P_{\ast}(\calC) \ar[d] \\
\Fun( \Delta^1, \calD) \ar[r] & \Fun( \Delta^1, \calC). }$$
We wish to prove that the outer square is rectilinear. In view of Lemma \ref{lem1},
it will suffice to prove the upper and bottom squares are rectilinear. For the upper square, we
observe that Proposition \ref{sametan} gives a homotopy pullback diagram
$$ \xymatrix{ T_{\calD} \ar[r] \ar[d] & T_{\calC} \ar[d] \\
P_{\ast}(\calD) \ar[d] \ar[r] & P_{\ast}(\calC) \ar[d] \\
\calC \ar[r] & \calD. }$$
Lemma \ref{lem3} is therefore a consequence of the following pair of results:

\begin{lemma}\label{lem4}
Suppose given a commutative diagram
$$ \xymatrix{ \calD \ar[r] \ar[d] & \calC \ar[d]^{G} \\
\calD' \ar[r] \ar[d] & \calC'  \ar[d] \\
\calD'' \ar[r] & \calC'' }$$ 
of $\infty$-categories, where each square is homotopy Cartesian.
If $G$ admits a left adjoint relative to $\calC''$, then the upper square is
rectilinear.
\end{lemma}

\begin{lemma}\label{lem5}
Let $\calC$ be a presentable $\infty$-category containing an object $A$, let
$\calD = \calC_{A/}$. Then the diagram
$$ \xymatrix{ P_{\ast}(\calD) \ar[r] \ar[d]^{G'} & P_{\ast}(\calC) \ar[d]^{G} \\
\Fun( \Delta^1, \calD) \ar[r] & \Fun( \Delta^1, \calC) }$$
is rectilinear.
\end{lemma}

\begin{proof}[Proof of Lemma \ref{lem4}]
Without loss of generality, we may assume that every map in the diagram
$$ \xymatrix{ \calD \ar[r] \ar[d]^{G'} & \calC \ar[d]^{G} \\
\calD' \ar[r] \ar[d] & \calC'  \ar[d] \\
\calD'' \ar[r] & \calC'' }$$ 
is a categorical fibration, and that each square is a pullback in the category of simplicial sets.
Let $F$ be a left adjoint to $G$ relative to $\calC''$, and choose a counit map
$v: F \circ G \rightarrow \id_{\calC}$ which is compatible with the projection
to $\calC''$ (so that $v$ can be identified with a morphism in the $\infty$-category
$\bHom_{ \calC''}( \calC, \calC)$ ). Let $F': \calD' \rightarrow \calD$
be the map induced by $F$, so that $v$ induces a natural transformation
$F' \circ G' \rightarrow \id_{\calD}$, which is easily verified to be the counit of an adjunction.
It follows that we have a strictly commutative diagram
$$ \xymatrix{ \calD \ar[r] & \calC \\
\calD' \ar[u]^{F'} \ar[r] & \calC' \ar[u]^{F}. }$$
To complete the proof it suffices to observe that any invertible natural transformation is automatically coCartesian.
\end{proof}

\begin{proof}[Proof of Lemma \ref{lem5}]
The forgetful functor $G: P_{\ast}(\calC) \rightarrow \Fun( \Delta^1, \calC)$ has a left adjoint $F$.
We can identify $F$ with the functor which carries a diagram $B \rightarrow C$
in $\calC$ to the induced diagram
$$ \xymatrix{ & B \coprod C \ar[dr] & \\
C \ar[ur] \ar[rr]^{\id} & & C, }$$
regarded as an object of $P_{\ast}(\calC)$. Similarly, $G'$ has a left adjoint $F'$, which carries
a diagram $A \rightarrow B \rightarrow C$ to the induced diagram
$$ \xymatrix{ & & B \coprod_{A} C \ar[dr] & \\
A \ar[r] & C \ar[rr]^{\id} \ar[ur] & & C.}$$
We observe that a diagram in $P_{\ast}(\calC)$ is a pushout square if and only if
it determines a pushout square in $\calC$ after evaluating at each vertex in $\Delta^2$.
Unwinding the definition, we see that the Lemma \ref{lem5} is equivalent to the following elementary assertion:
for every commutative diagram
$$ \xymatrix{ A \ar[dr] \ar[r] & B \ar[r] \ar[d] & C \ar[d] \\
& B' \ar[r] & C' }$$
in $\calC$, the induced diagram
$$ \xymatrix{ B \coprod C \ar[d] \ar[r] & B \coprod_{A} C \ar[d] \\
B' \coprod C' \ar[r] & B' \coprod_{A} C' }$$
is a pushout square.
\end{proof}

\subsection{The Tangent Bundle of the $\infty$-Category of $E_{\infty}$-Rings}\label{sec2}

Let $A$ be a commutative ring, and let $M$ be an $A$-module. Then the direct
sum $A \oplus M$ inherits the structure of a commutative ring, with multiplication described by the formula
$$ (a, m) (a', m') = (aa', am' + a' m).$$
We wish to describe an analogous construction in the case where
$A$ is an $E_{\infty}$-ring and $M$ is a module {\em spectrum} over $A$.
Of course, in this context we cannot define a ring structure on $A \oplus M$ simply by writing formulas: we must obtain $A \oplus M$ in some other way. We begin by listing some features which we expect of this construction:

\begin{itemize}
\item[$(a)$] The square-zero extension $A \oplus M$ admits a projection map
$A \oplus M \rightarrow A$.
\item[$(b)$] The square-zero extension $A \oplus M$ depends functorially on $M$.
In other words, it is given by a functor
$$ G: \Mod_{A} \rightarrow (\EInfty)_{/A}.$$
\item[$(c)$] The underlying spectrum of $A \oplus M$ can be identified (functorially) with
the usual coproduct of $A$ and $M$ in the $\infty$-category of $\Spectra$.
\end{itemize}

Condition $(c)$ automatically implies that the functor $G$ preserves limits. Since the
$\infty$-category $\Mod_{A}$ is stable, the functor $G$ would then be equivalent to a composition
$$ \Mod_{A} \stackrel{G'}{\rightarrow} \Stab( (\EInfty)_{/A} ) \stackrel{\Omega^{\infty}}{\rightarrow} (\EInfty)_{/A}.$$
In fact, we can say more: the functor $G'$ is an equivalence of $\infty$-categories. 
Let us describe a functor $F'$ which is homotopy inverse to $G'$.
Let $X$ be an object of $\Stab( (\EInfty)_{/A}$. Then the $0$th space of $X$ is
a {\em pointed} object of $(\EInfty)_{/A}$, which we can identify with an {\em augmented
$A$-algebra}: that is, an $E_{\infty}$-ring $B$ which fits into a commutative diagram
$$ \xymatrix{ & B \ar[dr]^{f} & \\
A \ar[ur] \ar[rr]^{\id} & & A. }$$
We now observe that in this situation, the kernel $\ker(f)$ inherits the structure of an
$A$-module. We can therefore define a functor $F': \Stab( ( \EInfty)_{/A} ) \rightarrow \Mod_A$
by setting $F'(X) = \ker(f)$. 

We now have an approach to defining the desired functor $G$. Namely, we first construct
the functor $F': \Stab( ( \EInfty)_{/A} ) \rightarrow \Mod_{A}$ described above. If we can prove that $F'$ is an equivalence of $\infty$-categories, then we can define $G'$ to be a homotopy inverse to $F'$, and
$G$ to be the composition of $G'$ with the $0$th space functor $\Omega^{\infty}: \Stab( (\EInfty)_{/A}) \rightarrow (\EInfty)_{/A}$. 

Our goal in this section is to flesh out the ideas sketched above. It will be convenient to work in a bit more generality: rather than only considering commutative algebras, we consider algebras over an arbitrary coherent $\infty$-operad. We begin with some generalities.


\begin{definition}
Let $\calO^{\otimes}$ be an $\infty$-operad. We will say that a map
$q: \calC^{\otimes} \rightarrow \calO^{\otimes}$ is a {\it stable $\calO$-monoidal
$\infty$-category} if the following conditions are satisfied:
\begin{itemize}
\item[$(1)$] The map $q$ is a coCartesian fibration of $\infty$-operads.
\item[$(2)$] For each object $X \in \calO$, the fiber
$\calC_{X}$ is a stable $\infty$-category.
\item[$(3)$] For every morphism $\alpha \in \Mul_{\calO}( \{ X_i \}, Y)$, the
associated functor
$\alpha_{!}: \prod_{i} \calC_{X_i} \rightarrow \calC_{Y}$ is exact separately
in each variable.
\end{itemize}
\end{definition}

\begin{remark}\label{underil}
Let $\calO^{\otimes}$ be an $\infty$-operad and let
$q: \calC^{\otimes} \rightarrow \calO^{\otimes}$ be a stable $\calO$-monoidal
$\infty$-category. Then the $\infty$-category $\Fun_{ \calO}( \calO, \calC)$ of
sections of the restricted map $q_0: \calC \rightarrow \calO$ is stable: this follows
immediately from Proposition \toposref{prestorkus}.
\end{remark}

\begin{definition}
Let $\calO^{\otimes}$ be a unital $\infty$-operad, and let
$q: \calC^{\otimes} \rightarrow \calO^{\otimes}$ be a coCartesian fibration of
$\infty$-operads. An {\it augmented $\calO$-algebra object of $\calC$} is
a morphism $f: A \rightarrow A_0$ in $\Alg_{\calO}(\calC)$ such that
$A_0$ is an initial object of $\Alg_{\calO}(\calC)$. (In view of Proposition \symmetricref{gargle1}, this is equivalent to the requirement that $A_0(0) \rightarrow A_0(X)$ is $q$-coCartesian whenever
$0 \rightarrow X$ is a morphism in $\calO^{\otimes}$ with $0 \in \calO^{\otimes}_{\seg{0}}$.)
We let $\Alg_{\calO}^{\aug}(\calC)$ denote the full subcategory of 
$\Fun( \Delta^1, \Alg_{\calO}(\calC))$ spanned by the augmented $\calO$-algebra objects in $\calC$.

Suppose further that $\calC^{\otimes}$ is a stable $\calO$-monoidal $\infty$-category, so that
$\Fun_{\calO}(\calO, \calC)$ is stable (Remark \ref{underil}). Let
$\theta: \Alg_{\calO}(\calC) \rightarrow \Fun_{\calO}(\calO,\calC)$ denote the restriction functor.
Given an augmented $\calO$-algebra object $A \rightarrow A_0$ of $\calC$, we define
the {\it augmentation ideal} to the the fiber of the induced morphism $\theta(A) \rightarrow \theta(A_0)$.
The formation of augmentation ideals determines a functor
$$ \Alg_{\calO}^{\aug}(\calC) \rightarrow \Fun_{\calO}(\calO, \calC).$$
\end{definition}

\begin{remark}\label{stabbe}
Let $\calO^{\otimes}$ be a small $\infty$-operad and let $q: \calC^{\otimes} \rightarrow \calO^{\otimes}$
be a presentable $\calO$-monoidal $\infty$-category. It follows from Proposition \toposref{prestorkus} that the $\infty$-category $\Fun_{\calO}(\calO, \calC)$ is presentable, and that for each
object $X \in \calO$ the evaluation functor $e_{X}: \Fun_{\calO}(\calO, \calC)$ preserves small limits and small colimits. It follows from Corollary \toposref{adjointfunctor} that $e_{X}$ admits both a left
and a right adjoint, which we will denote by $(e_{X})_{!}$ and $(e_{X})_{\ast}$.
\end{remark}

The following result characterizes the augmentation ideal functor by a universal property:

\begin{proposition}\label{hugger}
Let $\calO^{\otimes}$ be a small unital $\infty$-operad and let
$\calC^{\otimes} \rightarrow \calO^{\otimes}$ be a presentable stable $\calO$-monoidal $\infty$-category. Let $0_{\calC}$ denote a zero object of the stable $\infty$-category $\Fun_{\calO}(\calO, \calC)$,
and let $1_{\calC}$ denote an initial object of $\Alg_{\calO}(\calC)$. Then there
exists a pair of adjoint functors
$ \Adjoint{f}{ \Fun_{\calO}(\calO, \calC) }{ \Alg^{\aug}_{\calO}(\calC)}{g}$ with the following properties:
\begin{itemize}
\item[$(1)$] The functor $f$ is given by composition
$$\Fun_{\calO}(\calO, \calC) \simeq \Fun_{\calO}(\calO, \calC)^{/0_{\calC}} \rightarrow \Alg_{\calO}(\calC)^{/ 1_{\calC} }
\simeq \Alg^{\aug}_{\calO}(\calC),$$ where the middle map is induced by a left adjoint
$F$ to the forgetful functor $G: \Alg_{\calO}(\calC) \rightarrow \Fun_{\calO}(\calO, \calC)$.
Here we implicitly invoke the identification $1_{\calC} \simeq F(0_{\calC})$; note that
the existence of $F$ follows from Corollary \symmetricref{spaltwell}.

\item[$(2)$] The functor $g: \Alg^{\aug}_{\calO}(\calC) \rightarrow \Fun_{\calO}(\calO, \calC)$ is the augmentation ideal functor.

\item[$(3)$] Let $X$ and $Y$ be objects of $\calO^{\otimes}$, and let
$(e_{X})_{!}: \calC_{X} \rightarrow \Fun_{\calO}(\calO, \calC)$ and
$e_Y: \Fun_{\calO}(\calO, \calC)$ be as in Remark \ref{stabbe}. Then the composition
$$ \calC_{X} \stackrel{(e_{X})_{!}}{\rightarrow}
\Fun_{\calO}(\calO, \calC) \stackrel{f}{\rightarrow}
\Alg^{\aug}_{\calO}(\calC) \stackrel{g}{\rightarrow} \Fun_{\calO}(\calO, \calC)
\stackrel{ e_Y}{\rightarrow} \calC_{Y}$$
is equivalent to the functor
$C \mapsto \coprod_{ n > 0} \Sym^{n}_{\calO,Y}(C)$
(see Construction \symmetricref{aball}).
\end{itemize}
\end{proposition}

\begin{proof}
The existence of the functor $g$ and assertion $(2)$ follow from Proposition \toposref{curpse}, together with the definition of the augmentation ideal functor.
Invoking $(2)$, we deduce that there is a distinguished triangle
$$ g \circ f \rightarrow G \circ F \stackrel{h}{\rightarrow} \underline{G}(1_{\calC}) \rightarrow (g \circ f)[1]$$
in the stable $\infty$-category of functors from $\Fun_{\calO}(\calO, \calC)$ to itself,
where $\underline{ G}(1_{\calC})$ denotes the constant functor taking the value
$G(1_{\calC})$. Theorem 
The results of \S \symmetricref{comm33} guarantee that
$e_{Y} \circ G \circ F \circ (e_{X})_{!}$ can be idenitifed with the functor
$\coprod_{n \geq 0} \Sym^{n}_{\calO,Y}$. We observe that
the map $h$ is split by the inclusion
$\Sym^0_{\calO,Y} \rightarrow \coprod_{n \geq 0} \Sym^{n}_{\calO,Y}$, so that we obtain
an identification of $g \circ f$ with the complementary summand $\coprod_{ n > 0} \Sym^n_{\calO,Y}$.
\end{proof}

\begin{remark}\label{sabler}
Let $\Adjoint{f}{ \Fun_{\calO}(\calO, \calC) }{ \Alg^{\aug}_{\calO}(\calC)}{g}$ be as in 
Proposition \ref{hugger}, and let $X,Y \in \calO$. Unwinding the definitions, we see that the unit map
$\id \rightarrow g \circ f$ induces a functor
$e_{Y} \circ (e_X)_{!} \rightarrow e_{Y} \circ g \circ f \circ (e_{X})_{!}$. This can be identified
with the inclusion of the summand
$\Sym^{1}_{\calO,Y} \rightarrow \coprod_{ n > 0} \Sym^{n}_{\calO,Y}$.
\end{remark}

The main result of this section is the following:

\begin{theorem}\label{uckaboo1}
Let $\calO^{\otimes}$ be a unital $\infty$-operad and let
$\calC^{\otimes} \rightarrow \calO^{\otimes}$ be a stable $\calO$-monoidal $\infty$-category.
Then the augmentation ideal functor
$G: \Alg^{\aug}_{\calO}(\calC) \rightarrow \Fun_{\calO}(\calO,\calC)$ induces an equivalence of $\infty$-categories $\Stab( \Alg^{\aug}_{\calO}(\calC) ) \rightarrow \Stab( \Fun_{\calO}(\calO,\calC)) \simeq 
\Fun_{\calO}(\calO,\calC)$.
\end{theorem}

The proof of Theorem \ref{uckaboo1} will use some basic ideas from Goodwillie's calculus of functors. We refer the reader to \cite{derivative} for a review of this theory (including explanations for some of  the terminology which appears below).

\begin{lemma}\label{linbin}
Let $K$ be a simplicial set. Let $\calC$ be a pointed $\infty$-category which admits finite colimits, and let $\calD$ be stable $\infty$-category which admits sequential colimits and $K$-indexed colimits. Then the derivative functor $D: \Fun_{\ast}( \calC, \calD) \rightarrow \Exc( \calC, \calD)$ preserves $K$-indexed colimits.
\end{lemma}

\begin{proof}
Since $\calD$ is stable, the loop functor $\Omega_{\calD}$ is an equivalence of $\infty$-categories.
It follows that $\Omega_{\calD}$ preserves $K$-indexed colimits. We observe that $\Exc(\calC, \calD)$ is the full subcategory of $\Fun(\calC, \calD)$ spanned by those functors which are right exact; it follows that $\Exc(\calC, \calD)$ is stable under $K$-indexed colimits in $\Fun(\calC, \calD)$. Similarly,
$\Fun_{\ast}(\calC, \calD)$ is stable under $K$-indexed colimits in $\Fun(\calC, \calD)$; we therefore conclude that $K$-indexed colimits in $\Fun_{\ast}(\calC, \calD)$ and $\Exc(\calC, \calD)$ are computed pointwise. The desired result now follows from the formula for computing the derivative given in Remark \bicatref{derexi}.
\end{proof}

\begin{lemma}\label{suksy}
Let $n \geq 2$ be an integer, let $\calC_1, \ldots, \calC_n$ and $\calD$ be pointed $\infty$-categories
which admit finite colimits, and let $F: \calC_1 \times \ldots \times \calC_n \rightarrow \calD$
be a functor which preserves finite colimits separately in each variable.
Then:
\begin{itemize}
\item[$(1)$] For every object $C = (C_1, \ldots, C_n) \in \calC_1 \times \ldots \times \calC_n$, the
canonical map $\alpha: \Sigma_{\calD} F(C) \rightarrow F( \Sigma_{\calC_1 \times \ldots \times \calC_n} C)$
is nullhomotopic.
\item[$(2)$] Suppose that $\calD$ admits finite limits and sequential colimits and that the loop functor $\Omega_{\calD}$ preserves sequential colimits. Then the derivative
$DF: \calC_1 \times \ldots \times \calC_n \rightarrow \calD$ is nullhomotopic.
\end{itemize}
\end{lemma}

\begin{proof}
We first prove $(1)$. Enlarging the universe if necessary, we may assume that
$\calC_1, \ldots, \calC_n$ and $\calD$ are small. Passing to $\infty$-categories of $\Ind$-objects,
we can reduce to the case where $\calC_1, \ldots, \calC_n$ and $\calD$ are presentable, and
the functor $F$ preserves small colimits separately in each variable (since the construction
$\calE \mapsto \Ind(\calE)$ is a symmetric monoidal functor; see \S \symmetricref{comm7}).
Since the $\infty$-categories $\calC_i$ are pointed, evaluation on the (pointed) zero sphere
$S^0$ induces an equivalence of $\infty$-categories
$\Fun'( \SSet_{\ast}, \calC_i) \rightarrow \calC_i$; here $\SSet_{\ast}$ denotes the $\infty$-category of pointed spaces and $\Fun'( \SSet_{\ast}, \calC_i)$ denotes the full subcategory of
$\Fun( \SSet_{\ast}, \calC_i)$ spanned by those functors which preserve small colimits.
In particular, there exist functors $f_i: \SSet_{\ast} \rightarrow \calC_i$ such that
$f_i( S^0) \simeq C_i$. We may therefore replace $\calC_i$ by $\SSet_{\ast}$, and reduce to the case
where each of the objects $C_i$ can be identified with the zero sphere $S^0 \in \SSet_{\ast}$.

The functor $F: \SSet_{\ast} \times \ldots \times \SSet_{\ast} \rightarrow \calD$ preserves colimits separately in each factor, and therefore factors as a composition
$$ \SSet_{\ast} \times \ldots \times \SSet_{\ast} \stackrel{F'}{\rightarrow} \SSet_{\ast} \otimes
\ldots \otimes \SSet_{\ast} \stackrel{F''}{\rightarrow} \calD$$
where $F''$ preserves small colimits.
Here the tensor product is taken in the monoidal $\infty$-category $\LPress$ of presentable $\infty$-categories (see \S \monoidref{jurmit}) and is equivalent to $\SSet_{\ast}$, while the functor
$F': \SSet_{\ast} \times \ldots \times \SSet_{\ast} \rightarrow \SSet_{\ast}$ can be identified
with the classical smash product $( X_1, \ldots, X_n) \mapsto X_1 \wedge \ldots \wedge X_n$
of pointed spaces. We may therefore replace $\calD$ by $\SSet_{\ast}$, and thereby reduce to the case where $F$ is given by the iterated smash product. In this case, $\alpha$ can be identified with
a map of pointed spaces $S^1 \rightarrow S^1 \wedge \ldots \wedge S^1 \simeq S^n$, which
is nullhomotopic since the $n$-sphere $S^n$ is simply connected for $n > 1$.

We now prove $(2)$. According to Remark \bicatref{derexi}, the derivative
$DF( C_1, \ldots, C_n)$ can be computed as the colimit of the sequence of maps
$$ F( C_1, \ldots, C_n) \rightarrow \Omega_{\calD}( \Sigma_{\calC_1} C_1, \ldots,
\Sigma_{\calC_n} C_n) \rightarrow \ldots$$
Assertion $(1)$ implies that every map in this sequence is nullhomotopic, so the colimit
of the sequence is equivalent to the zero object $\ast \in \calD$. It follows that
$DF$ can be identified with the zero object of $\Fun(\calC, \calD)$.
\end{proof}

\begin{remark}\label{skabber}
Let $\calC \rightarrow \calO$ be a presentable fibration of $\infty$-categories,
where $\calO$ is small. For each $X \in \calO$, let $(e_X)_{!}$ denote 
a left adjoint to the evaluation functor $e_{X}: \Fun_{\calO}(\calO, \calC) \rightarrow \calC_{X}$.
Then the essential images of the functors $(e_{X})_{!}$ generate the $\infty$-category
$\calD = \Fun_{\calO}(\calO, \calC)$ under small colimits. To prove this, let $\calD_0$ denote the smallest
full subcategory of $\calD$ containing the essential image of
each $(e_X)_{!}$ and closed under small colimits in $\calD$. Since the essential image
of each $(e_{X})_{!}$ is generated under small colimits by a small collection of objects,
we deduce that $\calD_0 \subseteq \calD$ is presentable. Let $D$ be an object of
$\calD$; we wish to prove that $D \in \calD_0$. Let $\chi: \calD^{op} \rightarrow \SSet$ be the functor represented by $D$. The composite functor
$$ \chi| \calD_0: \calD_0^{op} \rightarrow \calD^{op} \rightarrow \SSet$$
preserves small limits, and is therefore representable by an object $D_0 \in \calD_0$
(Proposition \toposref{representable}). We therefore obtain a map $f: D_0 \rightarrow D$
which exhibits $D_0$ as a $\calD_0$-colocalization of $D$. In particular, for each
$X \in \calO$ and each $C \in \calC_{X}$, composition with $f$ induces a homotopy equivalence
$$ \bHom_{\calC_X}( C, D_0(X)) \simeq
\bHom_{\calD}( (e_X)_{!}(C), D_0) \rightarrow \bHom_{\calD}( (e_X)_{!}(C), D)
\simeq \bHom_{\calC_{X}}( C, D(X) ).$$
This proves that $e_X(f)$ is an equivalence for each $X \in \calO$, so that $f$ is an equivalence
and $D \in \calD_0$ as required.
\end{remark}

\begin{proposition}\label{sillfun}
Let $\calO^{\otimes}$ be a unital $\infty$-operad, and let
$\calC^{\otimes}$ be a presentable stable $\calO$-monoidal $\infty$-category.
Let $G: \Alg^{\aug}_{\calO}(\calC) \rightarrow \Fun_{\calO}(\calO,\calC)$ be the augmentation ideal functor, and let $F$ be a left adjoint to $G$. Then the unit map $\id \rightarrow GF$ induces an equivalence of derivatives $\alpha: D(\id) \rightarrow D(GF)$.
\end{proposition}

\begin{proof}
We wish to show that for every object $M \in \Fun_{\calO}(\calO, \calC)$, the 
natural transformation $\alpha$ induces an equivalence
$$ \alpha_{M}: M \simeq D(\id)(M) \rightarrow D(GF)(M).$$
Since both sides are compatible with the formation of colimits in $M$, it will suffice
to prove this in the case where $M = (e_X)_{!}(C)$ for some $X \in \calO$ and
some $C \in \calC_{X}$ (Remark \ref{skabber}). Moreover, to prove that $\alpha_{M}$ is
an equivalence, it suffices to show that $e_{Y}(\alpha_M)$ is an equivalence in
$\calC_{Y}$, for each $Y \in \calO$. In other words, it suffices to show that
$\alpha$ induces an equivalence
$$ \beta: e_Y \circ (e_X)_{!} \rightarrow e_{Y} \circ D(GF) \circ (e_X)_{!}.$$
Since the functors $(e_{X})_{!}$ and $e_{Y}$ are exact, we can identify the latter composition
with $D( e_{Y} \circ G \circ F \circ (e_X))_{!}$ (Proposition \bicatref{easychain}).

According to Proposition \ref{hugger}, the functor 
$e_{Y} \circ G \circ F \circ (e_X)_{!}$ can be identified with the total symmetric power functor
$C \mapsto \coprod_{ n > 0 } \Sym^{n}_{\calO,Y}(C)$. According to Remark \ref{sabler},
we can express this as the coproduct of $e_{Y} \circ (e_X)_{!}$ with the functor
$T$ given by the formula 
$T(C) \simeq \coprod_{ n \geq 2} \Sym^{n}_{\calO,Y}(C)$. In view of Lemma \ref{linbin}, it will suffice
to show that each of the derivatives $D \Sym^{n}_{\calO,Y}$ is nullhomotopic for $n \geq 2$.
We observe that $\Sym^{n}_{\calO,Y}$ can be expressed as a colimit of functors of the form
$$ \calC_{X} \stackrel{\delta}{\rightarrow} \calC_{X}^{n} \stackrel{\gamma_{!}}{\rightarrow} \calC_{Y}
$$
where $\gamma_{!}$ denotes the functor associated to an operation
$\gamma \in \Mul_{\calO}( \{ X \}_{1 \leq i \leq n}$. In view of Lemma \ref{linbin}, it suffices to show that each constituent $D( \delta \circ \gamma_{!})$ is nullhomotopic, which follows from Lemma \ref{suksy}.
\end{proof}

\begin{lemma}\label{sidewise}
Let $\calC$ be a stable $\infty$-category, let $f: C \rightarrow D$ be a morphism in
$\calC$, and let $f^{\ast}: \calC^{/D} \rightarrow \calC^{/C}$ be the functor given by pullback along $f$. 
Then:
\begin{itemize}
\item[$(1)$] The functor $f^{\ast}$ is conservative.
\item[$(2)$] Let $K$ be a weakly contractible simplicial set, and assume that
$\calC$ admits $K$-indexed colimits. Then the functor $f^{\ast}$ preserves $K$-indexed colimits.
\end{itemize}
\end{lemma}

\begin{proof}
Let $\calE$ denote the full subcategory of $\Fun( \Delta^1 \times \Delta^1, \calC)
\times_{ \Fun( \{1\} \times \Delta^1, \calC )} \{f\}$ spanned by the pullback diagrams
$$ \xymatrix{  C' \ar[d] \ar[r] & D' \ar[d] \\
C \ar[r]^{f} & D.}$$
Since $\calC$ admits pullbacks, Proposition \toposref{lklk} implies that
evaluation along $\Delta^1 \times \{1\}$ induces a trivial Kan fibration
$\calE \rightarrow \calC^{/D}$. Let $g$ denote a section of this trivial fibration. Then
the functor $f^{\ast}$ can be identified with the composition
$$ \calC^{/D} \stackrel{g}{\rightarrow} \calE \stackrel{g'}{\rightarrow} \calC^{/C},$$
where $g'$ is given by evaluation along $\Delta^1 \times \{0\}$. 

Let $u$ be a morphism in $\calC^{/D}$. Let $\sigma$ denote the kernel of the morphism $g(u)$, formed in the stable $\infty$-category $\Fun(\Delta^1 \times \Delta^1, \calC)$. Then $\sigma$ is a pullback diagram
$$ \xymatrix{ W \ar[r] \ar[d] & X \ar[d] \\
Y \ar[r] & Z }$$
in the $\infty$-category $\calC$. The objects $Y$ and $Z$ are both zero, so the bottom horizontal map is an equivalence. It follows that the upper horizontal map is an equivalence. If $f^{\ast}(u)$ is an equivalence, then $W \simeq 0$. It follows that $X \simeq 0$, so that $u$ is an equivalence
in $\calC^{/D}$. This completes the proof of $(1)$.

To prove $(2)$, let us choose a colimit diagram $\overline{p}: K^{\triangleright} \rightarrow \calC^{/D}$.
Let $\overline{q} = g \circ \overline{p}$. We wish to prove that $g' \circ \overline{q}$ is a colimit diagram in $\calC^{/C}$. In view of Proposition \toposref{needed17}, it will suffice to show that
$\overline{q}$ defines a colimit diagram in $\Fun( \Delta^1 \times \Delta^1, \calC)$. Let
$q = \overline{q} | K$, and let $\sigma \in \Fun( \Delta^1 \times \Delta^1, \calC)$ be a colimit of
$q$ in $\Fun(\Delta^1 \times \Delta^1, \calC)$. Since the class of pushout diagrams in
$\calC$ is stable under colimits, we conclude that $\sigma$ is a pushout diagram.
Let $\sigma'$ be the image under $\overline{q}$ of the cone point of $K^{\triangleright}$, let
$\alpha: \sigma \rightarrow \sigma'$ be the map determined by $\overline{q}$, and let
$\tau \in \Fun( \Delta^1 \times \Delta^1, \calC)$ be the cokernel of $\alpha$. We wish to prove
that $\alpha$ is an equivalence, which is equivalent to the assertion that $\tau \simeq 0$. 
We may view $\tau$ as a pushout diagram
$$ \xymatrix{ W \ar[r] \ar[d] & X \ar[d] \\
Y \ar[r] & Z }$$
in $\calC$. Since $\calC$ is stable, this diagram is also a pullback. Consequently, it will suffice to show that the objects $X, Y, Z \in \calC$ are equivalent to zero. For the object $X$, this follows from our assumption that $\overline{p}$ is a colimit diagram (and Proposition \toposref{needed17}). To show that $Y$ and $Z$ are zero, it suffices to observe that every constant map
$K^{\triangleright} \rightarrow \calC$ is a colimit diagram, because $K$ is weakly contractible
(Corollary \toposref{silt}).
\end{proof}

\begin{lemma}\label{curpse2}
Suppose given an adjunction of $\infty$-categories
$$ \Adjoint{F}{\calC}{\calD}{G}$$
where $\calC$ is stable. Let $C$ be an object
of $\calC$, and consider the induced adjunction
$$ \Adjoint{f}{\calC_{/C}}{\calD_{/FC}}{g}$$
(see Proposition \toposref{curpse}). Then:
\begin{itemize}
\item[$(1)$] If the functor $G$ is conservative, then $g$ is conservative.
\item[$(2)$] Let $K$ be a weakly contractible simplicial set. Assume that $\calC$ and $\calD$ admit $K$-indexed colimits, that the functor $G$ preserves $K$-indexed colimits, and that $\calC$ is stable.
Then the $\infty$-categories $\calD^{/FC}$ and $\calC^{/C}$ admit $K$-indexed colimits, and the functor $g$ preserves $K$-indexed colimits.
\end{itemize}
\end{lemma}

\begin{proof}
We first prove $(1)$. Proposition \toposref{curpse} shows that $g$ can be written as a composition
$$ \calD_{/FC} \stackrel{g'}{\rightarrow} \calC_{/GFC} \stackrel{g''}{\rightarrow} \calC_{/C},$$
where $g'$ is induced by $G$ and $g''$ is given by pullback along the unit map
$C \rightarrow GFC$. It will therefore suffice to show that $g'$ and $g''$ are conservative.
We have a commutative diagram of $\infty$-categories
$$ \xymatrix{ \calD^{/FC} \ar[d] \ar[r]^{g'} & \calC^{/GFC} \ar[d] \\
\calD \ar[r]^{G} & \calC. }$$
Since the vertical functors detect equivalences and $G$ is conservative, we deduce that $g'$ is conservative. It follows from Lemma \toposref{sidewise} that $g''$ is conservative as well.

We now prove $(2)$. Proposition \toposref{needed17} implies that the $\infty$-categories $\calC^{/C}$,
$\calC^{/GFC}$, and $\calD^{/FC}$ admit $K$-indexed colimits. Consequently, it will suffice to show that
$g'$ and $g''$ preserve $K$-indexed colimits. For the functor $g'$, this follows from
Proposition \toposref{needed17} and our assumption that $G$ preserves $K$-indexed colimits.
For the functor $g''$, we invoke Lemma \ref{sidewise}.
\end{proof}

\begin{proof}[Proof of Theorem \ref{uckaboo1}]
Enlarging the universe if necessary, we may suppose that $\calO^{\otimes}$ and
$\calC^{\otimes}$ are small. The coCartesian fibration $\calC^{\otimes} \rightarrow \calO^{\otimes}$
is classified by a map of $\infty$-operads $\chi: \calO^{\otimes} \rightarrow \Cat_{\infty}^{\times}$.
Let $\chi'$ denote the composition of $\chi$ with the $\infty$-operad map
$\Ind^{\otimes}: \Cat_{\infty} \rightarrow \widehat{\Cat_{\infty}}$ of ***, and let
${\calC'}^{\otimes} \rightarrow \calO^{\otimes}$ be the $\calO$-monoidal $\infty$-category
classified by $\chi'$. Then we have a fully faithful functor
$\calC^{\otimes} \rightarrow {\calC'}^{\otimes}$ which induces a homotopy pullback diagram
$$ \xymatrix{ \Alg^{\aug}_{\calO}(\calC) \ar[d] \ar[r] & \Alg^{\aug}_{\calO}( \calC' ) \ar[d] \\
\Fun_{\calO}(\calO, \calC) \ar[r] & \Fun_{\calO}(\calO, \calC') }$$
where the horizontal maps are fully faithful inclusions. Passing to stable envelopes, we get a homotopy pullback diagram
$$ \xymatrix{ \Stab( \Alg^{\aug}_{\calO}( \calC) ) \ar[d] \ar[r] & \Stab( \Alg^{\aug}_{\calO}( \calC' ) \ar[d] \\
\Stab( \Fun_{\calO}(\calO, \calC)) \ar[r] & \Stab(\Fun_{\calO}(\calO, \calC')). } $$
It will therefore suffice to show that the right vertical map is an equivalence. In other words, we may replace $\calC^{\otimes}$ by ${\calC'}^{\otimes}$ and thereby reduce to the case where $\calC^{\otimes} \rightarrow \calO^{\otimes}$ is a presentable stable $\calO$-monoidal $\infty$-category.

The forgetful functor $\Alg_{\calO}(\calC) \rightarrow \Fun_{\calO}(\calC)$ is conservative (Corollary \symmetricref{jumunj22}) and preserves geometric realizations of simplicial objects (Proposition\symmetricref{fillfemme}). It follows from Lemma \ref{curpse2} that $G$ has the same properties. Using Theorem \monoidref{barbeq}, we deduce that $G$ exhibits $\Alg^{\aug}_{\calO}(\calC)$ as monadic over $\Fun_{\calO}(\calO,\calC)$ (see \S \bicatref{linadj}). The desired result now follows by combining Proposition \ref{sillfun} with Corollary \bicatref{hurpek}.
\end{proof}

In the special case where the $\infty$-operad $\calO^{\otimes}$ is coherent, we can use
Theorem \ref{uckaboo1} to describe other fibers of the tangent bundle of $\Alg_{\calO}(\calC)$:

\begin{theorem}\label{subbe}
Let $\calO^{\otimes}$ be a coherent $\infty$-operad, let
$\calC^{\otimes} \rightarrow \calO^{\otimes}$ be a stable $\calO$-monoidal $\infty$-category, and
let $A \in \Alg_{\calO}(\calC)$ be a $\calO$-algebra object of $\calC$. Then the
stabilization $\Stab( \Alg_{\calO}(\calC)_{/A})$ is canonically equivalent to
the $\infty$-category $\Fun_{\calO}(\calO, \Mod^{\calO}_{A}(\calC) )$. 
\end{theorem}

\begin{corollary}\label{subber}
Let $\EInfty$ denote the $\infty$-category of $E_{\infty}$-rings, and let
$A \in \EInfty$. Then the $\infty$-category $\Stab( (\EInfty)_{/A})$ is equivalent
to the $\infty$-category of $A$-module spectra.
\end{corollary}

\begin{remark}\label{suffik}
In the situation of Theorem \ref{subbe}, we have an evident functor
$$\Omega^{\infty}: \Fun_{\calO}(\calO, \Mod^{\calO}_{A}(\calC)) \simeq \Stab( \Alg_{\calO}(\calC)_{/A}) \rightarrow
\Alg_{\calO}(\calC)_{/A}.$$
This functor associates to each $M \in \Fun_{\calO}(\calO, \Mod^{\calO}_{A}(\calC))$ a commutative algebra object which we will denote by $A \oplus M$. The proof of
Theorem \ref{subbe} will justify this notation; that is, we will see that when regarded
as an object of $\Fun_{\calO}(\calO, \calC)$, $A \oplus M$ can be canonically identified with the 
coproduct of $A$ and $M$.
\end{remark}

\begin{proof}
The desired equivalence is given by the composition
\begin{eqnarray*} \Stab( \Alg_{\calO}(\calC)_{/A} ) & \simeq & \Stab( (\Alg_{\calO}(\calC)_{/A})_{A/}) \\
& \simeq & \Stab(( \Alg_{\calO}(\calC)_{A/} )_{/A}) \\
& \stackrel{\phi}{\simeq} & \Stab( \Alg_{\calO}( \Mod^{\calO}_{A}(\calC) )_{/A} \\
& \simeq & \Stab( \Alg^{\aug}_{\calO}( \Mod^{\calO}_{A}(\calC) )) \\
& \stackrel{\phi'}{\simeq} & \Fun_{\calO}(\calO, \Mod^{\calO}_{A}(\calC)).\\
\end{eqnarray*}
Here $\phi$ is the equivalence of Corollary \symmetricref{skoke}, 
$\phi'$ is given by Proposition \ref{uckaboo1}.
\end{proof}


\begin{remark}\label{toaster}
Let $\calC^{\otimes}$ be a stable symmetric monoidal $\infty$-category (such that the tensor
product on $\calC$ is exact in each variable) let $A$ be a commutative algebra object of $\calC$, let $M$ be an $A$-module, and let $A \oplus M$ denote the image
of $M$ under the composition
$$\Mod_A(\calC) \simeq \Stab( \CAlg(\calC)_{/A}) \stackrel{\Omega^{\infty}}{\rightarrow}
\CAlg(\calC)_{/A}.$$
We claim that the algebra structure on $A \oplus M$ is ``square-zero'' in the homotopy category
$\h{\calC}$. In other words:
\begin{itemize}
\item[$(1)$] The unit map $1_{\calC} \rightarrow A \oplus M$ is homotopic to the composition of
$1_{\calC} \rightarrow A$ with the inclusion $A \rightarrow A \oplus M$.
\item[$(2)$] The multiplication
$$ m: (A \otimes A) \oplus (A \otimes M) \oplus (M \otimes A) \oplus (M \otimes M) \simeq (A \oplus M) \otimes (A \oplus M) \rightarrow A \oplus M$$
is given as follows:
\begin{itemize}
\item[$(i)$] On the summand $A \otimes A$, the map $m$ is homotopic to the composition
of the multiplication map $A \otimes A \rightarrow A$ with the inclusion $A \rightarrow A \oplus M$.
\item[$(ii)$] On the summands $A \otimes M$ and $M \otimes A$, the map $m$ is given
by composing the action of $A$ on $M$ with the inclusion $M \rightarrow A \oplus M$.
\item[$(iii)$] On the summand $M \otimes M$, the map $m$ is nullhomotopic.
\end{itemize}
\end{itemize}
Only assertion $(iii)$ requires proof. For this, we will invoke the fact that the commutative algebra structure on $A \oplus M$ depends functorially on $M$. Consequently, for every $A$-module $N$ we obtain a map $\psi_N: N \otimes N \rightarrow N$, which we must show to be nullhomotopic.
Let $M'$ and $M''$ be copies of the $A$-module $M$, which we will distinguish notationally for clarity, and let $f: M' \oplus M'' \rightarrow M$ denote the ``fold'' map which is the identity on each factor.
Invoking the functoriality of $\psi$, we deduce that the map $\psi_M: M \otimes M \rightarrow M$
factors as a composition
$$ M \otimes M = M' \otimes M'' \rightarrow (M' \oplus M'') \otimes (M' \oplus M'')
\stackrel{\psi_{M' \oplus M''}}{\rightarrow} M' \oplus M'' \stackrel{f}{\rightarrow} M.$$
Consequently, to prove that $\psi_M$ is nullhomotopic, it will suffice to show that
$\phi = \psi_{M' \oplus M''}| (M' \otimes M'')$ is nullhomotopic. Let
$\pi_{M'}: M' \oplus M'' \rightarrow M'$ and $\pi_{M''}: M \oplus M'' \rightarrow M''$ denote the projections onto the first and second factor, respectively. To prove that $\psi_{M' \oplus M''}$
is nullhomotopic, it suffices to show that $\pi_{M'} \circ \phi$ and $\pi_{M''} \circ \phi$ are nullhomotopic. 
We now invoke functoriality once more to deduce that $\pi_{M'} \circ \phi$ is homotopic to the composition
$$ M' \otimes M'' \stackrel{ (\id, 0)}{\rightarrow} M' \otimes M' \stackrel{ \psi_{M'}}{\rightarrow} M'.$$
This composition is nullhomotopic, since the first map factors through $M' \otimes 0 \simeq 0$.
The same argument shows that $\pi_{M''} \circ \phi$ is nullhomotopic, as desired.
\end{remark}

\begin{remark}
Let $A$ be an $E_{\infty}$-ring, let $M$ be an $A$-module, and let $A \oplus M$ denote the corresponding square-zero extension. As a graded abelian group, we may identify
$\pi_{\ast}( A \oplus M)$ with the direct sum $(\pi_{\ast} A) \oplus (\pi_{\ast} M)$. It follows
from Remark \ref{toaster} that the multiplication on $\pi_{\ast}( A \oplus M)$ is given on homogeneous elements by the formula
$$ (a, m)(a', m') = (aa', am' + (-1)^{\deg(a') \deg(m)} a'm).$$
In particular, if $A$ is an ordinary commutative ring (viewed as a discrete $E_{\infty}$-ring) and $M$ is an ordinary $A$-module, then we can identify the discrete $E_{\infty}$-ring $A \oplus M$ with the classical square-zero extension discussed in the introduction to this section.
\end{remark}

We now prove a ``global'' version of Theorem \ref{subbe}:

\begin{theorem}\label{scummer2}
Let $\calO^{\otimes}$ be a coherent $\infty$-operad, and let
$\calC^{\otimes} \rightarrow \calO^{\otimes}$ be a presentable stable
$\calO$-monoidal $\infty$-category. Then there is a canonical equivalence
$$ \phi: T_{\Alg_{\calO}(\calC) } \rightarrow \Alg_{\calO}(\calC)
\times_{ \Fun(\calO, \Alg_{\calO}(\calC)} \Fun_{\calO}( \calO, \Mod^{\calO}(\calC))$$
of presentable fibrations over $\Alg_{\calO}(\calC)$.
\end{theorem}

In other words, we may view $T_{ \Alg_{\calO}(\calC)}$ as the $\infty$-category whose objects are pairs
$(A,M)$, where $A$ is a $\calO$-algebra object of $\calC$ and $M$ is an $A$-module.
The idea of the proof is simple: we will define $\phi$ using a relative version of the augmentation ideal functor defined above. We will then show that $\phi$ is a map of Cartesian fibrations, so that the condition that $\phi$ be an equivalence can be checked fibrewise. We are then reduced to the situation of Theorem \ref{subbe}. 

\begin{proof}
We will denote objects of $\calM = \Alg_{\calO}(\calC)
\times_{ \Fun(\calO, \Alg_{\calO}(\calC)} \Fun_{\calO}( \calO, \Mod^{\calO}(\calC))$
by pairs $(A, M)$, where $A \in \Alg_{\calO}(\calC)$ and $M \in \Fun_{\calO}( \calO, \Mod^{\calO}(\calC)$
is a module over $A$.

Let $\calE = \Fun( \Delta^1 \times \Delta^1, \Alg_{\calO}(\calC) ) \times_{ \Fun( \Delta^{2}, \Alg_{\calO}(\calC) ) } \Alg_{\calO}(\calC)$ denote the $\infty$-category of diagrams of the form 
$$ \xymatrix{ A \ar[d]^{\id} \ar[r] \ar[dr]^{\id} & B \ar[d] \\
A \ar[r]^{\id} & A,}$$
of $\calO$-algebra objects of $\calC$. The canonical map
$\Alg_{\calO}(\calC) \rightarrow \Alg_{\calO}( \Mod^{\calO}(\calC))$ determines
a section $s$ of the projection
$$ p:  \calX \rightarrow
\Alg_{\calO}(\calC),$$
which we can think of informally as assigning to an algebra $A$ the pair $(A,A)$
where we regard $A$ as a module over itself.

Let $\calD$ denote the fiber product
$$ \Fun( \Delta^1 \times \Delta^1, \calM ) \times_{ \Fun( \Delta^1 \times \{1\}, \calM) }
\Fun( \Delta^1 \times \{1\}, \Alg_{\calO}(\calC) ),$$
so that we can identify objects of $\calD$ with commutative squares
$$ \xymatrix{ (A, M) \ar[d] \ar[r] & (B, B) \ar[d] \\
(A', M') \ar[r] & (B', B') }$$
in the $\infty$-category $\calM$. Let $\overline{\calE}$ denote the full subcategory of
$\calE \times_{ \Fun( \Delta^1 \times \Delta^1, \Alg_{\calO}(\calC) } \calD$
spanned by those squares
$$ \xymatrix{ (A, M) \ar[d] \ar[r] & (B,B) \ar[d] \\
(A, M') \ar[r] & (A,A) }$$
which are $p$-limit diagrams, and such that $M'$ is a zero object of 
$\Fun_{\calO}(\calO, \Mod_{A}^{\calO}(\calC))$. Invoking Proposition \toposref{lklk} twice (and Theorem \symmetricref{maincand}), we deduce that the projection map
$\overline{\calE} \rightarrow \calE$ is a trivial Kan fibration.
Let $r: \calE \rightarrow \overline{\calE}$ be a section of this projection, and let
$r': \overline{\calE} \rightarrow \calX$ be given by evaluation in the upper left hand corner.
Let $\psi$ denote the composition
$$ \psi: \calE \stackrel{r}{\rightarrow} \overline{\calE} \stackrel{r'}{\rightarrow}
\calM,$$
so that $\psi$ carries a diagram
$$ \xymatrix{ A \ar[d]^{\id} \ar[r] \ar[dr]^{\id} & B \ar[d]^{f} \\
A \ar[r]^{\id} & A,}$$
to the augmentation ideal $\ker(f)$, regarded as an $A$-module. 

We observe that
the restriction map
$ \calE \rightarrow \Fun( \Delta^1 \times \{1\}, \Alg_{\calO}(\calC) )$ can be regarded
as a pointed envelope of the presentable fibration
$$ \Fun( \Delta^1 \times \{1\}, \Alg_{\calO}(\calC) )
\rightarrow \Fun( \{1\} \times \{1\}, \Alg_{\calO}(\calC) ) \simeq \Alg_{\calO}(\calC).$$
Let $\Omega^{\infty}_{\ast}: T_{\Alg_{\calO}(\calC)} \rightarrow \calE$ exhibit $T_{\Alg_{\calO}(\calC)}$ as a tangent bundle to $\Alg_{\calO}(\calC)$. Let $\phi$ denote the composition
$$ T_{ \Alg_{\calO}(\calC)} \stackrel{\Omega^{\infty}}{\rightarrow} \calE \stackrel{\psi}{\rightarrow}
\calM.$$
To complete the proof, it will suffice to show that $\phi$ is an equivalence of $\infty$-categories.

By construction, we have a commutative diagram
$$ \xymatrix{ T_{ \Alg_{\calO}(\calC) } \ar[dr]^{q} \ar[r]^{\Omega^{\infty}_{\ast}} & \calE \ar[d]^{q'} \ar[r]^-{\phi_0} &  \calM \ar[dl]^{q''} \\
& \Alg_{\calO}(\calC), & }$$
with $\phi = \phi_0 \circ \Omega^{\infty}_{\ast}$, where $q$, $q'$, and $q''$ are presentable fibrations.
Since $\Omega^{\infty}_{\ast}$ is a right adjoint relative to $\Alg_{\calO}(\calC)$, it carries
$q$-Cartesian morphisms to $q'$-Cartesian morphisms. We observe that
$\phi_0$ carries $q'$-Cartesian morphisms to $q''$-Cartesian morphisms; in concrete terms, this
merely translates into the observation that every pullback diagram
$$ \xymatrix{ A \ar[r] \ar[d]^{f} & B \ar[d]^{f'} \\
A' \ar[r] & B' }$$
in $\Alg_{\calO}(\calC)$ is also a pullback diagram in $\Fun_{\calO}(\calO,\calC)$ (Corollary \symmetricref{slimycamp2}), and therefore induces an equivalence $\ker(f) \simeq \ker(f')$ in $\calM$. It follows that $\phi$ carries $q$-Cartesian morphisms to $q''$-Cartesian morphisms. 

We now invoke Corollary \toposref{usefir}: the map $\phi$ is an equivalence of $\infty$-categories if and only if, for every commutative algebra object $A \in \Alg_{\calO}(\calC)$, the induced map
$$ \phi_A: \Stab( \Alg_{\calO}(\calC)_{/A} ) \rightarrow \Fun_{\calO}(\calO, \Mod^{\calO}_{A}(\calC))$$
is an equivalence of $\infty$-categories. We now observe that $\phi_A$ can be identified with
the augmentation ideal functor which appears in the proof of Theorem \ref{subbe}, and therefore an equivalence as required.
\end{proof}

\section{Cotangent Complexes of $E_{\infty}$-Rings}\label{sec3}

In \S \ref{gentheory}, we studied the general theory of cotangent complexes. For every presentable $\infty$-category $\calC$, we defined the tangent bundle $T_{\calC}$ and a relative cotangent complex functor
$$ \Fun( \Delta^1, \calC) \rightarrow T_{\calC}$$ 
$$ (f: A \rightarrow B) \mapsto L_{B/A} \in \Stab( \calC^{/B} ).$$
We now wish to specialize to the situation where $\calC$ is the $\infty$-category $\EInfty$ of
$E_{\infty}$-rings. In this case, Theorem \ref{scummer2} allows us to identify
the tangent bundle $T_{\calC}$ with the $\infty$-category of pairs $(A, M)$, where
$A$ is an $E_{\infty}$-ring and $M$ is an $A$-module. We will henceforth use this identification to view the relative cotangent complex $L_{B/A}$ as taking its value in the $\infty$-category
$\Mod_{B}$ of $B$-module spectra.

Our goal in this section is to prove some results about the cotangent complexes of $E_{\infty}$-rings which are more quantitative in nature. We will begin in \S \ref{estawha} by studying the connectivity
properties of the relative cotangent complex functor $L$. For example, we will show that if
$f: A \rightarrow B$ is an $n$-connective morphism between connective $E_{\infty}$-rings, then
the induced map $L_{A} \rightarrow L_{B}$ is $n$-connective (Corollary \ref{swepu}). 
This is a simple consequence of our main result, Theorem \ref{tulbas}, which is considerably more precise.

In \S \ref{fictot}, we will study finiteness properties of the relative cotangent complex $L_{B/A}$ associated to a morphism $f: A \rightarrow B$ between connective $E_{\infty}$-rings. It is not difficult to show that finiteness properties of $f$ are inherited by the relative cotangent complex $L_{B/A}$. For example, if $f$ is of finite presentation, then the relative cotangent complex $L_{B/A}$ is a perfect $B$-module. Somewhat surprisingly, the converse holds under some mild additional assumptions
(Theorem \ref{sucker}). 

The final goal of this section is to introduce the definition of an {\it \etalenospace} map between $E_{\infty}$-rings. 
A morphism $f: A \rightarrow B$ is said to be \etale if $f$ is flat, and the induced map
$\pi_0 A \rightarrow \pi_0 B$ is an \etale map of ordinary commutative rings. Our main result concerning \etale morphisms is Proposition \ref{etrel}, which asserts that the relative cotangent complex
$L_{B/A}$ vanishes whenever $f: A \rightarrow B$ is \etale. 

\subsection{Connectivity Estimates}\label{estawha}

Let $f: A \rightarrow B$ be a morphism of $E_{\infty}$-rings. According to Remark \ref{blurn}, the 
relative cotangent complex $L_{B/A}$ vanishes whenever $f$ is an equivalence. We may therefore
regard $L_{B/A}$ as a measure of a failure of $f$ to be an equivalence. A more direct measure is the 
cokernel $\coker(f)$ of the map $f$. Our goal in this section is to prove Theorem \ref{tulbas}, which asserts that these invariants are related: namely, there is a canonical map
$$ \alpha: \coker(f) \rightarrow L_{B/A}.$$
Moreover, this map has good connectivity properties if $f$ does (we will formulate this statement more precisely below).

In order to prove Theorem \ref{tulbas}, we need a mechanism for computing the cotangent complex
$L_{B/A}$ in certain examples. We therefore begin with a simple calculation.

\begin{proposition}\label{puffle}
Let $M$ be a spectrum, and let $A = \Sym^{\ast} M$ denote the free $E_{\infty}$-ring generated by $M$. 
Then there is a canonical equivalence $L_A \simeq M \otimes A$
in the $\infty$-category of $A$-modules. 
\end{proposition}

\begin{proof}
For every $A$-module $N$, we have a chain of homotopy equivalences
$$ \bHom_{ \Mod_A}( M \otimes A, N)
\simeq \bHom_{ \Spectra}( M, N)
\simeq \bHom_{ \Spectra_{/A} }( M, A \oplus N)
\simeq \bHom_{ (\EInfty)_{/A} }( A, A \oplus N)
\simeq \bHom_{ \Mod_{A} }( L_A, N).$$
It follows that $M \otimes A$ and $L_A$ corepresent the same functor in the homotopy category
$\h{ \Mod_{A} }$, and are therefore equivalent.
\end{proof}

According to Corollary \symmetricref{spaltwell} and Remark \symmetricref{tappus}, for every
spectrum $M$ we have a canonical equivalence $\Sym^{\ast} M \simeq \oplus_{i \geq 0} \Sym^{i} M$, where $\Sym^{i} M$ is obtained from the $i$th tensor power of $M$ by extracting the (homotopy-theoretic) coinvariants of the action of the symmetric group $\Sigma_{i}$. Our connectivity
estimates for the cotangent complex all hinge on the following basic observation:

\begin{remark}\label{spukkle}
Let $M$ and $N$ be spectra. Assume that $M$ is $m$-connective and $N$ is $n$-connective. Then
the tensor product $M \otimes N$ is $(m+n)$-connective. Iterating this observation, we deduce that
every tensor power $M^{\otimes k}$ is $(mk)$-connective. Using the stability of connective spectra under colimits, we conclude that any symmetric power $\Sym^{k}(M)$ is $(km)$-connective.
\end{remark}

\begin{lemma}\label{cake}
Let $\Sym^{\ast}: \Spectra \rightarrow \EInfty$ denote a left adjoint to the forgetful functor
$\EInfty \rightarrow \Spectra$ $($see \S \symmetricref{comm33}$)$. 
Let $f: A \rightarrow B$ be a map of connective $E_{\infty}$-rings, and assume that
$f$ is $n$-connective for some $n \geq -1$. Then
there exists an $n$-connective spectrum $M$ and a commutative diagram of
$E_{\infty}$-rings
$$ \xymatrix{ \Sym^{\ast} M \ar[r]^{\epsilon} \ar[d] & \Sphere \ar[d] \\
A \ar[r] \ar[dr]^{f} & A' \ar[d]^{f'} \\
& B, }$$
where the upper square is a pushout, the $E_{\infty}$-ring $A'$ is connective,
the map $f'$ is $(n+1)$-connective, and $\epsilon$ is adjoint to the zero map
$M \rightarrow \Sphere$ in the $\infty$-category of spectra. Here $\Sphere$ denotes the sphere spectrum.
\end{lemma}

\begin{proof}
We will abuse notation by not distinguishing between the $E_{\infty}$-rings
$A$ and $B$ and their underlying spectra. 
Let $M = \ker(f)$, so that we have a pushout diagram of spectra
$$ \xymatrix{ M \ar[d] \ar[r] & 0 \ar[d] \\
A \ar[r]^{f} & B. }$$
Invoking the universal property of $\Sym^{\ast}$, we obtain a commutative diagram
$$ \xymatrix{ \Sym^{\ast} M \ar[r] \ar[d] & \Sym^{\ast} 0 \ar[d] \\
\Sym^{\ast} A \ar[r] \ar[d] & \Sym^{\ast} B \ar[d]  \\
A \ar[r] & B }$$
in the $\infty$-category of $E_{\infty}$-rings, where the upper square is a pushout. We observe that $\Sym^{\ast} 0$ is equivalent to the sphere spectrum $\Sphere$. Let $A'$ denote the tensor product $A \otimes_{ \Sym^{\ast} M} \Sphere$ so that we obtain a commutative diagram
$$ \xymatrix{ \Sym^{\ast} M \ar[r]^{\epsilon} \ar[d] & \Sphere \ar[d] \\
A \ar[r] \ar[dr]^{f} & A' \ar[d]^{f'} \\
& B }$$
as above. Since $A'$ can also be identified with the tensor product
$$ A \otimes_{ \Sym^{\ast} A} \Sym^{\ast} B,$$
we conclude that $A'$ is connective.
The only nontrivial point is to verify that $\ker(f')$ is $(n+1)$-connective. Suppose first that
$n = -1$; in this case, we wish to show that $f'$ induces an epimorphism
$\pi_0 A' \rightarrow \pi_0 B$. To prove this, we observe that the counit map
$$ f'': \Sym^{\ast} B \rightarrow B$$
factors through $f'$. The map $f''$ induces an epimorphism on all homotopy groups, because
the underlying map of spectra admits a section.

We now treat the generic case $n \geq 0$.
Let
us define $I$ denote the kernel of the projection map $\Sym^{\ast} M \rightarrow \Sphere$, so that
we have a map of distinguished triangles
$$ \xymatrix{ A \otimes_{ \Sym^{\ast} M } I \ar[d]^{g} \ar[r] & A \ar[d]^{=} \ar[r] & A' \ar[d]^{f'} \ar[r] & 
A \otimes_{ \Sym^{\ast} M} I[1] \ar[d] \\
M \ar[r] & A \ar[r] & B \ar[r] & M[1] }$$
in the homotopy category of spectra.
Consequently, we obtain an equivalence of spectra $\ker(f') \simeq \ker(g)[1]$, so it will suffice
to show that $g$ is $n$-connective. Using Remark \symmetricref{tappus}, we can identify
$I$ with the coproduct $\oplus_{i > 0} \Sym^{i}(M)$.
The map $g$ admits a section, given by the composition
$$ M \simeq \Sym^{1}(M) \rightarrow I \rightarrow A \otimes_{ \Sym^{\ast} M} I.$$
We may therefore identify $\ker(g)$ with a summand of the tensor product
$A \otimes_{ \Sym^{\ast} M} I$. It will now suffice to show that this tensor product is
$n$-connective. Since $A$ and $\Sym^{\ast} M$ are connective, it will suffice to show that
$I$ is $n$-connective. This follows immediately from Remark \ref{spukkle}.
\end{proof}

\begin{theorem}\label{tulbas}
Let $f: A \rightarrow B$ be a morphism between $E_{\infty}$-rings, and consider the associated
diagram
$$ \xymatrix{ L_{A} \ar[r]^{\eta_0} \ar[d] & 0 \ar[d] \\
L_{B} \ar[r]^{\eta} & L_{B/A}. }$$
This diagram induces a map of derivations
$$(d_0: A \rightarrow 0) \rightarrow (d: B \rightarrow L_{B/A}).$$
Applying the functor $\Phi$ of Notation \ref{stubble}, we obtain a commutative diagram
$$ \xymatrix{ A \ar@{=}[r] \ar[d] \ar[dr]^{f} & A \ar[d] \\
B^{\eta} \ar[r]^{g} & B }$$
of $E_{\infty}$-rings (here we implicitly identify $A^{\eta_0}$ with $A$; see Example \ref{invu}).
This commutative diagram induces a map $\alpha_{f}: \ker(f) \rightarrow \ker{g} \simeq L_{B/A}[-1]$
in the $\infty$-category of $A$-modules. Let $\alpha'_{f}: \ker(f) \otimes_{A} B \rightarrow L_{B/A}[-1]$ be the adjoint morphism. Suppose that $A$ and $B$ are connective and that $f$ is $n$-connective, for some $n \geq -1$. Then $\alpha'_{f}$ is $(2n+1)$-connective.
\end{theorem}

\begin{proof}
Let us say that a morphism $f: A \rightarrow B$ of $E_{\infty}$-rings is {\it $n$-good} if the
kernel $\alpha'_{f}$ is $(2n+1)$-connective. We make the following observations:
\begin{itemize}
\item[$(a)$] Suppose given a commutative triangle
$$ \xymatrix{ & B \ar[dr]^{g} & \\
A \ar[ur]^{f} \ar[rr]^{h} & & C }$$
of connective $E_{\infty}$-rings. If $f$ and $g$ are $n$-good, then $h$ is $n$-good.
This follows from the existence of a commutative diagram of exact triangles
$$ \xymatrix{ 
\ker(f) \otimes_{A} C \ar[d]^{\alpha'_{f}} \ar[r] & \ker(h) \otimes_{A} C \ar[d]^{\alpha'_{h}}
\ar[r] & \ker(g) \otimes_{B} C \ar[d]^{ \alpha'_{g} } \\
L_{B/A}[-1] \otimes_{B} C \ar[r] & L_{C/A}[-1] \ar[r] & L_{C/B}[-1] }$$
of $C$-modules. 

\item[$(b)$] Suppose given a pushout diagram
$$ \xymatrix{ A \ar[r]^{f} \ar[d] & B \ar[d] \\
A' \ar[r]^{f'} & B' }$$
of $E_{\infty}$-rings, where $B$ and $B'$ are connective. If $f$ is $n$-good, then so is $f'$. This
follows immediately from the equivalence
$\ker( \alpha'_{f'} ) \simeq B' \otimes_{B} \ker( \alpha'_{f} )$.

\item[$(c)$] The collection of $n$-good morphisms is closed under the formation of filtered colimits. This follows from the fact that the functor $f \mapsto \ker(\alpha'_{f})$ preserves filtered colimits, and the observation that a filtered colimit of $(2n+1)$-connective spectra is again $(2n+1)$-connective.

\item[$(d)$] Let $f: A \rightarrow B$ be an arbitrary morphism of $E_{\infty}$-rings. Then
the source $\ker(f) \otimes_{A} B$ of the morphism $\alpha'_{f}$ can be identified with the kernel of the induced map $B \rightarrow B \otimes_{A} B$, which identifies $B$ with one of the tensor factors.

\item[$(e)$] Assume that $M$ is an $n$-connective spectrum, and let
$f: \Sym^{\ast} M \rightarrow \Sphere$ be a map of $E_{\infty}$-rings which is adjoint to the
zero map $M \rightarrow \Sphere$ in the $\infty$-category of spectra. Then $f$ is
$n$-good. To prove this, we will explicitly compute both the source and target
of $\alpha'_{f}$. 

Using Corollary \ref{longtooth} we obtain a distinguished triangle
$$ L_{ \Sym^{\ast} M} \otimes_{ \Sym^{\ast} M} \Sphere
\rightarrow L_{\Sphere} \rightarrow L_{ \Sphere / \Sym^{\ast} M }
\rightarrow L_{ \Sym^{\ast} M } \otimes_{ \Sym^{\ast} M} \Sphere [1]$$
in the homotopy category of spectra. In view of Proposition \ref{puffle}, we may rewrite this distinguished triangle as
$$ M \rightarrow 0 \rightarrow L_{ \Sphere/ \Sym^{\ast} M} \rightarrow M[1], $$
so that the target $L_{ \Sphere / \Sym^{\ast} M}[-1]$ of the morphism $\alpha'_{f}$ is
canonically equivalent to $M$. 

We next observe that the pushout diagram of spectra
$$ \xymatrix{ M \ar[r] \ar[d] & 0 \ar[d] \\
0 \ar[r] & M[1] }$$
induces an equivalence of $E_{\infty}$-rings $\Sphere \otimes_{ \Sym^{\ast} M } \Sphere \simeq \Sym^{\ast} M[1]$.
Invoking $(d)$, we deduce that the source of the map $\alpha'_{f}$ can be identified with
the kernel of the unit map $\Sphere \rightarrow \Sym^{\ast} M[1]$. Using Remark \symmetricref{tappus}, we can identify this kernel with the direct sum $\oplus_{i > 0} \Sym^{i}( M[1] )[-1]$. 

We now observe that the composition
$$ M \simeq \Sym^{1}( M[1] )[-1]
\rightarrow \oplus_{ i > 0} \Sym^{i}( M[1] )[-1]
\stackrel{\alpha'_{f} }{\rightarrow} M$$
is homotopic to the identity. Consequently, the kernel of $\alpha'_{f}$ can be identified with the direct sum $\oplus_{i > 1} \Sym^{i}( M[1])[-1]$. To complete the proof that $\alpha'_{f}$ is
$(2n+1)$-connective, it will suffice to show that each symmetric power
$\Sym^{i}( M[1] )$ is $(2n+2)$-connective, which follows immediately from Remark
\ref{spukkle}. 
\end{itemize}

We are now ready to proceed with the proof of Theorem \ref{tulbas}. Let $f: A \rightarrow B$
be an $n$-connective map of connective $E_{\infty}$-rings; we wish to show that $f$ is $n$-good. Applying Lemma \ref{cake} repeatedly, we deduce the existence of a sequence of objects
$$ A_{n} \rightarrow A_{n+1} \rightarrow A_{n+2} \rightarrow \ldots$$
in $(\EInfty)_{/B}$, with the following properties:
\begin{itemize}
\item[$(i)$] The object $A_{n}$ can be identified with the original morphism $f: A \rightarrow B$.
\item[$(ii)$] For all $m \geq n$, let us identify $A_{m}$ with a morphsim of $E_{\infty}$-rings $f_{m}: A(m) \rightarrow B$. Then $f_m$ is $m$-connective, and $A(m)$ is connective.
\item[$(iii)$] For each $m \geq n$, there exists an $m$-connective spectrum $M$ and
a pushout diagram
$$ \xymatrix{ \Sym^{\ast} M \ar[r]^{\epsilon_{m}} \ar[d] & \Sphere \ar[d] \\
A(m) \ar[r]^{g_{m,m+1}} & A(m+1), }$$
where $g_{j,k}$ denotes the morphism of $E_{\infty}$-rings underlying the map
from $A_j$ to $A_k$ in our direct system and $\epsilon$ is adjoint to the zero map
$M \rightarrow \Sphere$ in the $\infty$-category of spectra.
\end{itemize}

This direct system induces a map of $E_{\infty}$-rings $f_{\infty}: \colim \{ A(n) \} \rightarrow B$.
We observe that $\ker(f_{\infty}) \simeq \colim \{ \ker(f_m) \}$. It follows that $f_{\infty}$ is $(k+1)$-connective for every integer $k$, so that $f_{\infty}$ is an equivalence. This implies that
$f$ can be identified with the direct limit of the sequence of morphisms
$\{ g_{n, m} \}_{m \geq n}$. In view of $(c)$, it will suffice to show that each
$g_{n,m}$ is $n$-good. Applying $(a)$ repeatedly, we can reduce to showing that
each of the morphisms $g_{m, m+1}$ is $n$-good. Using $(b)$, we are reduced to showing
that each of the morphisms $\epsilon_{m}$ is $n$-good, which follows immediately from $(e)$.
\end{proof}

\begin{corollary}\label{spiffle}
Let $f: A \rightarrow B$ be a map of connective $E_{\infty}$-rings. Assume that
$f$ is $n$-connective, for $n \geq -1$. Then the relative cotangent complex
$L_{B/A}$ is $(n+1)$-connective. The converse holds provided that
$f$ induces an isomorphism $\pi_0 A \rightarrow \pi_0 B$.
\end{corollary}

\begin{proof}
Let $\alpha'_{f}: \ker(f) \otimes_{A} B \rightarrow L_{B/A}[-1]$ be the map described in Theorem \ref{tulbas}, so that we have a distinguished triangle of $B$-modules:
$$ \ker(f) \otimes_{A} B \rightarrow L_{B/A}[-1] \rightarrow \coker( \alpha'_{f} )
\rightarrow \ker(f)[1] \otimes_{A} B.$$
To prove that $L_{B/A}$ is $(n+1)$-connective, it suffices to show that
$\ker(f) \otimes_{A} B$ and $\coker(\alpha'_{f})$ are $n$-connective. The first assertion is obvious, and the second follows from Theorem \ref{tulbas} since $2n+1 \geq n$. 

To prove the converse, let us suppose that $f$ is {\em not} $n$-connective. 
We wish to show that $L_{B/A}$ is not $(n+1)$-connective. By assumption, $f$ induces an
isomorphism $\pi_0 A \rightarrow \pi_0 B$, so that $\ker(f)$ is connective; thus $n \geq 0$. 
Without loss of generality, we may suppose that $n$ is chosen as small as possible, so that $f$ is $(n-1)$-connective. Applying Theorem \ref{tulbas}, we conclude that $\alpha'_{f}$ is
$(2n-1)$-connective. Our assumption that $f$ induces an isomorphism $\pi_0 A \rightarrow \pi_0 B$
guarantees that $n > 0$, so that $\alpha'_{f}$ is $n$-connective.
Using the long exact sequence 
$$ \pi_{n} \coker( \alpha'_{f} ) \rightarrow \pi_{n-1}( \ker(f) \otimes_{A} B)
\rightarrow \pi_{n-1}(L_{B/A}[-1]) \rightarrow \pi_{n-1} \coker( \alpha'_{f} ),$$
we deduce that $\pi_{n} L_{B/A}$ is isomorphic to 
$$\pi_{n-1} ( \ker(f) \otimes_{A} B)
\simeq \pi_{n-1} \ker(f) \otimes_{ \pi_0 A} \pi_0 B \simeq \pi_{n-1} \ker(f),$$
so that $L_{B/A}$ is not $(n+1)$-connective.
\end{proof}

\begin{corollary}\label{spat}
Let $A$ be a connective $E_{\infty}$-ring. Then the absolute cotangent complex
$L_{A}$ is connective.
\end{corollary}

\begin{proof}
Apply Corollary \ref{spiffle} to the unit map $\Sphere \rightarrow A$ in the case $n=-1$.
\end{proof}

\begin{corollary}\label{twain}
Let $f: A \rightarrow B$ be a map of connective $E_{\infty}$-rings. 
Then $f$ is an equivalence if and only if the following conditions are satisfied:
\begin{itemize}
\item[$(1)$] The map $f$ induces an isomorphism $\pi_0 A \rightarrow \pi_0 B$.
\item[$(2)$] The relative cotangent complex $L_{B/A}$ vanishes.
\end{itemize}
\end{corollary}

\begin{corollary}\label{swepu}
Let $f: A \rightarrow B$ be a map of connective $E_{\infty}$-rings. Assume that
$f$ is $n$-connective for $n \geq -1$. Then the induced map
$L_{f}: L_{A} \rightarrow L_{B}$ is $n$-connective. In particular, the canonical map
$\pi_0 L_{A} \rightarrow \pi_0 L_{ \pi_0 A}$ is an isomorphism.
\end{corollary}

\begin{proof}
The map $L_{f}$ factors as a composition
$$ L_{A} \stackrel{g}{\rightarrow} B \otimes_{A} L_{A} \stackrel{g'}{\rightarrow} L_{B}.$$
We observe that $\ker(g) \simeq \ker(f) \otimes_{A} L_A$. Since the cotangent complex $L_A$ is connective (Corollary \ref{spat}) and $f$ is $n$-connective, we conclude that $g$ is $n$-connective. It will therefore suffice to show that $g'$ is $n$-connective. The kernel of $g'$ can be identified
with $L_{B/A}[-1]$. Let $\alpha: B \otimes_{A} \ker(f) \rightarrow L_{B/A}[-1]$ be as
in Theorem \ref{tulbas}, so we have a distinguished triangle
$$ B \otimes_{A} \ker(f) \rightarrow L_{B/A}[-1] \rightarrow \coker(\alpha) \rightarrow
B \otimes_{A} \ker(f)[1].$$
It therefore suffices to show that $B \otimes_{A} \ker(f)$ and $\coker(\alpha)$ are
$n$-connective. The first assertion follows immediately from the $n$-connectivity
of $\ker(f)$, and the second from Theorem \ref{tulbas} since $2n+2 \geq n$.
\end{proof}

We conclude this section by discussing the connection between the classical theory of \Kahler differentials and the cotangent complexes of $E_{\infty}$-rings. If $R$ is a commutative ring, then the module of (absolute) \Kahler differentials is the free $R$-module
generated by the symbols $\{ dr \}_{r \in R}$, subject to the relations
$$ d( rr') = r dr' + r' dr$$
$$ d( r + r') = dr + dr'.$$
We denote this $R$-module by $\Omega_{R}$. Given a map of commutative rings $\eta: R' \rightarrow R$, we let $\Omega_{R/R'}$ denote the quotient of $\Omega_{R}$ by the submodule generated by
the elements $\{ d \eta(r') \}_{r' \in R'}$.

\begin{remark}\label{summa}
Let $\eta: R' \rightarrow R$ be a homomorphism of commutative rings. Then we have a canonical short
exact sequence
$$ \Omega_{R'} \otimes_{R'} R \rightarrow \Omega_{R} \rightarrow \Omega_{R/R'} \rightarrow 0$$
in the category of $R$-modules.
\end{remark}

\begin{lemma}\label{postat}
Let $A$ be a discrete $E_{\infty}$-ring. Then
there is a canonical isomorphism
$$ \pi_0 L_A \simeq \Omega_{ \pi_0 A}$$
in the category of $\pi_0 A$-modules.
\end{lemma}

\begin{proof}


It will suffice to show that $\pi_0 L_A$ and
$\Omega_{ \pi_0 A}$ corepresent the same functor on the ordinary category
of modules over the commutative ring $\pi_0 A$. Let $M$ be a $\pi_0 A$-module, which we will identify with the corresponding discrete $A$-module (see Proposition \monoidref{tmod}). We have
homotopy equivalences
$$ \bHom_{ \Mod_A}( \pi_0 L_A, M) \simeq \bHom_{ \Mod_A}( L_A, M) \simeq \bHom_{/A}( A, A \oplus M ).$$
Since $A$ and $M$ are both discrete, the space on the right is homotopy equivalent to the discrete set of ring homomorphisms from $\pi_0 A$ to $\pi_0(A \oplus M)$ which reduce to the identity on
$\pi_0 A$. These are simply derivations from $\pi_0 A$ into $M$ in the classical sense, which
are classified by maps from $\Omega_{\pi_0 A}$ into $M$.
\end{proof}

\begin{proposition}\label{spate}
Let $f: A \rightarrow B$ be a morphism of connective $E_{\infty}$-rings. Then:
\begin{itemize}
\item[$(1)$] The relative cotangent $L_{B/A}$ is connective.
\item[$(2)$] As a $\pi_0 B$-module, $\pi_0 L_{B/A}$ is canonically isomorphic to
the module of relative \Kahler differentials $\Omega_{ \pi_0 B/ \pi_0 A}$.
\end{itemize}
\end{proposition}

\begin{proof}
Assertion $(1)$ follows from Corollary \ref{spat} and the existence of a distinguished triangle
$$ L_{A} \otimes_{A} B \rightarrow L_{B} \rightarrow L_{B/A} \rightarrow (L_{A} \otimes_{A} B)[1].$$
Associated to this triangle we have an exact sequence
$$ \pi_{0} (L_{A} \otimes_{A} B) \stackrel{g}{\rightarrow} \pi_0 L_B \rightarrow
\pi_0 L_{B/A} \rightarrow \pi_{-1} (L_A \otimes_{A} B) \simeq 0$$
of discrete $\pi_0 B$-modules. Consequently, we may identify
$\pi_0 L_{B/A}$ with the cokernel of the map $g$. 

Using Corollary \ref{swepu} and Lemma \ref{postat}, we can identify
$\pi_0 L_{A}$ and $\pi_0 L_B$ with the modules $\Omega_{\pi_0 A}$ and $\Omega_{\pi_0 B}$, respectively. Using Corollary \monoidref{huty}, we can identify $\pi_0 ( L_{A} \otimes_{A} B)$ with the discrete $\pi_0 B$-module
$\Omega_{ \pi_0 A} \otimes_{ \pi_0 A} \pi_0 B$. The desired result now follows from the short exact sequence of Remark \ref{summa}.
\end{proof}

\subsection{Finiteness Properties of the Cotangent Complex}\label{fictot}

Our goal in this section is to prove the following result:

\begin{theorem}\label{sucker}
Let $A$ be a connective $E_{\infty}$-ring, and let 
$B$ be a connective commutative $A$-algebra. Then:
\begin{itemize}
\item[$(1)$] If $B$ is of finite presentation as a commutative $A$-algebra, then
$L_{B/A}$ is perfect as a $B$-module. The converse holds provided that
$\pi_0 B$ is finitely presented as a $\pi_0 A$-algebra.
\item[$(2)$] If $B$ is almost of finite presentation as a commutative $A$-algebra, then
$L_{B/A}$ is almost perfect as a $B$-module. The converse holds provided that $\pi_0 B$ is finitely presented as a $\pi_0 A$-algebra.
\end{itemize}
\end{theorem}

As an immediate consequence, we deduce the following analogue of
Remark \symmetricref{unkid}:

\begin{corollary}
Suppose given a commutative diagram
$$ \xymatrix{ & B \ar[dr] & \\
A \ar[ur] \ar[rr] & & C }$$
of connective $E_{\infty}$-rings. Assume furthermore that $B$ is of almost of finite presentation
over $A$. Then $C$ is almost of finite presentation
over $A$ if and only if $C$ is almost of finite presentation over $B$.
\end{corollary}

To prove Theorem \ref{sucker}, we will need an easy lemma about the structure of projective modules over $A_{\infty}$-rings. First, let us introduce a bit of notation.
For every connective $A_{\infty}$-ring $R$, we let $\Proj(R)$ denote the full subcategory of
$\Mod_{R}$ spanned by the projective (left) $R$-modules.

\begin{lemma}\label{corpus}
Let $f: R \rightarrow R'$ be a map of connective $A_{\infty}$-rings. Suppose that
$f$ induces an isomorphism $\pi_0 R \rightarrow \pi_0 R'$. Then the base change functor
$M \mapsto R' \otimes_{R} M$ induces an equivalence of homotopy categories
$$\phi: \h{ \Proj(R) } \rightarrow \h{ \Proj(R') }.$$
\end{lemma}

\begin{proof}
We first show that the functor $\phi$ is fully faithful. For this, we must show that if
$P$ and $Q$ are projective left $R$-modules, then the canonical map
$$ \Ext^{0}_{R}( P, Q) \rightarrow \Ext^{0}_{R'}( R' \otimes_{R} P, R' \otimes_{R} Q)$$
is bijective. Without loss of generality, we may suppose that $P$ is free. In this case,
the left hand side can be identified with a product of copies of
$\pi_0 Q$, while the right hand side can be identified with a product of copies of
$\pi_0 (R' \otimes_{R} Q)$. Since $Q$ is connective, the latter module can be identified with
$\Tor_{0}^{\pi_0 R}( \pi_0 R', \pi_0 Q)$ (Corollary \monoidref{huty}), which is isomorphic
to $\pi_0 Q$ in view of our assumption that $f$ induces an isomorphism
$\pi_0 R \rightarrow \pi_0 R'$.

We now prove that $\phi$ is essentially surjective. Let $\overline{P}$ be a projective $R'$-module.
Then there exists a free $R'$-module $\overline{F}$ and an idempotent map
$\overline{e}: \overline{F} \rightarrow \overline{F}$, so that $\overline{P}$ can be identified
with the colimit of the sequence
$$ \overline{F} \stackrel{ \overline{e} }{\rightarrow} \overline{F}
\stackrel{\overline{e}}{\rightarrow} \ldots. $$
Choose a free left $R$-module $F$ and an equivalence $\phi(F) \simeq \overline{F}$.
Using the first part of the proof, we deduce the existence of a map
$e: F \rightarrow F$ (not necessarily idempotent) such that the diagram
$$ \xymatrix{ \phi(F) \ar[r]^{ \phi(e) } \ar[d] & \phi(F) \ar[d] \\
\overline{F} \ar[r]^{\overline{e} } & \overline{F} }$$
commutes up to homotopy. Since the functor $M \mapsto R' \otimes_{R} M$ preserves colimits, we deduce that $\overline{P}$ is equivalent to $\phi(P)$, where $P$ denotes the colimit of the sequence
$$ F \stackrel{e}{\rightarrow} F \stackrel{e}{\rightarrow} \ldots.$$
To complete the proof, it will suffice to show that $P$ is projective. 
In view of Proposition \monoidref{redline}, it will suffice to show that $\pi_0 P$ is a projective module
over the ordinary associative ring $\pi_0 R$, and that $P$ is a flat $R$-module.
The first assertion follows from the isomorphism
$$\pi_0 P \simeq \pi_0 (R' \otimes_{R} P) \simeq \pi_0 \overline{P},$$
and the second from the observation that the collection of flat left $R$-modules is stable under filtered colimits (Lemma \monoidref{sudsy}).
\end{proof}

\begin{remark}
Let $A$ be an $A_{\infty}$-ring, and let $P$ be a projective left $A$-module. Then $P$ is a finitely generated projective $A$-module if and only if $\pi_0 P$ is finitely generated as a (discrete) left module
over $\pi_0 A$. The ``only if'' direction is obvious. For the converse, suppose that
$\pi_0 P$ is generated by a finite set of elements $\{ x_i \}_{i \in I}$. Let $M$ be the (finitely generated) free module on a set of generators $\{ X_i\}_{i \in I}$, so that we have a canonical map
$\phi: M \rightarrow P$. Since $P$ is projective and $\phi$ induces a surjection
$\pi_0 M \rightarrow \pi_0 P$, the map $\phi$ splits (Proposition \stableref{charprojjj}), so that $P$ is a direct summand of $M$.
\end{remark}

\begin{proof}[Proof of Theorem \ref{sucker}]
We first prove the forward implications. It will be convenient to phrase these results in a slightly more general form. Suppose given a commutative diagram $\sigma$:
$$ \xymatrix{ & B \ar[dr] & \\
A \ar[ur] \ar[rr] & & C}$$
of connective $E_{\infty}$-rings, and let $F(\sigma) = L_{B/A} \otimes_{B} C$.  We will show:
\begin{itemize}
\item[$(1')$] If $B$ is of finite presentation as a commutative $A$-algebra, then
$F(\sigma)$ is perfect as a $C$-module.
\item[$(2')$] if $B$ is almost of finite presentation as a commutative $A$-algebra, then
$F(\sigma)$ is almost perfect as a $C$-module.
\end{itemize}
We will obtain the forward implications of $(1)$ and $(2)$ by applying
these results in the case $B = C$. 

We first observe that the construction $\sigma \mapsto F(\sigma)$ defines a functor
from $(\EInfty)_{A/ \, /C}$ into $\Mod_{C}$. Using Remark \ref{stuck} and
Proposition \toposref{relcolfibtest}, we deduce that this functor preserves colimits.
Since the collection of finitely presented $C$-modules is closed under finite colimits and retracts,
it will suffice to prove $(1')$ in the case where $B$ is finitely generated and free. In this case,
$B = \Sym^{\ast}_{A} M$ for some finitely generated free $A$-module $M$. Using Proposition
\ref{puffle}, we deduce that $F(\sigma) \simeq M \otimes_{A} C$ is a finitely generated free $C$-module, as desired. 

We now prove $(2')$. It will suffice to show that for each $n \geq 0$, there exists a
commutative diagram
$$ \xymatrix{ & B' \ar[r]^{f} & B \ar[dr] & \\
A \ar[ur] \ar[rrr] & & & C }$$
such that $L_{B'/A} \otimes_{B'} C$ is perfect, and the induced map
$$ \tau_{ \leq n} ( L_{B'/A} \otimes_{B'} C) \rightarrow \tau_{ \leq n}( L_{B/A} \otimes_{B} C)$$
is an equivalence. To guarantee the latter condition, it suffices to choose $B'$ so that
the relative cotangent complex $L_{B/B'}$ is $n$-connective. Using
Corollary \ref{spiffle}, it suffices to guarantee that $f$ is $(n+1)$-connective.
Moreover, assertion $(1')$ implies that $L_{B'/A} \otimes_{B'} C$ will be finitely generated
so long as $B'$ is finitely presented as an $A$-algebra. The existence of a commutative
$A$-algebra with the desired properties now follows from Proposition \symmetricref{baseprops}.

We now prove the reverse implication of $(2)$. Assume that
$L_{B/A}$ is almost perfect, and that $\pi_0 B$ is a finitely presented as a (discrete) $\pi_0 A$-algebra.
To prove $(2)$, it will suffice to construct a sequence of maps
$$ A \rightarrow B(-1) \rightarrow B(0) \rightarrow B(1) \rightarrow \ldots \rightarrow B$$
such that each $B(n)$ is of finite presentation as an $A$-algebra, and each map $f_n: B(n) \rightarrow B$ is $(n+1)$-connective.
We begin by constructing $B(-1)$ with an even stronger property:
the map $f_{-1}$ induces an isomorphism $\pi_0 B(-1) \rightarrow \pi_0 B$. 
Choose a finite presentation
$$ \pi_0 B \simeq (\pi_0 A)[ x_1, \ldots, x_k ] / ( g_1, \ldots, g_m )$$
for the ordinary commutative ring $\pi_0 B$. Let $M$ denote the free
$A$-module generated by symbols $\{ X_i \}_{ 1 \leq i \leq k }$, so that the elements 
$\{ x_i \} \subseteq \pi_0 B$ determine a map of $A$-modules $M \rightarrow B$. Let
$h: \Sym^{\ast}_A(M) \rightarrow B$ be the adjoint map. We observe that there is a canonical isomorphism $\pi_0 ( \Sym^{\ast}_{A} (M) ) \simeq (\pi_0 A)[x_1, \ldots, x_k]$.
It follows that the image of the induced map
$$ \pi_0 \ker(h) \rightarrow \pi_0 \Sym^{\ast}_{A}(M)$$
can be identified with the ideal in $( \pi_0 A)[x_1, \ldots, x_k]$ generated by the elements
$\{ g_j \}_{1 \leq j \leq m}$. Choose elements $\{ \overline{g}_j \}_{1 \leq j \leq m}$
in $\pi_0 \ker(h)$ lifting $\{ g_j \}_{1 \leq j \leq m}$. Let $N$ be the free $A$-module
generated by symbols $\{ G_j \}_{1 \leq j \leq m}$, so that the elements $\{ \overline{g}_j \}_{1 \leq j \leq m}$ determine a map of $A$-modules $N \rightarrow \ker(h)$. This map classifies a commutative
diagram of $A$-modules
$$ \xymatrix{ N \ar[d] \ar[r] & 0 \ar[d] \\
\Sym^{\ast}_{A}(M) \ar[r]^-{h} & B. }$$
Adjoint to this, we obtain a commutative diagram of commutative $A$-algebras
$$ \xymatrix{ \Sym^{\ast}_{A} N \ar[d] \ar[r] & A \ar[d] \\
\Sym^{\ast}_{A}(M) \ar[r] & B.}$$
Let $B(-1)$ denote the tensor product
$$ A \otimes_{ \Sym^{\ast}_{A} N } \Sym^{\ast}_{A} M.$$
Then the above diagram classifies a map of commutative
$A$-algebras $f_{-1}: B(-1) \rightarrow B$. By construction, $B(-1)$ is of finite presentation as a commutative $A$-algebra, and $f_{-1}$ induces an isomorphism
$$\pi_0 B(-1) \simeq (\pi_0 A)[ x_1, \ldots, x_k]/(g_1, \ldots, g_m) \simeq \pi_0 B.$$

We now proceed in an inductive fashion. Assume that we have already constructed a connective $A$-algebra $B(n)$ which is of finite presentation over $A$, and an $(n+1)$-connective morphism $f_{n}: B(n) \rightarrow B$ of commutative $A$-algebras. Moreover,
we assume that the induced map $\pi_0 B(n) \rightarrow \pi_0 B$ is an isomorphism
(if $n \geq 0$ this is automatic; for $n = -1$ it follows from the specific construction given above).
We have a distinguished triangle of $B$-modules
$$ L_{B(n)/A} \otimes_{ B(n) } B \rightarrow L_{B/A} \rightarrow L_{B/ B(n) }
\rightarrow ( L_{ B(n)/A} \otimes_{ B(n) } B)[1].$$
By assumption, $L_{B/A}$ is almost perfect. Assertion $(2')$ implies that
$L_{B(n)/A} \otimes_{ B(n) } B$ is perfect. Using Proposition \monoidref{almor}, we deduce that
the relative cotangent complex $L_{B/B(n)}$ is almost perfect. Moreover, Corollary \ref{spiffle} ensures that $L_{B/B(n)}$ is $(n+2)$-connective. It follows that
$\pi_{n+2} L_{ B/ B(n) }$ is a finitely generated as a (discrete) module over
$\pi_0 B$. Using Theorem \ref{tulbas} and the bijectivity of the map $\pi_0 B(n) \rightarrow \pi_0 B$, we
deduce that the canonical map
$$ \pi_{n+1} \ker(f_{n}) \rightarrow \pi_{n+2} L_{B / B(n) }$$ is bijective.
Choose a finitely generated projective $B(n)$-module $M$ and a map
$M[n+1] \rightarrow \ker( f_{n} )$ such that the composition
$$ \pi_{0} M \simeq \pi_{n+1} M[n+1]
\rightarrow \pi_{n+1} \ker(f) \simeq \pi_{n+2} L_{B/B(n)}$$
is surjective (for example, we can take $M$ to be a free $B(n)$-module indexed by a set of
generators for the $\pi_0 B$-module $L_{B/B(n)}$). 
By construction, we have a commutative diagram of
$B(n)$-modules
$$ \xymatrix{ M[n+1] \ar[r] \ar[d] & 0 \ar[d] \\
B(n) \ar[r] & B. }$$
Adjoint to this, we obtain a diagram
$$ \xymatrix{ \Sym^{\ast}_{B(n)} ( M[n+1] ) \ar[r] \ar[d] & B(n) \ar[d] \\
B(n) \ar[r] & B. }$$
in the $\infty$-category of commutative $A$-algebras.
We now define $B(n+1)$ to be the pushout
$$ A \otimes_{ \Sym^{\ast}_{A} M[n+1] } B(n),$$ and $f_{n+1}: B(n+1) \rightarrow B$ to be the
induced map. It is clear that $B(n+1)$ is of finite presentation over $B(n)$, and therefore of finite presentation over $A$ (Remark \symmetricref{unkid}). To complete the proof of $(3)$, it will suffice to show that $\ker( f_{n+1})$ is $(n+2)$-connective. 

By construction, we have a commutative diagram
$$ \xymatrix{ & \pi_0 B(n+1) \ar[dr]^{e''} & \\
\pi_0 B(n) \ar[ur]^{e'} \ar[rr]^{e} & & \pi_0 B}$$
where the map $e'$ is surjective and $e$ is bijective. It follows that $e'$ and $e''$ are also bijective.
In view of Corollary \ref{spiffle}, it will now suffice to show $L_{B/B(n+1)}$ is $(n+3)$-connective.
We have a distinguished triangle of $B$-modules
$$ L_{B(n+1)/B(n)} \otimes_{ B(n+1)} B \rightarrow
L_{B/B(n)} \rightarrow L_{B/B(n+1)} \rightarrow
L_{ B(n+1)/B(n) }[1] \otimes_{ B(n+1)} B.$$
Using Proposition \ref{puffle} and Proposition \ref{basechunge}, we conclude that $L_{B(n+1)/B(n)}$ is canonically equivalent to $M[n+2] \otimes_{ B(n)} B(n+1)$. 
We may therefore rewrite our distinguished triangle
as $$ M[n+2] \otimes_{ B(n) } B \rightarrow
L_{B/B(n)} \rightarrow L_{B/B(n+1)} \rightarrow M[n+3] \otimes_{ B(n) }B.$$
Our inductive hypothesis and Corollary \ref{spiffle} guarantee that
$L_{B/B(n)}$ is $(n+2)$-connective. The $(n+3)$-connectiveness of 
$L_{B/B(n+1)}$ is therefore equivalent to the surjectivity of the map
$$ \pi_0 M \simeq \pi_{n+2} ( M[n+2] \otimes_{B(n)} B)
\rightarrow \pi_{n+2} L_{B/B(n)},$$
which is evident from our construction. This completes the proof of $(3)$.

To complete the proof of $(1)$, we use the same strategy but make a more careful choice of $M$. 
Let us assume that $L_{B/A}$ is perfect. It follows from the above construction that each cotangent complex $L_{B/B(n)}$ is likewise perfect. Using Proposition \monoidref{lastone}, we may assume 
$L_{B/B(-1)}$ is of $\Tor$-amplitude $\leq k+2$ for some $k \geq 0$. Moreover, for each $n \geq 0$
we have a distinguished triangle of $B$-modules
$$ L_{B/B(n-1)} \rightarrow L_{B/B(n)} \rightarrow P[n+2] \otimes_{B(n)} B
\rightarrow L_{B/B(n-1)}[-n-1],$$
where $P$ is finitely generated and projective, and therefore of $\Tor$-amplitude
$\leq 0$. Using Proposition \monoidref{lastone} and induction on $n$, we deduce that 
the $\Tor$-amplitude of $L_{B/B(n)}$ is $\leq k+2$ for $n \leq k$. In particular,  
the $B$-module $\overline{M} = L_{B/B(k)}[-k-2]$ is connective and has $\Tor$-amplitude $\leq 0$. 
It follows from Remark \monoidref{lazardapp} that $\overline{M}$ is a flat $B$-module. 
Invoking Proposition \monoidref{projj}, we conclude that $\overline{M}$ is
a finitely generated projective $B$-module. Using Lemma \ref{corpus}, we can choose a finitely
generated projective $B(k)$-module $M$ and an equivalance
$M[n+2] \otimes_{B(k)} B \simeq L_{B/B(k)}.$
Using this map in the construction outlined above, we guarantee that
the relative cotangent complex $L_{B/B(k+1)}$ vanishes. It follows from
Corollary \ref{twain} that the map $f_{k+1}: B(k+1) \rightarrow B$ is an equivalence,
so that $B$ is of finite presentation as a commutative $A$-algebra as desired.
\end{proof}

\subsection{\Etale Algebras}\label{etalg}

In this section, we will introduce and study the class of {\it \etalenospace} morphisms between $E_{\infty}$-rings.

\begin{definition}\label{defet}
Let $f: A \rightarrow B$ be a morphism of $E_{\infty}$-rings. We will say that $f$ is {\it \etalenospace} if the following conditions are satisfied:
\begin{itemize}
\item[$(1)$] The induced map $\pi_0 A \rightarrow \pi_0 B$ is an \etale homomorphism of commutative rings.
\item[$(2)$] For every integer $n \in \Z$, the associated map
$\pi_{n} A \otimes_{ \pi_0 A} \pi_0 B \rightarrow \pi_{n} B$
is an isomorphism of abelian groups.
\end{itemize}
\end{definition}

In other words, a map of $E_{\infty}$-rings $f: A \rightarrow B$ is \etale if and only if it
is flat and the underlying map $\pi_0 A \rightarrow \pi_0 B$ is \etalenospace.

\begin{remark}
Let $A$ be an ordinary commutative ring, regarded as a discrete $E_{\infty}$-ring.
A morphism of $E_{\infty}$-rings $f: A \rightarrow B$ is \etale (in the sense of Definition \ref{defet}) if and only if $B$ is discrete, and \etale over $A$ when regarded as an ordinary commutative ring.
\end{remark}

\begin{remark}\label{swiide}
Suppose given a commutative diagram
$$ \xymatrix{ & B \ar[dr]^{g} & \\
A \ar[ur]^{f} \ar[rr]^{h} & & C. }$$
If $f$ is \etalenospace, then $g$ is \etale if and only if $h$ is \etale (in other words, any map between
\etale commutative $A$-algebras is automatically \etale). The ``only if'' direction is obvious.
For the converse, let us suppose that $f$ and $h$ are both \etalenospace. The induced maps
$\pi_0 A \rightarrow \pi_0 B$ and $\pi_0 A \rightarrow \pi_0 C$ are both \etale map of ordinary commutative rings, so that $g$ also induces an \etale map $\pi_0 B \rightarrow \pi_0 C$. 
We now observe that for $n \in \Z$, we have a commutative diagram
$$ \xymatrix{ (\pi_{n} A \otimes_{ \pi_0 A} \pi_0 B) \otimes_{ \pi_{0} B } \pi_0 C
\ar[r] \ar[d] & \pi_{n} A \otimes_{ \pi_0 A} \pi_0 C \ar[d] \\
\pi_{n} B \otimes_{ \pi_0 B} \pi_0 C \ar[r] & \pi_{n} C. }$$
Since $f$ and $h$ are flat, the vertical maps are isomorphisms. The upper horizontal map is obviously an isomorphism, so the lower horizontal map is an isomorphism as well.
\end{remark}

\begin{remark}\label{etalepush}
Suppose given a pushout diagram of $E_{\infty}$-rings
$$ \xymatrix{ A \ar[r] \ar[d]^{f} & A' \ar[d]^{f'} \\
B \ar[r] & B'. }$$
If $f$ is \etalenospace, then so is $f'$. The flatness of $f$ follows from Proposition \monoidref{urwise}.
Moreover, Corollary \monoidref{siwe} ensures that the induced diagram
$$ \xymatrix{ \pi_0 A \ar[r] \ar[d] & \pi_0 A' \ar[d] \\
\pi_0 B \ar[r] & \pi_0 B' }$$
is a pushout in the category of ordinary commutative rings. Since the left vertical map is \etalenospace, it follows that the right vertical map is \etalenospace, so that $f'$ is likewise \etalenospace.
\end{remark}

\begin{remark}\label{unter}
Let $f: A \rightarrow A'$ be a morphism of $E_{\infty}$-rings which induces an isomorphism
$\pi_{i} A \rightarrow \pi_{i} A'$ for $i \geq 0$. According to Proposition \symmetricref{flatner},
the tensor product $\otimes_{A} A'$ induces an equivalence from the $\infty$-category of
flat commutative $A$-algebras to the $\infty$-category of flat commutative $A'$-algebras. Moreover,
if $B$ is a flat commutative $A$-algebra, then the canonical map
$\pi_0 B \rightarrow \pi_0 (B \otimes_A A')$ is an isomorphism (Corollary \monoidref{siwe}), so that
$B$ is \etale over $A$ if and only if $(B \otimes_{A} A'$ is \etale over $A'$. It follows that
$\otimes_{A} A'$ induces an equivalence from the $\infty$-category of
\etale commutative $A$-algebras to the $\infty$-category of \etale commutative $A'$-algebras.
\end{remark}

The main result of section asserts that if $f: A \rightarrow B$ is an \etale map of $E_{\infty}$-rings, then the relative cotangent complex $L_{B/A}$ vanishes. We first treat the case where $A$ and $B$ are discrete.

\begin{lemma}\label{ablepoot}
Let $f: A \rightarrow B$ be an \etale homomorphism of commutative rings (which we regard as discrete $E_{\infty}$-rings). Then the relative cotangent complex $L_{B/A}$ vanishes.
\end{lemma}

\begin{proof}
For every residue field $k$ of $A$, the tensor product $k \otimes_{A} B$ is a $k$-algebra of some finite dimension $d(k)$. Let $n$ be the maximum of all these dimensions (in other words, the maximal cardinality of any geometric fiber of the associated map $\Spec B \rightarrow \Spec A$). We will prove our result by induction on $n$. If $n = 0$, then $B$ is the zero ring and there is nothing to prove.

Let $B' = B \otimes_{A} B$. We observe that, since $B$ is flat over $A$, this tensor product is again
discrete (and may therefore be identified with the classical tensor product in the setting of commuttive algebra). According to Proposition \ref{basechunge}, we have a canonical equivalence
$L_{B/A} \otimes_{B} B' \simeq L_{B'/B}$. Since $B'$ is faithfully flat over $B$, 
we deduce that $L_{B/A}$ vanishes if and only if $L_{B'/B}$ vanishes.

If $n = 1$, then $B' \simeq B$ so that $L_{B'/B}$ vanishes as desired. Let us assume therefore that $n > 2$. Since $f$ is \etalenospace, the commutative ring $B'$ factors as a product $B'_0 \times B'_1$, where
$B'_0$ is the kernel of the multiplication map $B \otimes_{A} B \rightarrow B$, and $B'_1 \simeq B$.
To complete the proof, we will show that $L_{B'/B} \otimes_{B'} B'_i$ vanishes for $0 \leq i \leq 1$.
We have a distinguished triangle
$$ L_{B'/B} \otimes_{B'} B'_i \rightarrow L_{B'_i/B} \rightarrow L_{B'_i/B'} \rightarrow (L_{B'/B} \otimes_{B'} B'_i)[1].$$
It therefore suffices to show that the relative cotangent complexes $L_{B'_i/B}$ and
$L_{B'_i/B'}$ vanish. Both of these results follow from our inductive hypothesis.
\end{proof}

\begin{proposition}\label{etrel}
Let $f: A \rightarrow B$ be an \etale homomorphism of $E_{\infty}$-rings. Then the relative cotangent complex $L_{B/A}$ vanishes.
\end{proposition}

\begin{proof}
Choose a connective cover $A' \rightarrow A$. It follows from Remark \ref{unter} that there exists
an \etale $A'$-algebra $B'$ and an equivalence $B \simeq B' \otimes_{A'} A$. According to
Proposition \ref{basechunge}, we have an equivalence $L_{B'/A'} \otimes_{B'} B \simeq L_{B/A}$. It will therefore suffice to show that the relative cotangent complex $L_{B'/A'}$ vanishes. In other words, we may reduce to the case where $A$ is connective. Since $B$ is flat over $A$, $B$ is also connective.
According to Proposition \ref{spate}, the relative cotangent complex $L_{B/A}$ is connective.
If $L_{B/A}$ does not vanish, then there exists a smallest integer $n$ (automatically nonnegative) such that $\pi_{n} L_{B/A} \neq 0$.

Let us regard the ordinary commutative ring $\pi_0 A$ as a discrete $E_{\infty}$-ring, so that we have
a morphism of $E_{\infty}$-rings $A \rightarrow \pi_0 A$. Since $B$ is flat over $A$, we have an equivalence of discrete $E_{\infty}$-rings $\pi_0 B \simeq \pi_0 A \otimes_{A} B$ (Corollary \monoidref{siwe}). 
Using Corollary \monoidref{presiwe} and Proposition \ref{basechunge}, we deduce the existence of isomorphisms
$$\pi_n L_{\pi_0 B/ \pi_0 A} \simeq \pi_n ( L_{B/A} \otimes_{B} \pi_0 B)
\simeq \pi_n L_{B/A}.$$
This leads to a contradiction, since $\pi_{n} L_{\pi_0 B/ \pi_0 A}$ vanishes by
Lemma \ref{ablepoot}.
\end{proof}

\begin{remark}\label{spunk}
In view of the distinguished triangle
$$ L_{A} \otimes_{A} B \rightarrow L_B \rightarrow L_{B/A} \rightarrow (L_{A} \otimes_{A} B)[1]$$
associated to a morphism $f: A \rightarrow B$ of $E_{\infty}$-rings, Proposition 
\ref{etrel} is equivalent to the assertion that if $f$ is \etalenospace, then $f$ induces an equivalence
$L_{A} \otimes_{A} B \rightarrow L_{B}$. In other words, if we regard the cotangent complex
functor $L$ as a section of the projection map $p: T_{\EInfty} \rightarrow \EInfty$, then
$L$ carries \etale morphisms of $E_{\infty}$-rings to $p$-coCartesian morphisms in
$\CMod(\Spectra)$.
\end{remark}

\section{Deformation Theory}\label{sec4}

Let $R$ be a commutative ring. A {\it square-zero extension} of $R$ is a commutative ring
$\widetilde{R}$ equipped with a surjection $\phi: \widetilde{R} \rightarrow R$, with the property that
the product of any two elements in $\ker(\phi)$ is zero. In this case, the kernel $M = \ker(\phi)$ inherits the structure on $R$-module.

Let $\widetilde{R}$ be a square-zero extension of a commutative ring $R$ by an $R$-module $M$.
There exists a ring homomorphism
$$ ( R \oplus M) \times_{R} \widetilde{R} \rightarrow \widetilde{R},$$
given by the formula
$$ (r, m, \widetilde{r} ) \mapsto \widetilde{r} + m. $$
This map exhibits $\widetilde{R}$ as endowed with an {\em action} of $R \oplus M$ in the category
of commutative rings with a map to $R$ (we observe that $R \oplus M$ has the structure of an abelian group object in this category). Consequently, in some sense square-zero extensions of $R$ by $M$ can be viewed as {\em torsors} for the trivial square-zero extension $R \oplus M$.  

In general, if $\phi: \widetilde{R} \rightarrow R$ is a square-zero extension of $R$ by $M \simeq \ker(\phi)$, we say that $\widetilde{R}$ is {\em trivial} if there exists an isomorphism of commutative rings
$\widetilde{R} \simeq R \oplus M$. Equivalently, $\widetilde{R}$ is a trivial square-zero extension of $R$ if and only if the surjection $\phi: \widetilde{R} \rightarrow R$ admits a section. In fact, more is true: giving a section of $\phi$ is equivalent to giving an isomorphism $\widetilde{R} \simeq R \oplus M$, which is the identity on $M$ and compatible with the projection to $R$. Such an isomorphism need not exist (for example, we could take $R = \Z / p \Z$ and $\widetilde{R} = \Z / p^2 \Z$), and need not be unique.
However, any two sections of $\phi$ differ by some map $d: R \rightarrow M$. In this case, it is easy to see that $d$ is a derivation of $R$ into $M$, and therefore classified by a map from
the module of \Kahler differentials $\Omega_{R}$ into $M$. Conversely, any derivation of $R$ into $M$ determines an automorphism of $\widetilde{R}$ (whether $\widetilde{R}$ is trivial or not), which permutes the set of sections of $\phi$. Consequently, we deduce that the {\em automorphism group}
of the trivial square zero extension of $R$ by $M$ can be identified with the group of $R$-module homomorphisms $\Ext^{0}_{R}( \Omega_{R}, M)$. 

It is tempting to try to pursue this analogy further, and to try identify the {\em isomorphism classes}
of square-zero extensions of $R$ by $M$ with the higher $\Ext$-group $\Ext^{1}_{R}( \Omega_{R}, M)$. Given an extension class $\eta \in \Ext^1_{R}( \Omega_{R}, M)$, we can indeed construct a square-zero extension $\widetilde{R}$ of $R$ by $M$. Indeed, let us view $\eta$ as defining an exact sequence
$$ 0 \rightarrow M \rightarrow \widetilde{M} \stackrel{f}{\rightarrow} \Omega_{R} \rightarrow 0$$
in the category of $R$-modules. We now form a pullback diagram
$$ \xymatrix{ \widetilde{R} \ar[r] \ar[d] & R \ar[d]^{d} \\
\widetilde{M} \ar[r] & \Omega_{R} }$$
in the category of abelian groups. We can identify elements of $\widetilde{R}$ with pairs
$(r, \widetilde{m})$, where $r \in R$ and $\widetilde{m} \in \widetilde{M}$ satisfy the equation
$f( \widetilde{m} ) = dr$. The abelian group $\widetilde{R}$ admits a ring structure, given by the formula
$$( r, \widetilde{m} ) ( r', \widetilde{m}') = (rr', r' \widetilde{m} + r \widetilde{m}' ).$$
It is easy to check that $\widetilde{R}$ is a square-zero extension of $R$ by $M$.
However, not every square-zero extension of $R$ by $M$ can be obtained from this construction. 
In order to obtain {\em all} square-zero extensions of $R$, it is necessary to replace the module of \Kahler differentials $\Omega_{R}$ by a more refined invariant, such as the $E_{\infty}$ cotangent complex $L_{R}$. 

Our goal in this section is to study analogues of all of the ideas sketched above in the setting of 
$E_{\infty}$-rings. In \S \ref{sqzero} we will introduce the class of {\em square-zero extensions} in the $\infty$-category of $E_{\infty}$-rings. Roughly speaking, we will mimic the above construction to produce a functor $\Phi: \CDer \Rightarrow \Fun( \Delta^1, \EInfty)$. Here
$\CDer$ denotes an $\infty$-category of triples $(A, M, \eta)$, where $A$ is an $E_{\infty}$-ring,
$M$ is an $A$-module, and $\eta: A \rightarrow M$ is a derivation (which we can identify
with an $A$-linear map from $L_{A}$ into $M$). The functor $\Phi$ carries
$(A, M, \eta)$ to a map $A^{\eta} \rightarrow A$; here we will refer to $A^{\eta}$ as 
{\it the square-zero extension of $A$ classified by $\eta$}. 

Using this definition, it follows more or less tautologically that square-zero extensions of an $E_{\infty}$-ring $A$ are ``controlled'' by the absolute cotangent complex of $L_A$. For example, if $L_{A}$ vanishes, then every square-zero extension of $A$ by an $A$-module $M$ is equivalent to the
trivial extension $A \oplus M$ constructed in \S \ref{sec2}. The trouble with this approach is that it is not obvious how to give an {\em intrinsic} characterization of the class of square-zero extensions.
For example, suppose that $f: \widetilde{A} \rightarrow A$ is a square-zero extension of 
$A$ by an $A$-module $M$. We then have a canonical identification $M \simeq \ker(f)$ in
the $\infty$-category of $\widetilde{A}$-modules. However, in general there is no way to recover
the $A$-module structure on $\ker(f)$ from the morphism $f$ alone. In other words, the functor
$\Phi$ described above fails to be fully faithful. In \S \ref{smallext}, we will attempt to remedy the situation by studying a restricted class of square-zero extensions, which we will call {\it $n$-small extensions}. 
This class of morphisms has two important features:
\begin{itemize}
\item[$(i)$] Given a map $f: \widetilde{A} \rightarrow A$, it is easy to decide whether or not
$f$ is an $n$-small extension. Namely, one must check that $\ker(f)$ has certain connectivity
properties, and that a certain multiplication map $\pi_{n} \ker(f) \otimes \pi_{n} \ker(f) \rightarrow \pi_{2n} \ker(f)$ vanishes.

\item[$(ii)$] On the class of $n$-small extensions of $E_{\infty}$-rings, one can construct an inverse to the functor $\Phi$ (Theorem \ref{exsqze}). In particular, every $n$-small extension is a square-zero extension.
\end{itemize}
In conjunction, $(i)$ and $(ii)$ imply that square-zero extensions exist in abundance. For
example, if $A$ is a connective $E_{\infty}$-ring, then the Postnikov tower
$$ \ldots \rightarrow \tau_{\leq 2} A \rightarrow \tau_{\leq 1} A \rightarrow \tau_{\leq 0} A$$
is a sequence of square-zero extensions. 

In \S \ref{defthe}, we will study the deformation theory of $E_{\infty}$-rings. A typical problem
is the following: let $f: A \rightarrow B$ be a map of connective $E_{\infty}$-rings, and suppose that
$\widetilde{A}$ is a square-zero extension of $A$ by a connective $A$-module $M$. Under what circumstances can we ``lift'' $f$ to obtain a map $\widetilde{f}: \widetilde{A} \rightarrow \widetilde{B}$, such that the diagram
$$ \xymatrix{ \widetilde{A} \ar[r] \ar[d]^{ \widetilde{f}} & A \ar[d]^{f} \\
\widetilde{B} \ar[r] & B }$$
is a pushout square? Our main result, Theorem \ref{h2h2}, asserts that in this case $\widetilde{B}$
can automatically be identified with a square-zero extension of $B$ by $M \otimes_{A} B$. 
The problem of constructing $\widetilde{B}$ is therefore ``linear'' in nature: if $\widetilde{A}$ is classified by a map $\eta: L_A \rightarrow M$ in the $\infty$-category of $A$-modules, then
$\widetilde{B}$ exists if and only if the induced map
$L_{A} \otimes_{A} B \rightarrow M \otimes_{A} B$ factors through the cotangent complex $L_{B}$.
This foundational result will play an important role in our study of moduli problems in derived algebraic geometry. 

In \S \ref{classet}, we will apply the ideas sketched above to a more concrete problem: the classification of commutative algebras which are \etale over a given $E_{\infty}$-ring $A$. Our main result, Theorem \ref{turncoat}, asserts that the $\infty$-category of \etale $A$-algebras is equivalent to the ordinary
category of \etale $\pi_0 A$-algebras. This result will play an important role in the foundations of derived algebraic geometry.

\subsection{Square-Zero Extensions}\label{sqzero}

In this section, we will introduce the theory of {\em square-zero extensions}. Although
we are ultimately interested in applying these ideas in the setting of $E_{\infty}$-rings, we will
begin by working in an arbitrary presentable $\infty$-category $\calC$. The theory of square-zero extensions presented here has many applications even in ``nonalgebraic'' situations. For example, when $\calC$ is the $\infty$-category of spaces, it is closely related to classical obstruction theory.

\begin{definition}\label{tancor}
Let $\calC$ be a presentable $\infty$-category, and let 
$p: \calM^{T}(\calC) \rightarrow \Delta^1 \times \calC$ denote a tangent correspondence to
$\calC$ (see Definition \ref{extcor}). A {\it derivation} in $\calC$ is a map $f: \Delta^1 \rightarrow \calM^{T}(\calC)$ such that $p \circ f$ coincides with the inclusion $\Delta^1 \times \{A\} \subseteq \Delta^1 \times \calC$, for some $A \in \calC$. In this case, we will identify $f$ with a morphism $\eta: A \rightarrow M$
in $\calM^{T}(\calC)$, where $M \in T_{\calC} \times_{\calC} \{A\} \simeq \Stab( \calC^{/A})$.
We will also say that $\eta: A \rightarrow M$ is a {\it derivation of $A$ into $M$}.

We let $\CDer(\calC)$ denote the fiber product
$\Fun( \Delta^1, \calM^{T}(\calC) ) \times_{ \Fun( \Delta^1, \Delta^1 \times \calC) } \calC$.
We will refer to $\CDer(\calC)$ as {\it the $\infty$-category of derivations in $\calC$}.
\end{definition}

\begin{remark}
In the situation of Definition \ref{tancor}, let $L: \calC \rightarrow T_{\calC}$ be a cotangent
complex functor. A derivation $\eta: A \rightarrow M$ can be identified with a map
$d: L_{A} \rightarrow M$ in the fiber $T_{\calC} \times_{\calC} \{A\} \simeq \Stab( \calC^{/A})$. 
We will often abuse terminology by identifying $\eta$ with $d$, and referring to $d$
as a {\it derivation of $A$ into $M$}.
\end{remark}

\begin{definition}\label{exder}
Let $\calC$ be a presentable $\infty$-category, and
let $p: \calM^{T}(\calC) \rightarrow \Delta^1 \times \calC$ be a tangent correspondence for $\calC$.
An {\it extended derivation} is a diagram $\sigma$
$$ \xymatrix{ \widetilde{A} \ar[r]^{f} \ar[d] & A \ar[d]^{\eta} \\
0 \ar[r] & M }$$
in $\calM^{T}(\calC)$ with the following properties:
\begin{itemize}

\item[$(1)$] The diagram $\sigma$ is a pullback square.

\item[$(2)$] The objects $\widetilde{A}$ and $A$ belong to $\calC \subseteq \calM^{T}(\calC)$, while
$0$ and $M$ belong to $T_{\calC} \subseteq \calM^{T}(\calC)$.

\item[$(3)$] Let $\overline{f}: \Delta^1 \rightarrow \calC$ be the map which classifies
the morphism $f$ appearing in the diagram above, and let 
$e: \Delta^1 \times \Delta^1 \rightarrow \Delta^1$ be the unique map such that
$e^{-1} \{0\} = \{0\} \times \{0\}$. Then the diagram
$$ \xymatrix{ \Delta^1 \times \Delta^1 \ar[r]^{\sigma} \ar[d]^{e} & \calM^{T}(\calC) \ar[r]^{p} & \Delta^1 \times \calC \ar[d] \\
\Delta^1 \ar[rr]^{ \overline{f} } & & \calC }$$
is commutative.

\item[$(4)$] The object $0 \in T_{\calC}$ is a zero object of $\Stab( \calC^{/A})$. Equivalently,
$0$ is a $p$-initial vertex of $\calM^{T}(\calC)$. 
\end{itemize}

We let $\widetilde{\CDer}( \calC )$ denote the full subcategory of
$$\Fun( \Delta^1 \times \Delta^1, \calM^{T}(\calC) ) \times_{ \Fun( \Delta^1 \times \Delta^1, \Delta^1 \times \calC) } \Fun( \Delta^1, \calC)$$
spanned by the extended derivations. 

If $\sigma$ is an extended derivation in $\calC$, then $\eta$ is a derivation in $\calC$. We
therefore obtain a restriction functor
$$ \widetilde{\CDer}(\calC) \rightarrow \CDer(\calC).$$
\end{definition}

Let $\calC$ and $\calM^{T}(\calC)$ be above, and let
$$\sigma \in \Fun( \Delta^1 \times \Delta^1, \calM^{T}(\calC) ) \times_{ \Fun( \Delta^1 \times \Delta^1, \Delta^1 \times \calC) } \Fun( \Delta^1, \calC).$$
Then $\sigma$ automatically satisfies conditions $(2)$ and $(3)$ of Definition \ref{exder}.
Moreover, $\sigma$ satisfies condition $(4)$ if and only if $\sigma$ is a $p$-left Kan extension
of $\sigma | \{1\} \times \Delta^1$ at the object $\{0\} \times \{1\}$. Invoking Proposition \toposref{lklk} twice, we deduce the following:

\begin{lemma}
Let $\calC$ be a presentable $\infty$-category. Then the forgetful 
$\psi: \widetilde{\CDer}(\calC) \rightarrow \CDer(\calC)$ is a trivial Kan fibration.
\end{lemma}

\begin{notation}\label{stubble}
Let $\calC$ be a presentable $\infty$-category. We let $\Phi: \CDer(\calC) \rightarrow \Fun( \Delta^1, \calC)$ denote the composition
$$ \CDer(\calC) \rightarrow \widetilde{\CDer}(\calC)
\rightarrow \Fun(\Delta^1, \calC),$$
where the first map is a section of the trivial fibration $\widetilde{\CDer}(\calC)
\rightarrow \CDer(\calC)$, and the second map is induced by the inclusion
$\Delta^1 \times \{0\} \subseteq \Delta^1 \times \Delta^1$. In other words, 
$\Phi$ associates to every derivation $\eta: A \rightarrow M$ a map
$f: \widetilde{A} \rightarrow A$ which fits into a pullback diagram
$$ \xymatrix{ A^{\eta} \ar[r]^{f} \ar[d] & A \ar[d]^{\eta} \\
0 \ar[r] & M }$$
in the $\infty$-category $\calM^{T}(\calC)$.
\end{notation}

\begin{definition}
Let $\calC$ be a presentable $\infty$-category, and let
$\Phi: \CDer(\calC) \rightarrow \Fun(\Delta^1, \calC)$ be the functor described in
Notation \ref{stubble}. We will denote the image of a derivation
$(\eta: A \rightarrow M) \in \CDer(\calC)$ under the functor $\Phi$ by
$(A^{\eta} \rightarrow A)$.

Let $f: \widetilde{A} \rightarrow A$ be a morphism in $\calC$. We will say that $f$ is a {\it square-zero extension} if there exists a derivation
$\eta: A \rightarrow M$ in $\calC$ and an equivalence
$B \simeq A^{\eta}$ in the $\infty$-category $\calC^{/A}$. In this case, we will also say that $\widetilde{A}$ is a {\it square-zero extension of $A$ by $M[-1]$}.
\end{definition}

\begin{remark}\label{intertine}
Let $\eta: A \rightarrow M$ be a derivation in a presentable $\infty$-category $\calC$, and let
$A \oplus M$ denote the image of $M$ under the functor $\Omega^{\infty}: \Stab( \calC^{/A} ) \rightarrow \calC$. Using Proposition \toposref{chocolatelast}, we conclude that there is a pullback diagram
$$ \xymatrix{ A^{\eta} \ar[r] \ar[d] & A \ar[d]^{d_{\eta}} \\
A \ar[r]^{d_0} & A \oplus M }$$
in the $\infty$-category $\calC$. Here we identify $d_0$ with the map associated to the zero derivation $L_A \rightarrow M$.
\end{remark}

\begin{example}\label{invu}
Let $\calC$ be a presentable $\infty$-category containing an object $A$. Let
$M \in \Stab( \calC^{/A})$, and let $\eta: A \rightarrow M$ be the derivation classified by the zero map
$L_A \rightarrow M$ in $\Stab( \calC^{/A})$. Since the functor
$\Omega^{\infty}: \Stab( \calC^{/A} ) \rightarrow \calC^{/A}$ preserves small limits, we conclude
from Remark \ref{intertine} that the square-zero extension $A^{\eta}$ can be identified with
$\Omega^{\infty} M[-1]$. In particular, if $M = 0$, then the canonical map $A^{\eta} \rightarrow A$ is an equivalence, so we can identify $A^{\eta}$ with $A$. 
\end{example}

\begin{warning}
Let $\calC$ be a presentable $\infty$-category, and let $f: \widetilde{A} \rightarrow A$ be a morphism in $\calC$. Suppose $f$ is a square-zero extension, so that there exists a map
$\eta: L_A \rightarrow M$ in $\Stab(\calC_{/A})$ and an equivalence $\widetilde{A} \simeq A^{\eta}$. In this situation, the object $M$ and the map $\eta$ need not be uniquely determined, even up to equivalence. However, this is true in some favorable situations; see Theorem \ref{exsqze}.
\end{warning}

\begin{example}
Suppose we are given a fibration of simply connected spaces
$$ F \rightarrow E \stackrel{f}{\rightarrow} B,$$
such that $\pi_{k} F \simeq \ast$ for all $k \neq n$. In this case, the fibration $f$ is
{\em classified} by a map $\eta$ from $B$ into an Eilenberg-MacLane space $K(A, n+1)$, where
$A = \pi_{n} F$. It follows that we have a homotopy pullback diagram
$$ \xymatrix{ E \ar[r]^{f} \ar[d] & B \ar[d]^{(\id, \eta)} \\
B \ar[r]^-{(\id, 0)} & B \times K(A, n+1). } $$
The space $B \times K(A, n+1)$ is an infinite loop object of the $\infty$-category
of spaces over $B$: it has deloopings given by $K(A, n+m)$ for $m \geq 1$.
Consequently, the above diagram exhibits $E$ as a square-zero extension of $B$ in the
$\infty$-category of spaces.

In fact, using a slightly more sophisticated version of the same construction, one can show that the same result holds without any assumptions of simple-connectedness; moreover it is sufficient that the homotopy groups of $F$ be confined to a small range, rather than a single degree. We will prove an algebraic analogue of this statement in the next section.
\end{example}

\subsection{Small Extensions}\label{smallext}

In \S \ref{sqzero} we introduced the notion of a square-zero extension in an arbitrary presentable $\infty$-category $\calC$. However, it is not so easy to recognize square-zero extensions directly from the definition. Our goal in this section is to produce a large class of easily recognized examples in the case where $\calC$ is the $\infty$-category of $E_{\infty}$-rings.

\begin{notation}\label{exxy}
Throughout this section, we will be exclusively concerned with square-zero extensions in
the setting of $E_{\infty}$-rings. To simplify the exposition, we will employ the following conventions:
\begin{itemize}
\item[$(a)$] We let $\calM^{T}$ denote a tangent correspondence to the $\infty$-category
$\EInfty$ of $E_{\infty}$-rings, and $p: \calM^{T} \rightarrow \Delta^1 \times \EInfty$ the corresponding projection map (see Definition \ref{extcor}).
\item[$(b)$] We let $\CDer$ denote the $\infty$-category $\CDer( \EInfty)$ of derivations
in $\EInfty$.
\item[$(c)$] We let $\widetilde{\CDer}$ denote the $\infty$-category $\widetilde{\CDer}(\EInfty)$ of extended derivations in $\EInfty$.
\end{itemize}
\end{notation}

Ideally, we would like to assert that a map $f: \widetilde{A} \rightarrow A$ is a square-zero extension if and only if the multiplication on $\widetilde{A}$ is trivial on $\ker(f)$. Our goal in this section is to prove Theorem \ref{exsqze}, which asserts that this description is correct provided that $A$ is connective and the homotopy groups of $\ker(f)$ are confined to a narrow range. To make a more precise statement, we need to introduce a bit of terminology.

\begin{definition}\label{spazzy}
Let $f: \widetilde{A} \rightarrow A$ be a map of $E_{\infty}$-rings, and let $n$ be a nonnegative integer.
We will say that $f$ is an {\it $n$-small extension} if the following conditions are satisfied:
\begin{itemize}
\item[$(1)$] The $E_{\infty}$-ring $A$ is connective.
\item[$(2)$] The homotopy groups $\pi_{i} \ker(f)$ vanish unless
$n \leq i \leq 2n$.
\item[$(3)$] The multiplication map
$$ \ker(f) \otimes_{ \widetilde{A} } \ker(f)
\rightarrow \widetilde{A} \otimes_{ \widetilde{A} } \ker(f) \simeq \ker(f)$$
is nullhomotopic.
\end{itemize}
We let $\Fun^{(n)}( \Delta^1, \EInfty)$ denote the full subcategory of
$\Fun( \Delta^1, \EInfty)$ spanned by the $n$-small extensions.
\end{definition}

\begin{remark}
In the situation of Definition \ref{spazzy}, we will also say that {\it $\widetilde{A}$ is an $n$-small extension of $A$}. Note that in this case $\widetilde{A}$ is automatically connective.
\end{remark}

\begin{remark}
Let $f: \widetilde{A} \rightarrow A$ be as in Definition \ref{spazzy}. The map
$\ker(f) \otimes_{\widetilde{A}} \ker(f) \rightarrow \ker(f)$ appearing in condition $(3)$ is defined assymetrically, and is not evidently invariant under interchange of the two factors in the source.
\end{remark}

\begin{remark}\label{spuzzy}
Let $f: \widetilde{A} \rightarrow A$ be a map of $E_{\infty}$-rings, and suppose that $f$ satisfies
conditions $(1)$ and $(2)$ of Definition \ref{spazzy}. It follows that
$\ker(f) \otimes_{ \widetilde{A} } \ker(f)$ is $(2n)$-connective, and that $\ker(f)$ is $(2n)$-truncated.
Consequently, the multiplication map $\ker(f) \otimes_{ \widetilde{A} } \ker(f)$ is determined by
the induced map of abelian groups
$$ \pi_{n} \ker(f) \otimes_{ \pi_0 \widetilde{A} } \pi_{n} \ker(f) 
\simeq \pi_{2n} ( \ker(f) \otimes_{ \widetilde{A} } \ker(f) ) \rightarrow \pi_{2n} \ker(f). $$
We can identify this map with a bilinear multiplication
$$\phi: \pi_{n} \ker(f) \times \pi_{n} \ker(f) \rightarrow \pi_{2n} \ker(f),$$
and condition $(3)$ is equivalent to the vanishing of $\phi$.
\end{remark}

\begin{remark}
Let $A$ be a commutative ring, which we regard as a discrete $E_{\infty}$-ring. A map
$f: \widetilde{A} \rightarrow A$ is a $0$-small extension if and only if the following conditions are satisfied:
\begin{itemize}
\item[$(a)$] The $E_{\infty}$-ring $\widetilde{A}$ is also discrete, so we can identify $f$ with a map of ordinary commutative rings.
\item[$(b)$] As a homomorphism of commutative rings, $f$ is surjective.
\item[$(c)$] The kernel of $f$ is square-zero, in the sense of classical commutative algebra.
\end{itemize}
In other words, the theory of $0$-small extensions of discrete $E_{\infty}$-rings is equivalent to the classical theory of square-zero extensions between ordinary commutative rings.
\end{remark}

We now come to our main result:

\begin{theorem}\label{exsqze}
Fix an integer $n \geq 0$. Let $\CDer^{(n)}$ denote the full subcategory
of $\CDer$ spanned by those derivations $(\eta: A \rightarrow M)$ such that
$A$ is connective, and $\pi_{i} M$ vanishes unless $n+1 \leq i \leq 2n+1$.
Then the functor
$ \Phi: \CDer( \EInfty) \rightarrow \Fun( \Delta^1, \EInfty)$
of Notation \ref{stubble} induces an equivalence of $\infty$-categories
$$ \Phi^{(n)}: \CDer^{(n)} \rightarrow \Fun^{(n)}( \Delta^1, \EInfty).$$
\end{theorem}

\begin{corollary}
Every $n$-small extension of $E_{\infty}$-rings is a square-zero extension.
\end{corollary}

\begin{corollary}\label{subte}
Let $A$ be a connective $E_{\infty}$-ring. Then every map
in the Postnikov tower
$$ \ldots \rightarrow \tau_{\leq 3} A \rightarrow \tau_{\leq 2} A \rightarrow \tau_{\leq 1} A \rightarrow
\tau_{\leq 0} A$$
is a square-zero extension.
\end{corollary}

Corollary \ref{subte} underlines the importance of the cotangent complex in the theory of $E_{\infty}$-rings. For example, suppose we wish to understand the space of maps $\bHom_{\EInfty}( A, B)$ between two connective $E_{\infty}$-rings $A$ and $B$. This space can be realized as the homotopy inverse limit of the mapping spaces $\bHom_{\EInfty}( A, \tau_{\leq n} B)$. In the case $n= 0$, this is simply the discrete set of ring homomorphisms from $\pi_0 A$ to $\pi_0 B$. For $n > 0$, Corollary
\ref{subte} implies the existence of a pullback diagram
$$ \xymatrix{ \tau_{\leq n} B \ar[r] \ar[d] & \tau_{\leq n-1} B \ar[d] \\
\tau_{\leq n-1} B \ar[r] & \tau_{\leq n-1} B \oplus ( \pi_n B)[n+1]. }$$
This reduces us to the study of $\bHom_{\EInfty}(A, \tau_{\leq n-1} B)$ and the
``linear'' problem of understanding derivations from $A$ into $(\pi_n B)[n+1]$. This linear problem is controlled by the cotangent complex of $A$. We will apply this observation in \S \ref{classet}.

\begin{proof}[Proof of Theorem \ref{exsqze}]
We let $\CDer' \subseteq \CDer$ be the full subcategory of $\CDer$ spanned by those
derivations $\eta: A \rightarrow M$ where $A$ is connective, $\CDer'' \subseteq \CDer'$ the full subcategory spanned by those derivations where $A$ is connective and $M$ is $(2n+1)$-truncated, and
$\Fun'( \Delta^1, \EInfty)$ the the full subcategory of $\Fun( \Delta^1, \EInfty)$ spanned by
those morphisms $\widetilde{A} \rightarrow A$ such that $A$ is connective.

Let $G: \CDer'' \rightarrow \Fun'( \Delta^1, \EInfty)$ denote the restriction of $\Phi$ to
$\CDer'' \subseteq \CDer$. The functor $G$ factors as a composition
$$ \CDer'' \stackrel{G_0}{\rightarrow} \CDer' \stackrel{G_1}{\rightarrow} \calE_1
\stackrel{G_2}{\rightarrow} \calE
\stackrel{G_3}{\rightarrow} \calE_0 \stackrel{G_4}{\rightarrow} \Fun'( \Delta^1, \EInfty),$$
where the functors $G_i$ and their targets can be described as follows:

\begin{itemize}
\item[$(G_0)$] The functor $G_0$ is simply the inclusion of $\CDer''$ into
$\CDer'$. This functor admits a left adjoint $F_0$, which carries a derivation
$\eta: A \rightarrow M$ to the induced derivation $\eta': A \rightarrow \tau_{\leq 2n+1} M$.

\item[$(G_1)$] Let $\calE_1$ denote the full subcategory of 
$$ \Fun( \Lambda^2_2, \calM^{T}) \times_{ \Fun( \Lambda^2_2, \Delta^1 \times \EInfty) } \EInfty$$
spanned by those diagrams
$$ \xymatrix{ & A \ar[d] \\
0 \ar[r] & M}$$
with the following properties:
\begin{itemize}
\item[$(i)$] The $E_{\infty}$-ring $A$ is connective.
\item[$(ii)$] The object $0$ belongs $T_{\EInfty} \subseteq \calM^{T}$, and
is a zero object of the fiber $T_{\EInfty} \times_{\EInfty} \{A\} \simeq \Mod_{A}$.
\item[$(iii)$] The object $M$ belongs to $T_{\EInfty} \times_{ \EInfty} \{A\} \simeq \Mod_{A}$.
\end{itemize}
An object
$$ \sigma \in \Fun( \Lambda^2_2, \calM^{T}) \times_{ \Fun( \Lambda^2_2, \Delta^1 \times \EInfty) } \EInfty$$
belongs to $\calE_1$ if and only if the restriction $\sigma | \Delta^1$ belongs to $\CDer'$, and
$\sigma$ is a $p$-left Kan extension of $\sigma | \Delta^1$. It follows from Proposition
\toposref{lklk} that the restriction map $F_1: \calE \rightarrow \CDer'$ is a trivial Kan fibration.
We let $G_1$ denote any section to $F_1$. 

\item[$(G_2)$]
We let $\calE$ denote the full subcategory of
$$\Fun( \Delta^1 \times \Delta^1, \calM^{T}) \times_{ \Fun( \Delta^1 \times \Delta^1), \Delta^1 \times \EInfty} \Fun(\Delta^1 \times \EInfty)$$
spanned by those diagrams
$$ \xymatrix{ \widetilde{A} \ar[r]^{f} \ar[d] & A \ar[d]^{\eta} \\
0 \ar[r] & M }$$
which satisfy conditions $(i)$, $(ii)$, and $(iii)$ above (so that $f: \widetilde{A} \rightarrow A$
is a map of $E_{\infty}$-rings). Let $F_2: \calE \rightarrow \calE_1$ denote the restriction map. Then $F_2$ admits a section $G_2$ which is right adjoint to $F_2$, given by the formation of pullback diagrams.

\item[$(G_3)$] Let $\calE_0$ denote the full subcategory of
$$ \Fun( \Lambda^2_0, \calM^{T}) \times_{ \Fun( \Lambda^2_0, \Delta^1 \times \EInfty)} \Fun( \Delta^1, \EInfty)$$
spanned by those diagrams
$$ \xymatrix{ \widetilde{A} \ar[r] \ar[d] & A \\
0 & }$$
satisfying conditions $(i)$ and $(ii)$. 
We let $G_3$ denote the evident restriction map $\calE \rightarrow \calE_0$.
This functor has a left adjoint $F_3$, which carries the above diagram to
$$ \xymatrix{ \widetilde{A} \ar[r] \ar[d] & A \ar[d] \\
0 \ar[r] & L_{A / \widetilde{A} }. }$$

\item[$(G_4)$] Using Proposition \toposref{lklk}, we deduce that
the forgetful functor $G_4: \calE_1 \rightarrow \Fun'( \Delta^1, \EInfty)$ 
is a trivial Kan fibration. Let $F_4$ denote any section to $G_4$.
\end{itemize}

The functor $G$ admits a left adjoint $F$, given by composing the functors
$\{ F_i \}_{0 \leq i \leq 4}$. Using the descriptions above, we can describe $F$ as follows: it carries a morphism $\widetilde{A} \rightarrow A$ to the derivation
$d: A \rightarrow \tau_{\leq 2n+1} L_{ A / \widetilde{A} }$ classified by the composite map
$$ L_A \rightarrow L_{A / \widetilde{A} } \rightarrow \tau_{\leq 2n+1} L_{A / \widetilde{A}}.$$

We wish to show that the adjunction between $F$ and $G$ restricts to an equivalence
between $\CDer^{(n)}$ and $\Fun^{(n)}( \Delta^1, \EInfty)$. For this, we must show four things:
\begin{itemize}
\item[$(a)$] The functor $G$ carries $\CDer^{(n)}$ into $\Fun^{(n)}( \Delta^1, \EInfty)$. 

\item[$(b)$] The functor $F$ carries $\Fun^{(n)}( \Delta^1, \EInfty)$ into
$\CDer^{(n)}$.

\item[$(c)$] For every object $f \in \Fun^{(n)}(\Delta^1, \EInfty)$, 
the unit map $$u_{f}: f \rightarrow (G \circ F)(f)$$
is an equivalence.

\item[$(d)$] For every object $\eta \in \CDer^{(n)}$, the counit map
$v_{\eta}: (F \circ G)(\eta) \rightarrow \eta$ is an equivalence.
\end{itemize}

We first prove $(a)$.
Let $\eta: A \rightarrow M$ be an object of $\CDer^{(n)}$, and $f: A^{\eta} \rightarrow A$ the
corresponding square-zero extension. The kernel $\ker(f)$ can be identified with
$M[-1]$ as an $\widetilde{A}$-module. Invoking $(ii)$, we deduce that
$\pi_{i} \ker(f)$ vanishes unless $n \leq i \leq 2n$. Moreover, the multiplication map
$$ \widetilde{A} \otimes \ker(f) \rightarrow \ker(f)$$
factors through $A \otimes \ker(f)$, and is therefore nullhomotopic when restricted to
$\ker(f) \otimes \ker(f)$. This implies the vanishing of the multiplication map
$\pi_{n} \ker(f) \times \pi_{n} \ker(f) \rightarrow \pi_{2n} \ker(f)$, which guarantees that
$f$ is an $n$-small extension (see Remark \ref{spuzzy}).

We now prove $(b)$ and $(c)$. Let $f: \widetilde{A} \rightarrow A$ be an $n$-small extension,
$\eta: A \rightarrow \tau_{\leq 2n+1} L_{A/ \widetilde{A}}$ the image of
$f$ under the functor $F$. The unit map $u$ determines a commutative
diagram
$$ \xymatrix{ \ker(f) \ar[r] \ar[d]^{g} & \widetilde{A} \ar[r]^{f} \ar[d]^{u_{f }} & A \ar[d]^{\id} \\
\tau_{\leq 2n} ( L_{A/ \widetilde{A}} [-1] ) \ar[r] & A^{\eta} \ar[r] & A }$$
in the $\infty$-category of spectra. To prove $(c)$, we must show that
$u_{f}$ is an equivalence. To prove $(b)$, we must show that
$\tau_{\leq 2n} ( L_{A / \widetilde{A}}[-1])$ is $n$-connective and
$(2n)$-truncated. In either case, it will suffice to show that $g$ is an equivalence.
Since $f$ is an $n$-small extension, $\pi_{i} \ker(f)$ vanishes for $i > 2n$. It will therefore suffice to show that the map $g$ induces an isomorphism 
$$\pi_{i} \ker(f) \rightarrow \pi_{i} \tau_{\leq 2n} ( L_{A/ \widetilde{A}} [-1]) $$
for $i \leq 2n$. 

We observe that the map $g$ factors as a composition
$$ \ker(f) \stackrel{g'}{\rightarrow} \ker(f) \otimes_{ \widetilde{A} } A
\stackrel{g''}{\rightarrow} L_{A/ \widetilde{A} }[-1] \stackrel{g'''}{\rightarrow}
\tau_{\leq 2n} (L_{ A/ \widetilde{A} } [-1]),$$
where $g''$ is the map described in Theorem \ref{tulbas}. 
It will therefore suffice to show that each of the induced maps
$$ \pi_{i} \ker(f) \stackrel{ g'_i }{\rightarrow}
\pi_{i} ( \ker(f) \otimes_{ \widetilde{A}} A) \stackrel{ g''_i }{\rightarrow}
\pi_{i} (L_{A / \widetilde{A} }[-1]) \stackrel{ g'''_i }{\rightarrow}
\pi_{i} \tau_{\leq 2n} ( L_{A/ \widetilde{A} }[-1] )$$
is an isomorphism for $i \leq 2n$. For $g'''_i$, this is clear, and
for $g''_i$ it follows from Theorem \ref{tulbas}. To analyze the map $g'_i$ we use the long
exact sequence
$$ \pi_{i} (\ker(f) \otimes_{\widetilde{A} } \ker(f) )
\rightarrow \pi_{i} \ker(f) \stackrel{g'_i}{\rightarrow} \pi_{i} (\ker(f) \otimes_{ \widetilde{A} } A)
\rightarrow \pi_{i-1} ( \ker(f) \otimes_{ \widetilde{A} } \ker(f) )$$
Since $\ker(f)$ is $n$-connective, the spectrum
$\ker(f) \otimes_{ \widetilde{A} } \ker(f)$ is $(2n)$-connective, so that the outer terms
vanish for $i < 2n$. If $i = 2n$, we conclude that $g'_i$ is surjective, and that the
kernel of $g'_i$ is generated by the image of the multiplication map
$$ \pi_{n} \ker(f) \times \pi_{n} \ker(f) \rightarrow \pi_{2n} \ker(f).$$
Since $f$ is an $n$-small extension, this multiplication map is trivial so that
$g'_i$ is an isomorphism as desired. This completes the proof of $(b)$ and $(c)$.

We now prove $(d)$. The map $G( v_{\eta} )$ admits a right homotopy inverse, given by the unit map
$u_{ G(\eta)}$. Assertion $(a)$ shows that $G(\eta)$ is an $n$-small extension, so that
$u_{ G(\eta)}$ is an equivalence by $(c)$. It follows that $G( v_{\eta})$ is an equivalence.
To complete the proof, it will suffice to show that $G$ is conservative. Suppose given
a morphism $\alpha: (\eta: A \rightarrow M) \rightarrow (\eta': B \rightarrow N)$ in $\CDer'$ such that
$G(\alpha)$ is an equivalence in $\Fun'( \Delta^1, \EInfty)$. 
We obtain a map of distinguished triangles
$$ \xymatrix{ M[-1] \ar[r] \ar[d] & A^{\eta} \ar[r] \ar[d]^{h} & A \ar[d]^{h'} \ar[r] & M^{h''} \ar[d] \\
N[-1] \ar[r] & B^{\eta'} \ar[r] & B \ar[r] & N }$$
in the homotopy category of spectra. 
Since $G(\alpha)$ is an equivalence, the maps $h$ and $h'$ are equivalences. It follows that $h''$ is an equivalence as well, so that $\alpha$ is an equivalence as desired.
\end{proof}

\subsection{Deformation Theory of $E_{\infty}$-Rings}\label{defthe}

Let $f: \widetilde{A} \rightarrow A$ be a square-zero extension between connective $E_{\infty}$-rings.
Our goal in this section is to show that we can use deformation theory to recover the $\infty$-category of connective $\widetilde{A}$-algebras from the $\infty$-category of connective $A$-algebras.
To be more precise, we need to introduce a bit of terminology.

\begin{notation}
In this section, we continue to follow the conventions of Notation \ref{exxy}:
we let $p: \calM^{T} \rightarrow \Delta^1 \times \EInfty$ denote a tangent correspondence
to the $\infty$-category of $E_{\infty}$-rings, $\CDer$ the $\infty$-category of derivations
in $\EInfty$, and $\widetilde{\CDer}$ the $\infty$-category of extended derivations in $\EInfty$. 
\end{notation}

\begin{notation}
We define a subcategory $\CDer^{+} \subseteq \CDer$ as follows:
\begin{itemize}
\item[$(i)$] An object $\eta: A \rightarrow M$ of $\CDer$ belongs
to $\CDer^{+}$ if and only if $A$ and $M[-1]$ are connective.

\item[$(ii)$] Let $f: (\eta: A \rightarrow M) \rightarrow (\eta': B \rightarrow N)$ be a morphism
in $\CDer$ between objects which belong to $\CDer^{+}$. Then $f$ belongs to
$\CDer^{+}$ if and only if the induced map $M \otimes_{A} B \rightarrow N$ is an equivalence
of $B$-modules.
\end{itemize}
\end{notation}

\begin{proposition}\label{swum}
Let $\eta: A \rightarrow M$ be an object of $\CDer^{+}$. Then the functor $\Phi$ of Notation \ref{stubble}
induces an equivalence of $\infty$-categories
$$ \CDer^{+}_{\eta/} \rightarrow (\EInfty)^{\conn}_{ A^{\eta}/}.$$
\end{proposition}

\begin{remark}
Let $\eta: A \rightarrow M$ be an object of $\CDer^{+}$. According to Proposition \ref{swum}, every
connective $A^{\eta}$-algebras has the form $B^{\eta'}$, where $B$ is an $A$-algebra and $\eta'$ is a derivation which fits into a commutative diagram
$$ \xymatrix{ A \ar[r]^{\eta} \ar[d] & M \ar[d] \\
B \ar[r]^{\eta'} & M \otimes_{A} B. }$$
We interpret this result as follows:
suppose we are given a connective $A$-algebra $B$, and we wish to lift $B$ to a commutative algebra defined over the square-zero extension $A^{\eta}$. In this case, it is necessary and sufficient to produce a commutative diagram as indicated:
$$ \xymatrix{ L_{A} \ar[r]^{\eta} \ar[d] & M \ar[d] \\
L_{B} \ar[r]^{\eta'} & B \otimes_{A} M. }$$
Here we encounter an obstruction to the existence of $\eta'$ lying in the abelian group
$\Ext^2( L_{B/A}, B \otimes_{A} M)$. Provided that this obstruction vanishes, the
collection of equivalence classes of extensions is naturally a torsor for the abelian group
$\Ext^1( L_{B/A}, B \otimes_{A} M)$. This is a precise analogue (and, as we will see later, a generalization) of the situation in the classical deformation theory of algebraic varieties. Suppose given a smooth morphism of smooth, projective varieties $X \rightarrow Y$. Given a first-order deformation
$\widetilde{Y}$ of $Y$, we encounter an obstruction in $\HH^2( X; T_{X/Y})$ to extending
$\widetilde{Y}$ to a first-order deformation of $X$. If this obstruction vanishes, then the set of
isomorphism classes of extensions is naturally a torsor for the cohomology group
$\HH^1( X; T_{X/Y})$.
\end{remark}

Proposition \ref{swum} is an immediate consequence of a more general result which we will formulate below (Theorem \ref{h2h2}). First, we need a bit more notation.

\begin{notation}\label{stablecamp}
We define a subcategory $\Fun^{+}( \Delta^1, \EInfty)$ as follows:
\begin{itemize}
\item[$(i)$] An object $f: \widetilde{A} \rightarrow A$ of $\Fun( \Delta^1, \EInfty)$ belongs to
$\Fun^{+}( \Delta^1, \EInfty)$ if and only if both $A$ and $\widetilde{A}$ are connective, and $f$
induces a surjection $\pi_0 \widetilde{A} \rightarrow \pi_0 A$.

\item[$(ii)$] Let $f,g \in \Fun^{+}( \Delta^1, \EInfty)$, and let
$\alpha: f \rightarrow g$ be a morphism in $\Fun( \Delta^1, \EInfty)$. Then $\alpha$
belongs to $\Fun^{+}( \Delta^1, \EInfty)$ if and only if it classifies
a pushout square in the $\infty$-category $\EInfty$.
\end{itemize}
\end{notation}

\begin{theorem}\label{h2h2}
Let
$\Phi: \CDer(\EInfty) \rightarrow \Fun( \Delta^1, \EInfty)$ be the functor
defined in Notation \ref{stubble}. Then $\Phi$ induces a functor
$\Phi^{+}: \CDer^{+} \rightarrow \Fun^{+}( \Delta^1, \EInfty)$. Moreover, the
functor $\Phi^{+}$ factors as a composition
$$ \CDer^{+} \stackrel{\Phi^{+}_0}{\rightarrow} \overline{\CDer}^{+} \stackrel{\Phi^{+}_1}{\rightarrow} \Fun^{+}( \Delta^1, \EInfty),$$
where $\Phi^{+}_0$ is an equivalence of $\infty$-categories and $\Phi^{+}_1$ is a left fibration.
\end{theorem}

\begin{proof}[Proof of Proposition \ref{swum}]
Let $\eta: A \rightarrow M$ be an object of $\CDer^{+}$, and let
$f: A^{\eta} \rightarrow A$ be the image of $\eta$ under the functor $\Phi$.
Theorem \ref{h2h2} implies that $\Phi$ induces an equivalence
$\CDer^{+}_{\eta/} \rightarrow \Fun^{+}( \Delta^1, \EInfty)_{f/}$. 
It now suffices to observe that the evaluation map
$\Fun^{+}( \Delta^1, \EInfty)_{f/} \rightarrow (\EInfty)_{\widetilde{A}/}^{\conn}$
is a trivial Kan fibration.
\end{proof}

The proof of Theorem \ref{h2h2} will require a few lemmas.

\begin{lemma}\label{susk}
Let
$$ \xymatrix{ & Y \ar[dr]^{f''} & \\
X \ar[ur]^{f'} \ar[rr]^{f} & & Z }$$
be a commutative diagram in the $\infty$-category $\CDer$. If
$f$ and $f'$ belong to $\CDer^{+}$, then so does $f''$.
\end{lemma}

\begin{proof}
This follows immediately from Proposition \toposref{protohermes}.
\end{proof}

\begin{lemma}\label{susp}
Let $f: (\eta: A \rightarrow M) \rightarrow (\eta': B \rightarrow N)$ be a morphism
in $\CDer^{+}$. If the induced map $A^{\eta} \rightarrow B^{\eta}$ is an equivalence of
$E_{\infty}$-rings, then $f$ is an equivalence.
\end{lemma}

\begin{proof}
The morphism
$f$ determines a map of distinguished triangles
$$ \xymatrix{ A^{\eta} \ar[r] \ar[d] & A \ar[r] \ar[d]^{f_0} & M \ar[d]^{f_1} \ar[r] & A^{\eta}[1] \ar[d] \\
B^{\eta'} \ar[r] & B \ar[r] & N \ar[r] & B^{\eta'}[1] }$$
in the homotopy category of spectra. Since the outer vertical maps are equivalences, we obtain
an equivalence $\alpha: \coker(f_0) \simeq \coker(f_1)$. To complete the proof, it will suffice to show that $\coker(f_0)$ vanishes. Suppose otherwise. Since $\coker(f_0)$ is connective, there exists some smallest integer $n$ such that $\pi_{n} \coker(f_0) \neq 0$. In particular, $\coker(f_0)$ is $n$-connective.

Since $f$ induces an equivalence $B \otimes_{A} M \rightarrow N$, the
cokernel $\coker(f_1)$ can be identified with $\coker(f_0) \otimes_{A} M$.
Since $M$ is $1$-connective, we deduce that $\coker(f_1)$ is $(n+1)$-connective.
Using the equivalence $\alpha$, we conclude that $\coker(f_0)$ is $(n+1)$-connective, which contradicts our assumption that $\pi_{n} \coker(f_0) \neq 0$.
\end{proof}

\begin{lemma}\label{dwwem}
Let $\calD_0 \subseteq \calD$ be small $\infty$-categories, and let $p: \calM \rightarrow \calC$ be a presentable fibration. Then:

Then:
\begin{itemize}
\item[$(1)$] The induced map 
$$q: \Fun( \calD, \calM) \rightarrow \Fun( \calD, \calC) \times_{ \Fun( \calD_0, \calC) }
\Fun( \calD_0, \calM) $$
is a coCartesian fibration.

\item[$(2)$] A morphism in $\Fun( \calD, \calM)$ is $q$-coCartesian if and only if the
induced functor $f: \calD \times \Delta^1 \rightarrow \calM$ is a 
$p$-left Kan extension of its restriction to $(\calD \times \{0\}) \coprod_{ \calD_0 \times \{0\} }
( \calD \times \Delta^1)$. 

\end{itemize}
\end{lemma}

\begin{proof}
The ``if'' direction of $(2)$ follows immediately from Lemma \toposref{kan1}. 
Since every diagram
$$ \xymatrix{ (\calD \times \{0\}) \coprod_{ \calD_0 \times \{0\} }
( \calD \times \Delta^1) \ar[r] \ar@{^{(}->}[d] & \calM \ar[d]^{p} \\
\calD \times \Delta^1 \ar[r] \ar@{-->}[ur] & \calC }$$
admits an extension as indicated, which is a $p$-left Kan extension, assertion
$(1)$ follows immediately. The ``only if'' direction of $(2)$ then follows from the
uniqueness properties of $q$-coCartesian morphisms.
\end{proof}

\begin{proof}[Proof of Theorem \ref{h2h2}]
Form a pullback diagram
$$ \xymatrix{ \widetilde{\CDer}^{+} \ar[r] \ar[d]^{u^{+}} & \widetilde{\CDer} \ar[d]^{u} \\
\CDer^{+} \ar[r] & \CDer(\EInfty). }$$
Since the map $u$ is a trivial Kan fibration, $u^{+}$ is also a trivial Kan fibration.

Let $\calX$ denote the full subcategory of $\Fun( \Delta^1 \times \Delta^1, \EInfty)$ spanned by those diagrams 
$$ \xymatrix{ \widetilde{A} \ar[r] \ar[d] & A \ar[d]^{\alpha} \\
A' \ar[r]^{\beta} & A''. }$$
such that $\alpha$ and $\beta$ are equivalences. The diagonal inclusion
$\Delta^1 \subseteq \Delta^1 \times \Delta^1$ induces a map
$\epsilon: \calX \rightarrow \Fun( \Delta^1, \EInfty)$. Using Proposition \toposref{lklk}, we deduce that
this map is a trivial Kan fibration. The map $\epsilon$ has a section $\upsilon$, which carries a morphism $\widetilde{A} \rightarrow A$ to the commutative diagram
$$ \xymatrix{ \widetilde{A} \ar[r] \ar[d] & A \ar[d]^{\id} \\
A \ar[r]^{\id} & A. }$$
It follows that $\upsilon$ is also an equivalence.

Let $\overline{\CDer}$ denote the full subcategory of $\Fun( \Delta^1 \times \Delta^1, \calM^{T} )$
spanned by those pullback diagrams
$$ \xymatrix{ \widetilde{A} \ar[r] \ar[d] & A \ar[d]^{\eta} \\
0 \ar[r]^{\gamma} & M }$$
such that the objects $\widetilde{A}$ and $A$ belong to $\EInfty \subseteq \calM^{T}$, the
objects $0$ and $M$ belong to $T_{\EInfty}$, the maps $\eta$ and $\gamma$ induce equivalences in $\EInfty$, and $0$ is a $p$-initial object of $T_{\EInfty}$. We have a homotopy pullback diagram
$$ \xymatrix{ \widetilde{\CDer} \ar[r]^{\upsilon'} \ar[d] & \overline{\CDer} \ar[d] \\
\Fun( \Delta^1, \EInfty) \ar[r]^{\upsilon} & \calX. }$$
Since $\upsilon$ is a categorical equivalence, we conclude that $\upsilon'$ is also a categorical equivalence.

The functor $\Phi$ is defined to be a composition
$$ \CDer( \EInfty) \stackrel{s}{\rightarrow} \widetilde{ \CDer}
\stackrel{\upsilon'}{\rightarrow} \overline{\CDer} \stackrel{s'}{\rightarrow} \calX
\stackrel{s''}{\rightarrow} \Fun( \Delta^1, \EInfty),$$
where $s$ is a section to $u$ and $s''$ is the map which carries
a diagram
$$ \xymatrix{ \widetilde{A} \ar[r] \ar[d]^{\alpha} & A \ar[d] \\
A' \ar[r] & A'' }$$
to the map $\alpha$. 
We define
$\Phi^{+}_0$ and $\Phi^{+}_1$ to be the restrictions of $\upsilon' \circ s$ and $s'' \circ s'$, respectively.
To complete the proof, it will suffice to show that $s'$ induces a left fibration of simplicial sets
$\overline{\CDer}^{+} \rightarrow \Fun^{+}( \Delta^1, \EInfty).$
To prove this, we will describe the $\infty$-category $\overline{\CDer}^{+}$ in another way.


Let $\overline{\calD}$ denote the full subcategory of
$\Fun( \Delta^1, \calM^{T} )$ spanned by morphisms of the form $\eta_0: \widetilde{A} \rightarrow 0$, 
satisfying the following conditions:
\begin{itemize}
\item[$(i)$] The object $\widetilde{A}$ belongs to $\EInfty \subseteq \calM^{T}$. 
\item[$(ii)$] Let $f: \widetilde{A} \rightarrow A'$ be the image of $\eta_0$ under the map
$\calM^{T} \rightarrow \EInfty$. Then $\widetilde{A}$ and $A'$ are connective, and
$f$ induces a surjection $\pi_0 \widetilde{A} \rightarrow \pi_0 A'$.
\item[$(iii)$] The object $0$ belongs to $T_{\EInfty} \subseteq \calM^{T}$.
Moreover, $0$ is a zero object of $T_{\EInfty} \times_{ \EInfty} \{A'\} \simeq \Mod_{A'}$.
\end{itemize}

Using Proposition \toposref{lklk}, we deduce that projection map
$\psi_0: \overline{\calD} \rightarrow \Fun'( \Delta^1, \EInfty)$ is a trivial Kan fibration, where
$\Fun'( \Delta^1, \EInfty)$ denotes the full subcategory of $\Fun( \Delta^1, \EInfty)$ spanned by those morphisms $\widetilde{A} \rightarrow A$ which satisfy condition $(ii)$.

Let $\calD$ denote the full subcategory of
$\Fun( \Delta^1 \times \Delta^1, \calM )$
spanned by those diagrams
$$ \xymatrix{ \widetilde{A} \ar[r]^{e} \ar[d]^{\eta_0} & A \ar[d]^{\eta} \\
0 \ar[r]^{\gamma} & M }$$
satisfying properties $(i)$, $(ii)$, and $(iii)$ above, where $A \in \EInfty \subseteq \calM^{T}$ and
$M \in T_{\EInfty} \subseteq \calM^{T}$. 
Restriction to the left half of the diagram yields a forgetful functor
$\psi_1: \calD \rightarrow \overline{\calD}$, which fits into a pullback square
$$ \xymatrix{ \calD \ar[d]^{\psi_1} \ar[r] & \Fun( \Delta^1 \times \Delta^1, \calM )
\ar[d]^{\psi'_1} \\
\overline{\calD} \ar[r] & \Fun( \Delta^1, \calM^{T} ) \times_{ \Fun( \Delta^1, \Delta^1) }
\Fun( \Delta^1 \times \Delta^1, \Delta^1). }$$
Applying Lemma \ref{dwwem} to the presentable fibration $\calM^T \rightarrow \Delta^1$, we conclude that $\psi'_{1}$ is a coCartesian fibration.
It follows that $\psi_1$ is also a coCartesian fibration, and that a morphism in
$\calD$ is $\psi_1$-coCartesian if and only if it satisfies criterion $(2)$ in the statement of
Lemma \ref{dwwem}.

We define subcategories
$\calD_0 \subseteq \calD_1 \subseteq \calD$ as follows:
\begin{itemize}
\item Every object of $\calD$ belongs to $\calD_1$.

\item A morphism $f$ in $\calD$ belongs to $\calD_1$ if and only if 
$(\psi_0 \circ \psi_1)(f)$ belongs to $\Fun^{+}(\Delta^1, \EInfty)$, and $f$ is $\psi_1$-coCartesian.
Since $\psi_0$ is a trivial Kan fibration, this is equivalent to the requirement that 
$f$ is $\psi_1 \circ \psi_0$-coCartesian. 

\item We define $\calD_0$ to be the full subcategory of is the full subcategory of $\calD_1$ spanned by those diagrams
$$ \xymatrix{ \widetilde{A} \ar[r] \ar[d] & A \ar[d]^{\eta} \\
0 \ar[r]^{\gamma} & M }$$
which are pullback diagrams in $\calM^{T}$, such that
$\eta$ and $\gamma$ induce equivalences in $E_{\infty}$.
\end{itemize}

Using Corollary \toposref{relativeKan}, we deduce immediately that $\psi_0 \circ \psi_1$
induces a left fibration $\psi: \calD_1 \rightarrow \Fun^{+}(\Delta^1, \calC).$
To complete the proof, it will suffice to verify the following:

\begin{itemize}
\item[$(1)$] The subcategory $\calD_0 \subseteq \calD_1$ is a {\it cosieve} in
$\calD_1$. That is, if $f: X \rightarrow Y$ is a morphism in $\calD_1$ and $X$ belongs
to $\calD_0$, then $Y$ also belongs to $\calD_0$. It follows immediately that
$\psi$ restricts to a left fibration $\calD_0 \rightarrow \Fun^{+}( \Delta^1, \calC)$.

\item[$(2)$] We have an equality
$\calD_0 = \overline{\CDer}^{+}$ of subcategories of
$\overline{\CDer}$.
\end{itemize}

In order to prove these results, we will need to analyze the structure of a morphism
$f: X \rightarrow Y$ in the $\infty$-category $\calD$ in more detail.
Let us suppose that $X,Y \in \calD$ classify diagrams
$$ \xymatrix{ \widetilde{A} \ar[r] \ar[d] & A \ar[d] & & \widetilde{B} \ar[r] \ar[d] & B \ar[d] \\
0 \ar[r] & M & & 0' \ar[r] & N}$$
in $\calM^{T}$, lying over diagrams
$$ \xymatrix{ \widetilde{A} \ar[r] \ar[d] & A \ar[d] & & \widetilde{B} \ar[r] \ar[d] & B \ar[d] \\
A' \ar[r] & A'' & & B' \ar[r] & B''}$$
in $E_{\infty}$. 
Unwinding the definitions, we see that the morphism $f$ belongs to $\calD_1$ if and only if the following conditions are satisfied:
\begin{itemize}
\item[$(a)$] The morphism $\psi(f)$ belongs to $\Fun^{+}( \Delta^1, \EInfty)$. In other words, the diagram
$$ \xymatrix{ \widetilde{A} \ar[d] \ar[r] & \widetilde{B} \ar[d] \\
A' \ar[r] & B' }$$
is a pushout square of $E_{\infty}$-rings.
\item[$(b)$] The diagram
$$ \xymatrix{ \widetilde{A} \ar[d] \ar[r] & \widetilde{B} \ar[d] \\
A \ar[r] & B }$$
is a pushout square of $E_{\infty}$-rings.
\item[$(c)$] The diagram
$$ \xymatrix{ 0 \ar[r] \ar[d] & M \ar[d]^{j} \\
0' \ar[r] & N }$$
is a pushout square in $T_{\EInfty}$. Unwinding the definitions, this is equivalent to the
requirement that the diagram
$$ \xymatrix{ A \ar[r] \ar[d] & B \ar[d] \\
A'' \ar[r] & B'' }$$
is a pushout square of $E_{\infty}$-rings, and that the induced map
$M \otimes_{A''} B'' \rightarrow N$ is an equivalence of $B''$-modules.
\end{itemize}
We observe that $(b)$ and $(c)$ are simply a translation of the requirement that
$f$ satisfies criterion $(2)$ of Lemma \ref{dwwem}.

We now prove $(1)$. Suppose that $X \in \calD_0$; we wish to prove that $Y \in \calD_0$.
It follows from $(c)$ that the map $B \rightarrow B''$ is an equivalence.
To prove that the map $B' \rightarrow B''$ is an equivalence, we consider 
the commutative diagram
$$ \xymatrix{ \widetilde{A} \ar[r] \ar[d] & A' \ar[r] \ar[d] & A'' \ar[d] \\
\widetilde{B} \ar[r] & B' \ar[r] & B''. }$$
From $(a)$ we deduce that the left square is a pushout, and from $(b)$ and $(c)$ together
we deduce that the large rectangle is a pushout. It follows that the right square is a pushout as well.
Since the map $A' \rightarrow A''$ is an equivalence (in virtue of our assumption that
$X \in \calD_0$), we conclude that $B' \rightarrow B''$ is an equivalence as desired.

To complete the proof that $Y \in \calD_0$, it will suffice to show that $Y$ is a pullback diagram. 
This is equivalent to the assertion that the induced diagram $Y'$:
$$ \xymatrix{ \widetilde{B} \ar[r] \ar[d] & B' \ar[d] \\
B \ar[r] & B'' \oplus N }$$
is a pullback diagram of commutative $\widetilde{B}$-algebras. Since the forgetful functor
$\CAlg( \Mod_{ \widetilde{B}} ) \rightarrow \Mod_{ \widetilde{B} }$ preserves limits, it will suffice to show that $Y'$ is a pullback diagram in the $\infty$-category of $\widetilde{B}$-modules. 

Let $X'$ denote the diagram
$$ \xymatrix{ \widetilde{A} \ar[r] \ar[d] & A' \ar[d] \\
A \ar[r] & A'' \oplus M}$$
determined by $X$. Since $X \in \calD_0$, $X'$ is a pullback diagram of
$E_{\infty}$-rings, and therefore a pullback diagram of $\widetilde{A}$-modules. 
Since the relative tensor product functor $\otimes_{ \widetilde{A}} \widetilde{B}$ is exact, it will suffice to show that the map $f: X \rightarrow Y$ induces an equivalence
$X' \otimes_{ \widetilde{A} } \widetilde{B} \rightarrow Y'$. In other words, it suffices to show that
each of the induced diagrams
$$ \xymatrix{ \widetilde{A} \ar[r]\ar[d] & A \ar[d] & \widetilde{A} \ar[r] \ar[d] & A' \ar[d]  & \widetilde{A} \ar[d] \ar[r] & A'' \oplus M \ar[d] \\
\widetilde{B} \ar[r] & B & \widetilde{B} \ar[r] & B' & \widetilde{B} \ar[r] & B'' \oplus N }$$
is a pushout square of $E_{\infty}$-rings. For the left and middle squares, this follows from
$(a)$ and $(b)$.
The rightmost square fits into a commutative diagram
$$ \xymatrix{ \widetilde{A} \ar[r] \ar[d] & A \ar[r] \ar[d] & A'' \oplus M \ar[d] \\
\widetilde{B} \ar[r] & B \ar[r] & B'' \oplus N }$$
where the left part of the diagram is a pushout square by $(b)$ and the right square
is a pushout by $(c)$. This completes the proof that $Y \in \calD_0$, so that $\calD_0 \subseteq \calD_1$ is a cosieve as desired.

We now prove $(2)$. We first show that the subcategories
$$ \calD_0, \overline{\CDer}^{+} \subseteq \overline{\CDer}$$
consist of the same objects. Let $X \in \overline{\CDer}$ be given by a diagram
$$ \xymatrix{ \widetilde{A} \ar[r] \ar[d] & A' \ar[d] \\
0 \ar[r] & M,  }$$
projecting to a diagram
$$ \xymatrix{ \widetilde{A} \ar[r] \ar[d] & A' \ar[d] \\
A \ar[r] & A'' }$$
in the $\infty$-category $\EInfty$. Then $X$ belongs to $\calD_0$ if and only if
the both $A$ and $\widetilde{A}$ are connective, and the map $\pi_0 \widetilde{A} \rightarrow \pi_0 A$
is surjective. On the other hand, $X$ belongs to $\widetilde{\CDer}^{+}$ if and only if both
$A'$ and $M[-1]$ are connective. The equivalence of these conditions follows immediately from the
observation that $A$ and $A'$ are equivalent, and the long exact sequence of homotopy groups associated to the exact triangle of spectra
$$ M[-1] \rightarrow \widetilde{A} \rightarrow A \rightarrow M.$$

Now let us suppose that $f: X \rightarrow Y$ is a morphism in $\overline{\CDer}$, where
both $X$ and $Y$ belong to $\calD_0$. We wish to show that $f$ belongs to
$\calD_0$ if and only if $f$ belongs to $\overline{\CDer}^{+}$. We observe that
$f$ belongs to $\calD_0$ if and only if $f$ satisfies the conditions $(a)$, $(b)$, and $(c)$ described above. On the other hand, $f$ belongs to $\overline{\CDer}^{+}$ if and only if
the induced map $M \otimes_{A''} B'' \rightarrow N$ is an equivalence. Since this follows immediately from condition $(c)$, we conclude that we have inclusions
$$ \calD_0 \subseteq \widetilde{\CDer}^{+} \subseteq \widetilde{\CDer}(\EInfty).$$

To prove the reverse inclusion, let $f: X \rightarrow Y$ be a morphism in
$\overline{\CDer}^{+}$. We wish to show that $f$ belongs to $\calD_0$. In other words, we must show that $f$ satisfies conditions $(a)$, $(b)$, and $(c)$. Since the maps
$$ A \rightarrow A'' \leftarrow A'$$
$$ B \rightarrow B'' \leftarrow B'$$
are equivalences, condition $(c)$ is automatic and conditions $(a)$ and $(b)$ are equivalent to one another. We are therefore reduced to the problem of showing that the diagram
$$ \xymatrix{ \widetilde{A} \ar[r] \ar[d] & \widetilde{B} \ar[d] \\
A \ar[r] & B }$$
is a pushout square.

The image $\Phi(f)$ can be factored as a composition $g' \circ g''$, corresponding to a diagram
$$ \xymatrix{ \widetilde{A} \ar[d] \ar[r] & \widetilde{B} \ar[r]^{\id} \ar[d] & \widetilde{B} \ar[d] \\
A \ar[r] & C \ar[r] & B,}$$
where the left square is a pushout. Let $f': X \rightarrow Z$ be a $\psi$-coCartesian
lift of $g'$, so that $f$ is homotopic to some composition
$X \stackrel{f'}{\rightarrow} Z \stackrel{f''}{\rightarrow} Y.$
We observe that $f'$ belongs to $\calD_0$. It will therefore suffice to show that
$f''$ belongs to $\calD_0$ as well. Lemma \ref{susk} implies that
$f''$ belongs to $\overline{\CDer}^{+}$. We may therefore replace $f$ by $f''$ and thereby reduce to the situation where $f$ induces an equivalence $\widetilde{A} \rightarrow \widetilde{B}$. In this case,
condition $(a)$ is equivalent to the assertion that $f$ induces an equivalence $A \rightarrow B$, which follows from Lemma \ref{susp}.
\end{proof}

\subsection{Classification of \Etale Algebras}\label{classet}

Our goal in this section is to prove the following result:

\begin{theorem}\label{turncoat}
Let $A$ be an $E_{\infty}$-ring, and let $(\EInfty)^{\mathet}_{A/}$ denote the full subcategory of
$( \EInfty)_{A/}$ spanned by the \etale morphisms $A \rightarrow B$. Then
the functor $B \mapsto \pi_0 B$ induces an equivalence of $( \EInfty)_{A/}^{\mathet}$ with
$($the nerve of$)$ the ordinary category of \etale $\pi_0 A$-algebras.
\end{theorem}

The proof will occupy the remainder of this section. 
Recall that if $\calC$ is a presentable $\infty$-category, then we say that
{\it Postnikov towers in $\calC$ are convergent} if the canonical functor
$$ \calC \rightarrow \varprojlim \{ \tau_{\leq n} \calC \}_{n \geq 0}$$
is an equivalence of $\infty$-categories (we refer the reader to \S \toposref{truncintro} for a more detailed discussion of this condition).

\begin{example}\label{ape2}
Let $\calC$ be a presentable $\infty$-category equipped with an accessible t-structure, and
let $\calC_{\geq 0}$ be the full subcategory of $\calC$ spanned by the connective objects. Then
Postnikov towers in $\calC_{\geq 0}$ are convergent if and only if $\calC$ is {\it left complete}
(see \S \stableref{stable6.5}).
\end{example}

\begin{lemma}\label{ape1}
Let $\calC$ be a symmetric monoidal $\infty$-category. Assume that $\calC$ is presentable, and that the tensor product $\otimes: \calC \times \calC \rightarrow \calC$ preserves small colimits in each variable.
If Postnikov towers in $\calC$ are convergent, then Postnikov towers in $\CAlg( \calC)$ are convergent.
\end{lemma}

\begin{proof}
Using Remark \symmetricref{funtime} and Proposition \symmetricref{localjerk2}, we see that
$\Post^{+}(\calC)$ and $\Post(\calC)$ inherit a symmetric monoidal structures. Moreover, we have a 
commutative diagram
$$ \xymatrix{ \Post^{+}( \CAlg(\calC) ) \ar@{=}[r] \ar[d] & \CAlg( \Post^{+}(\calC) ) \ar[d] \\
\Post( \CAlg(\calC) ) \ar@{=}[r] & \CAlg( \Post(\calC) ) }$$
where the horizontal equivalences result from the observation that a map 
$f: A \rightarrow B$ in $\CAlg(\calC)$ exhibits $B$ as an $n$-truncation of $A$ in
$\CAlg(\calC)$ if and only if it exhibits $B$ as an $n$-truncation of $A$ in $\calC$ (see Proposition \symmetricref{superjumpp}).
Because Postnikov towers in $\calC$ are convergent, the right vertical map is an equivalence of $\infty$-categories. It follows that the left vertical map is an equivalence of $\infty$-categories as well.
\end{proof}

\begin{proposition}\label{sumptine}
Postnikov towers are convergent in the $\infty$-category of connective $E_{\infty}$-rings.
\end{proposition}

\begin{proof}
Combine Lemma \ref{ape1}, Example \ref{ape2}, and Proposition \stableref{specster1}.
\end{proof}

\begin{remark}\label{urtur}
Let $f: A \rightarrow B$ be a map of connective $E_{\infty}$-rings. Then
$f$ is \etale if and only if each of the induced maps $\tau_{\leq n} A \rightarrow \tau_{\leq n} B$
is \etalenospace.
\end{remark}

We now sketch the proof of Theorem \ref{turncoat}. First, using Remark \ref{unter}, we may reduce to the case where $A$ is connective. For each $0 \leq n \leq \infty$, let $\calC_{n}$ denote the full subcategory of $\Fun( \Delta^1, \EInfty)$ spanned by those morphisms $f: B \rightarrow B'$ such that
$B$ and $B'$ are connective and $n$-truncated, and let $\calC_{n}^{\mathet}$ denote the full subcategory of $\calC_{n}$ spanned by those morphisms which are also \etalenospace.
Using Proposition \ref{sumptine} and Remark \toposref{sumptime}, we deduce that
$\calC_{\infty}$ is the homotopy inverse limit of the tower
$$ \ldots \rightarrow \calC_{2} \stackrel{ \tau_{\leq 1}}{\rightarrow}
\calC_{1} \stackrel{\tau_{\leq 0}}{\rightarrow} \calC_0.$$
Using Remark \ref{urtur}, we deduce that $\calC_{\infty}^{\mathet}$ is the homotopy inverse limit of the restricted tower
$$ \ldots \rightarrow \calC^{\mathet}_{2} \rightarrow \calC^{\mathet}_{1} \rightarrow \calC^{\mathet}_{0}.$$
Choose a Postnikov tower
$$ A \rightarrow \ldots \rightarrow \tau_{\leq 2} A \rightarrow \tau_{\leq 1} A \rightarrow \tau_{\leq 0} A.$$
For $0 \leq n \leq \infty$, let $\calD_{n}$ denote the fiber product
$\calC_{n}^{\mathet} \times_{ \EInfty} \{ \tau_{\leq n} A \},$ so that we can identify
$\calD_{n}$ with the full subcategory of $\EInfty^{\tau_{\leq n} A/}$ spanned by the
\etale morphisms $f: \tau_{\leq n} A \rightarrow B$. It follows from the above analysis that
$\calD_{\infty}$ is the homotopy inverse limit of the tower
$$ \ldots \rightarrow \calD_{2} \stackrel{g_1}{\rightarrow} \calD_{1} \stackrel{g_0}{\rightarrow} \calD_0. $$
We wish to prove that the truncation functor induces an equivalence $\calD_{\infty} \rightarrow \calD_0$.
For this, it will suffice to show that each of the functors $g_i$ is an equivalence. Consequently, 
Theorem \ref{turncoat} follows from the following slightly weaker result:

\begin{proposition}\label{wiser}
Let $A$ be a connective $E_{\infty}$-ring. Suppose that $A$ is $(n+1)$-truncated for some
$n \geq 0$. Then the truncation functor
$\tau_{\leq n}: \EInfty^{A/} \rightarrow \EInfty^{ \tau_{\leq n} A/}$
restricts to an equivalence from the $\infty$-category of \etale $A$-algebras to the $\infty$-category of \etale $\tau_{\leq n} A$-algebras.
\end{proposition}

Let $A$ be as in the statement of Proposition \ref{wiser}. 
The natural transformation $\id \rightarrow \tau_{\leq n}$ induces, for every $A$-algebra
$B$, a commutative diagram
$$ \xymatrix{ A \ar[r] \ar[d] & \tau_{\leq n} A \ar[d]  \\
B \ar[r] & \tau_{\leq n} B, }$$
which determines a map $B \otimes_{A} \tau_{\leq n} A \rightarrow \tau_{\leq n} B$. If
$B$ is flat over $A$, then this map is an equivalence (Corollary \monoidref{siwe}). Consequently,
when restricted to {\em flat} commutative $A$-algebras, the truncation functor
$\tau_{\leq n}: \EInfty^{A/} \rightarrow \EInfty^{ \tau_{\leq n} A/}$
can be identified with the base change functor $M \mapsto M \otimes_{A} (\tau_{\leq n} A)$. 
We observe that, since $A$ is assumed to be $(n+1)$-truncated, the map
$A \rightarrow \tau_{\leq n} A$ is a square-zero extension (Corollary \ref{subte}). Proposition \ref{wiser} is therefore an immediate consequence of the following result:

\begin{proposition}\label{wisert}
Let $f: \widetilde{A} \rightarrow A$ be a square-zero extension of connective $E_{\infty}$-rings.
Then the relative tensor product functor
$$ \bigdot \otimes_{ \widetilde{A} } A: \EInfty^{\widetilde{A}/} \rightarrow \EInfty^{A/}$$
induces an equivalence from the $\infty$-category of \etale
$\widetilde{A}$-algebras to the $\infty$-category of \etale $A$-algebras.
\end{proposition}

To prove Proposition \ref{wisert}, we first need to establish some facts about the $\infty$-category $\CDer$ introduced in Notation \ref{stablecamp}.

\begin{lemma}\label{preturm}
Suppose given a morphism $\phi: (\eta: A \rightarrow M )
\rightarrow (\eta': B \rightarrow N)$ between derivations in $\EInfty$. Assume that
$A$, $B$, and $M[-1]$ are connective, and that $\phi$ induces an equivalence $M \otimes_{A} B \rightarrow N$ (so that $N[-1]$ is also connective). The following conditions are equivalent:
\begin{itemize}
\item[$(1)$] The map $\phi$ induces a flat map $f: A \rightarrow B$.
\item[$(2)$] The map $\phi$ induces a flat map $f': A^{\eta} \rightarrow B^{\eta'}$.
\end{itemize}
\end{lemma}

\begin{proof}
Theorem \ref{h2h2} implies that the diagram
$$ \xymatrix{ A^{\eta} \ar[d]^{f'} \ar[r] & A \ar[d]^{f} \\
B^{\eta'} \ar[r] & B, }$$
is a pushout square, so the implication $(2) \Rightarrow (1)$ follows from Proposition \monoidref{urwise}. Conversely, suppose that $f$ is flat. We wish to prove that $f'$ is flat. According to Theorem \monoidref{lazard}, it will suffice to show that for every {\em discrete} $A^{\eta}$-module $X$, the
tensor product $X \otimes_{ A^{\eta} } B^{\eta'}$ is discrete. We can identify $X$ with
a (discrete) module over the ordinary commutative ring $\pi_0 A^{\eta}$. Since $\pi_0 A^{\eta}$ is a square-zero extension of $\pi_0 A$, we deduce that $X$ admits a filtration
$$ 0 \rightarrow X' \rightarrow X \rightarrow X'' \rightarrow 0$$
where the kernel of the map $\pi_0 A^{\eta} \rightarrow \pi_0 A$ acts trivially on $X'$ and $X''$.
It will therefore suffice to prove that $X' \otimes_{ A^{\eta} } B^{\eta'}$ and
$X'' \otimes_{ A^{\eta} } B^{\eta'}$ are discrete. In other words, we can reduce to the case where
$X$ admits the structure of an $A$-module. But in this case, we have a canonical equivalence
$$ X \otimes_{ A^{\eta} } B^{\eta} \simeq (X \otimes_{A} A) \otimes_{ A^{\eta}}  B^{\eta' }
\simeq X \otimes_{A} B,$$
and $X \otimes_{A} B$ is discrete in virtue of our assumption that $f$ is flat (Theorem \monoidref{lazard}).
\end{proof}

\begin{lemma}\label{swiper}
Let $\phi: (\eta: A \rightarrow M )
\rightarrow (\eta': B \rightarrow N)$ be as in Lemma \ref{preturm}.
The following conditions are equivalent:
\begin{itemize}
\item[$(1)$] The map $\phi$ induces an \etale map $f: A \rightarrow B$.
\item[$(2)$] The map $\phi$ induces an \etale map $f': A^{\eta} \rightarrow B^{\eta'}$.
\end{itemize}
\end{lemma}

\begin{proof}
In view of Lemma \ref{preturm}, either hypothesis guarantees that the map $f'$ is flat.
It follows that we have a pushout diagram
$$ \xymatrix{ \pi_0 A^{\eta} \ar[d]^{f'_0} \ar[r] & \pi_0 A \ar[d]^{f_0} \\
\pi_0 B^{\eta'} \ar[r] & \pi_0 B }$$
in the category of ordinary commutative rings, where $f'_0$ is flat. Since the upper horizontal map
exhibits $\pi_0 A^{\eta}$ as a square-zero extension of $\pi_0 A$, the map $f'_0$ is \etale if and only if $f_0$ is \etalenospace.
\end{proof}

\begin{notation}
We define a subcategory $\CDer^{\mathet} \subseteq \CDer(\EInfty)$ as follows:
\begin{itemize}
\item[$(1)$] A derivation $\eta: A \rightarrow M$ belongs to
$\CDer^{\mathet}$ if and only if $A$ and $M[-1]$ are connective.
\item[$(2)$] Let $\phi: (\eta: A \rightarrow M) \rightarrow
(\eta': B \rightarrow N)$ be a morphism between derivations belonging to
$\CDer^{\mathet}$. Then $\phi$ belongs to $\CDer^{\mathet}$ if and only 
the map $A \rightarrow B$ is \etalenospace, and $\phi$ induces an equivalence
$M \otimes_{A} B \rightarrow N$.
\end{itemize}

We define a full subcategory $\EInfty^{\mathet} \subseteq \EInfty$ as follows:
\begin{itemize}
\item[$(1)$] An object $A \in \EInfty$ belongs to $\EInfty^{\mathet}$ if and only if $A$ is
connective.
\item[$(2)$] A morphism $f: A \rightarrow B$ of connective $E_{\infty}$-rings belongs to
$\EInfty^{\mathet}$ if and only if $f$ is \etalenospace.
\end{itemize}
\end{notation}

\begin{lemma}\label{swide} 
Let $f: \CDer(\EInfty) \rightarrow \EInfty$ denote the forgetful functor
$(\eta: A \rightarrow M) \mapsto A.$ Then $f$ induces a left fibration
$\CDer^{\mathet} \rightarrow \EInfty^{\mathet}$.
\end{lemma}

\begin{proof}
Fix $0 < i \leq n$; we must show that every lifting problem of the form
$$ \xymatrix{ \Lambda^{n}_{i} \ar@{^{(}->}[d] \ar[r] & \CDer^{\mathet} \ar[d] \\
\Delta^{n} \ar@{-->}[ur] \ar[r] & \EInfty^{\mathet} }$$
admits a solution. Unwinding the definitions, we can identify this with a mapping problem
$$ \xymatrix{ ( (\Lambda^{n}_{i})^{\sharp} \times (\Delta^1)^{\flat} ) \coprod_{ (\Lambda^{n}_{i})^{\sharp} \times \{0\}^{\flat} }
( (\Delta^n)^{\sharp} \times \{0\}^{\flat} ) \ar[r]^-{g} \ar@{^{(}->}[d]^{i} & \calM( \EInfty) \ar[d]^{p} \\
(\Delta^n)^{\sharp} \times (\Delta^1)^{\flat} \ar[r] \ar@{-->}[ur] & (\EInfty \times \Delta^2)^{\sharp} }$$
in the category of marked simplicial sets (see \S \toposref{twuf}); here $\calM(\EInfty)^{\natural}$ denotes the marked simplicial set $( \calM(\EInfty), \calE)$, where $\calM(\EInfty)$ is a tangent correspondence to $\EInfty$ and $\calE$ is the class of
$p$-coCartesian morphisms in $\CMod(\Spectra)$. Here the fact that $g$ preserves
marked edges follows from Remark \ref{spunk}.
Using the dual of Proposition \toposref{dubudu}, we are reduced to showing that $i^{op}$ is a marked anodyne map. In view of Proposition \toposref{markanodprod}, it will suffice to show that the inclusion $(\Lambda^{n}_{n-i})^{\sharp} \subseteq
(\Delta^{n})^{\sharp}$ is marked anodyne, which follows easily from Definition \toposref{markanod}.
\end{proof}

\begin{lemma}\label{trupe}
Let $f: \CDer(\EInfty) \rightarrow \EInfty$ be as in Lemma \ref{swide}, and let
$\eta: A \rightarrow M$ be an object of $\CDer^{\mathet}$.
Then $f$ induces a trivial Kan fibration from
$\CDer^{\mathet}_{\eta/}$ to the $\infty$-category of \etale commutative $A$-algebras.
\end{lemma}

\begin{proof}
Combine Lemma \ref{swide} with Remark \ref{swiide}.
\end{proof}

\begin{proof}[Proof of Proposition \ref{wisert}]
Any square-zero extension $\widetilde{A} \rightarrow A$ is associated to some
derivation $(\eta: A \rightarrow M) \in \CDer^{\mathet}$. Let
$\Phi: \CDer(\EInfty) \rightarrow \Fun( \Delta^1, \EInfty)$ be the functor defined
in Notation \ref{stubble}. 
Let $\Phi_0, \Phi_1: \CDer(\EInfty) \rightarrow \EInfty$ denote the composition of $\Phi$ with evaluation
at the vertices $\{0\}, \{1\} \in \Delta^1$. The functors $\Phi_0$ and $\Phi_1$ induce maps
$$ (\EInfty)_{\widetilde{A}/} \stackrel{\Phi'_0}{\leftarrow} \CDer^{\mathet}_{\eta/}
\stackrel{\Phi'_1}{\rightarrow} (\EInfty)_{A/}.$$
Moreover, the functor $\Phi$ exhibits $\Phi'_1$ as equivalent to the composition of
$\Phi'_0$ with the relative tensor product $\otimes_{ \widetilde{A} } A$. Consequently, it will suffice to prove the following:
\begin{itemize}
\item[$(1)$] The functor $\Phi'_0$ is fully faithful, and its essential image consists precisely of the
\etale commutative $\widetilde{A}$-algebras.
\item[$(2)$] The functor $\Phi'_1$ is fully faithful, and its essential image consists precisely of the \etale commutative $A$-algebras.
\end{itemize}
Assertion $(2)$ follows from Lemma \ref{trupe}, and assertion $(1)$ follows by combining
Proposition \ref{swum}, Lemma \ref{swiper}, and Remark \ref{swiide}.
\end{proof}

Let $A$ be an $E_{\infty}$-ring, let $B$ and $C$ be commutative $A$-algebras, and let
$\phi$ denote the canonical map
$ \bHom_{ (\EInfty)_{A/} }( B, C) \rightarrow \bHom_{ (\EInfty)_{ \pi_0 A/ }}( \pi_0 B, \pi_0 C)$.
Theorem \ref{turncoat} implies that $\phi$ is a homotopy equivalence if $B$ and $C$ are \etale over $A$. In fact, the assumption that $C$ is \etale over $A$ is superfluous:

\begin{proposition}\label{urvine}
Let $f: A \rightarrow B$ be an \etale map of $E_{\infty}$-rings, and let $C$ be an
\etale commutative $A$-algebra. Then the canonical map
$$ \bHom_{(\EInfty)_{A/}}(B, C) \rightarrow \bHom_{ (\EInfty)_{\pi_0 A/} }( \pi_0 B, \pi_0 C)$$
is a homotopy equivalence. In particular, 
$\bHom_{ (\EInfty)_{A/} }(B,C)$ is homotopy equivalent to a discrete space.
\end{proposition}

\begin{remark}
Let $A$ be an $E_{\infty}$-ring, and suppose we are given a map $f: \pi_0 A \rightarrow B'$ in the category of ordinary commutative rings. We can then consider the problem of trying to find
a commutative $A$-algebra $B$ such that $\pi_0 B$ is isomorphic to $B'$ (as a $\pi_0 A$-algebra).
In general, there exist many choices for $B$. There are (at least) two different ways to narrow our selection:

\begin{itemize}
\item[$(i)$] If $f$ is a flat map, then we can demand that $B$ be {\em flat} over $A$. In this case, the homotopy groups of $B$ are determined by the homotopy groups of $A$. Consequently, we have good understanding of mapping spaces
$\bHom_{ (\EInfty)_{A/} }(C, B)$ with {\em codomain} $B$, at least when $C$ is a free $A$-algebra.

\item[$(ii)$] We can demand that the canonical map
$$ \bHom_{ ( \EInfty)_{A/} }( B, C) \rightarrow \bHom_{ ( \EInfty)_{ \pi_0 A/}}( B', \pi_0 C)$$
be a homotopy equivalence for every commutative $A$-algebra $C$. In this case, we have a good understanding of the mapping spaces $\bHom_{ ( \EInfty)_{A/} }( B, C)$ with {\em domain} $B$.
\end{itemize}

It is clear that property $(ii)$ characterized $B$ up to equivalence. If $f$ is \etale, then
Proposition \ref{urvine} asserts that $(i) \Rightarrow (ii)$. Moreover, Theorem \ref{turncoat}
implies the existence of an $A$-algebra $B$ satisfying $(i)$. We therefore have an example satisfying both $(i)$ and $(ii)$; since property $(ii)$ characterizes $B$ up to equivalence, we conclude that
$(i) \Rightarrow (ii)$ (at least when $f$ is \etalenospace). The equivalence of $(i)$ and $(ii)$ makes the theory of \etale extensions of $E_{\infty}$-rings extremely well-behaved.
\end{remark}

\begin{proof}[Proof of Proposition \ref{urvine}]
Let $A_0$, $B_0$, and $C_0$ be connective covers of $A$, $B$, and $C$, respectively.
We have a pushout diagram
$$ \xymatrix{ A_0 \ar[r] \ar[d]^{f_0} & A \ar[d]^{f} \\
B_0 \ar[r] & B }$$
where $f_0$ is \etale (see Remark \ref{unter}). It follows that the induced maps
$$ \bHom_{ (\EInfty)_{A/} }(B, C) \rightarrow \bHom_{ (\EInfty)_{A_0/}}(B_0, C) \leftarrow
\bHom_{ ( \EInfty)_{A_0/} }(B_0, C_0)$$
are homotopy equivalences. We may therefore replace $A$, $B$ and $C$ by their connective covers, and thereby reduce to the case where $A$, $B$, and $C$ are connective.

We have a commutative diagram
$$ \xymatrix{ & \bHom_{ (\EInfty)_{A/} }( B, \pi_0 C) \ar[dr]^{\psi} & \\
\bHom_{ (\EInfty)_{A/} }( B, C) \ar[ur]^{\phi} \ar[rr] & & \bHom_{ (\EInfty)_{\pi_0 A/ }}( \pi_0 B, \pi_0 C) }$$
where the map $\psi$ is a homotopy equivalence. It will therefore suffice to show that $\phi$ is a homotopy equivalence.

Let us say that a map $g: D \rightarrow D'$ of commutative $A$-algebras is {\em good} if
the induced map $\phi_{g}: \bHom_{ ( \EInfty)_{A/} }(B, D) \rightarrow \bHom_{ ( \EInfty)_{A/} }(B,D')$
is a homotopy equivalence. Equivalently, $g$ is good if $e_{B}(g)$ is an equivalence, where
$e_{B}: ( \EInfty)_{A/} \rightarrow \SSet$ is the functor corepresented by $B$. We wish to show
that the truncation map $C \rightarrow \pi_0 C$ is good. We will employ the following chain of reasoning:

\begin{itemize}
\item[$(a)$] Let $D$ be a commutative $A$-algebra, let $M$ be a $D$-module, and let
$g: D \oplus M \rightarrow D$ be the projection. For every map of commutative $A$-algebras
$h: B \rightarrow D$, the homotopy fiber of $\phi_{g}$ over the point $h$ can be identified with
$\bHom_{ \Mod_{B} }( L_{B/A}, M) \simeq \bHom_{ \Mod_{D}}( L_{B/A} \otimes_{B} D, M)$.
Since $f$ is \etale, the relative cotangent complex $L_{B/A}$ vanishes (Proposition \ref{etrel}), so the homotopy fibers of $\phi_{g}$ are contractible. It follows that $\phi_{g}$ is a homotopy equivalence, so that $g$ is good.

\item[$(b)$] The collection of good morphisms is stable under pullback. This follows immediately from the observation that $e_{B}$ preserves limits.

\item[$(c)$] Any square-zero extension is good. This follows from $(a)$ and $(b)$.

\item[$(d)$] Suppose given a sequence of good morphisms
$$ \ldots D_{2} \rightarrow D_{1} \rightarrow D_0.$$
Then the induced map $\lim \{ D_i \} \rightarrow D_0$ is good. This follows again from the observation that $e_{B}$ preserves limits.

\item[$(e)$] For every connective $A$-algebra $C$, the truncation map $C \rightarrow \pi_0 C$
is good. This follows by applying $(d)$ to the Postnikov tower
$$ \ldots \rightarrow \tau_{\leq 2} C \rightarrow \tau_{\leq 1} C \rightarrow \tau_{\leq 0} C \simeq \pi_0 C,$$
which is a sequence of square-zero extensions by Corollary \ref{subte} (it follows from Proposition \ref{sumptine} that the limit of this tower is indeed equivalent to $C$).
\end{itemize}
\end{proof}


\begin{thebibliography}{99}

\bibitem{giraud} Artin, M. {\it Th\'{e}orie des topos et cohomologie
\'{e}tale des sch\'{e}mas.} SGA 4. Lecture Notes in Mathematics
269, Springer-Verlag, Berlin and New York, 1972.

\bibitem{artinmazur} Artin, M. and B. Mazur. {\it \'{E}tale Homotopy.} Lecture Notes in Mathematics 100, Springer-Verlag, Berlin and New York, 1969.

\bibitem{basterra} Basterra, M. {\it \Andre-Quillen cohomology of commutative $S$-algebras.} Journal of Pure and Applied Algebra 144 (1999) no. 2, 111-143.

\bibitem{virtual} Behrend, K. and B. Fantechi. {\it The intrinsic
normal cone.} Inventiones Mathematicae 128 (1997) no. 1, 45-88.

\bibitem{BBD} Beilinson, A. , Bernstein, J. and P. Deligne.
{\it Faisceaux pervers.} Asterisuqe 100, Volume 1, 1982.

\bibitem{bergner} Bergner, J.E. {\it A Model Category Structure on the Category of Simplicial Categories.} Transactions of the American Mathematical Society 359 (2007), 2043-2058.

\bibitem{bergner2} Bergner, J.E. {\it A survey of $(\infty,1)$-categories.} Available at math.AT/0610239

\bibitem{bergner3} Bergner, J.E. {\it Rigidification of algebras over multi-sorted theories.}
Algebraic and Geometric Topoogy 7, 2007.

\bibitem{bergner4} Bergner, J.E. {\it Three models for the homotopy theory of homotopy theories,} Topology 46 (2007), 397-436.

\bibitem{LKM} Bosch, Guntzer, U., and R. Remmert, R. {\it Non-Archimedean Analysis: a Systematic Approach to Rigid Analytic Geometry.} Springer-Verlag, Berlin and Heidelberg, 1984.

\bibitem{bousfieldkan} Bousfield, A.K. and D.M. Kan. {\it Homotopy
limits, completions, and localizations.} Lecture Notes in
Mathematics 304, Springer-Verlag, 1972.

\bibitem{cismoer} Cisinski, D-C and I. Moerdijk. {\it Dendroidal sets as models for homotopy operads.} Available for download as arXiv:0902.1954v1.

\bibitem{combmodel} Dugger, D. {\it Combinatorial model categories have presentations.} Advances in Mathematics 164, 2001, 177-201.

\bibitem{eilenbergsteenrod} Eilenberg, S. and N.E. Steenrod. {\it Axiomatic approach to homology theory.} Proc. Nat. Acad. Sci. U.S.A. 31, 1945, 117-120.

\bibitem{eisenbud} Eisenbud, D. {\it Commutative algebra.} Springer-Verlag, New York, 1995. 

\bibitem{EKMM} Elmendorf, A.D., Kriz, I. , Mandell, M.A., and J.P.
May. {\it Rings, modules and algebras in stable homotopy theory.}
Mathematical Surveys and Monographs 47, American Mathematical
Society, 1997.

\bibitem{bezout} Fulton, W. {\it Algebraic curves.} W.A.
Benjamin, Inc., New York, 1969.

\bibitem{goerssjardine} Goerss, P. and J.F. Jardine. {\it Simplicial Homotopy Theory.} Progress in Mathematics, Birkhauser, Boston, 1999.

\bibitem{goodwillie} Goodwillie, T. {\it Calculus III: Taylor Series.} Geometry and Topology, Volume 7 (2003) 645-711.

\bibitem{stein} Grauert, H. and R. Remmert. {\it Theory of Stein Spaces.} Springer-Verlag, Berlin Heidelberg, 2004.

\bibitem{gunning} Gunning, R. and H. Rossi. {\it Analytic functions of several complex variables.} Prentice-Hall, Englewood Cliffs, N.J, 1965.

\bibitem{hatcher} Hatcher, A. {\it Algebraic Topology}. Cambridge University Press, 2002.

\bibitem{bordism} Hook, E.C. {\it Equivariant cobordism and
duality.} Transactions of the American Mathematical Society 178
(1973) 241-258.

\bibitem{stablemodel} Hovey, M. {\it Model Categories.}
Mathematical Surveys and Monographs 63, AMS, Providence, RI, 1999.

\bibitem{symmetricspectra} Hovey, M., Shipley, B. and J. Smith. {\it Symmetric spectra.} Journal of the American Mathematical Society 13, 2000, no. 1, 149-208.

\bibitem{illusie} Illusie, L. {\it Complexe cotangent et d\'{e}formations
I}. Lecture Notes in Mathematics 239, Springer-Verlag, 1971.

\bibitem{illusie2} Illusie, L. {\it Complexe cotangent et d\'{e}formations
II}. Lecture Notes in Mathematics 283, Springer-Verlag, 1972.

\bibitem{mainjoyal} Joyal, A. {\it Notes on quasi-categories.}

\bibitem{joyalsimp} Joyal, A. {\it Simplicial categories vs. quasi-categories.}

\bibitem{joyalt} Joyal, A. and M. Tierney. {\it Quasi-categories vs. Segal Spaces.} Preprint available at 
math.AT/0607820. 

\bibitem{kerz} Kerz, M. {\it The complex of words and Nakaoka stability.} Homology, Homotopy and Applications, volume 7(1), 2005, pp. 77-85.

\bibitem{knutson} Knutson, D. {\it Algebraic spaces.} Lecture
Notes in Mathematics 203, Springer-Verlag, 1971.

\bibitem{categoricalring} Laplaza, M. {\it Coherence for
distributivity.} Coherence in categories, 29-65. Lecture Notes in
Mathematics 281, Springer-Verlag, 1972.

\bibitem{stacks} Laumon, G. and L. Moret-Bailly. {\it Champs
algebriques.} Springer-Verlag, 2000.

\bibitem{lazard} Lazard, Daniel. {\it Sur les modules plats.} C.R.
Acad. Sci. Paris 258, 1964, 6313-6316.

\bibitem{topoi} Lurie, J. {\it Higher Topos Theory.} Available for download at http://www.math.harvard.edu/~lurie/ .

\bibitem{DAGStable} Lurie, J. {\it Derived Algebraic Geometry I: Stable $\infty$-Categories.} Available for download.

\bibitem{monoidal} Lurie, J. {\it Derived Algebraic Geometry II: Noncommutative Algebra.} Available for download.

\bibitem{symmetric} Lurie, J. {\it Derived Algebraic Geometry III: Commutative Algebra.} Available for download.

\bibitem{deformation} Lurie, J. {\it Derived Algebraic Geometry IV: Deformation Theory.} Available for download.

\bibitem{structured} Lurie, J. {\it Derived Algebraic Geometry V: Structured Spaces.} Available for download.

\bibitem{spectral} Lurie, J. {\it Derived Algebraic Geometry VI: Spectral Schemes.} In preparation.

\bibitem{derivative} Lurie, J. {\it $(\infty,2)$-Categories and the Goodwillie Calculus I.} Available for download.

\bibitem{calculus} Lurie, J. {\it $(\infty,2)$-categories and the Goodwillie Calculus II.} In preparation.

\bibitem{elliptic1} Lurie, J. {\it Elliptic curves in spectral algebraic geometry.} In preparation.

\bibitem{ellipticloop} Lurie, J. {\it Toric varieties, elliptic cohomology at infinity, and loop group representations.} In preparation.

\bibitem{maclane} MacLane, S. {\it Categories for the Working Mathematician.} Second edition. Graduate Txts in Mathematics, 5. Springer-Verlag, New York, 1998.

\bibitem{gabriel} Mitchell, B. {\it A quick proof of the
Gabriel-Popesco theorem.} Journal of Pure and Applied Algebra 20
(1981), 313-315.

\bibitem{neeman} Neeman, A. {\it Triangulated categories.} Annals
of Mathematics Studies, 148. Princeton University Press, 2001.

\bibitem{homotopicalalgebra} Quillen, D. {\it Homotopical Algebra.} Lectures Notes in Mathematics 43, SpringerÐVerlag, Berlin, 1967. 

\bibitem{completesegal} Rezk, C. {\it A model for the homotopy theory of homotopy theory.} Transactions of the American Mathematical Society 35 (2001), no. 3, 973-1007.

\bibitem{homotopyvarieties} Rosicky, J. {\it On Homotopy Varieties.} Advances in Mathematics
214, 2007 no. 2, 525-550.

\bibitem{schwede} Schwede, S. {\it Spectra in model categories and applications to the algebraic
cotangent complex.} Journal of Pure and Applied Algebra 120 (1997), pp.
77-104.

\bibitem{monmod} Schwede, S. and B. Shipley. {\it Algebras and Modules in Monoidal Model Categories.} Proceedings of the London Mathematical Society (80) 2000, 491-511.

\bibitem{schwedeshipley} Schwede, S. and B. Shipley. {\it Stable model categories are categories of modules.} Topology 42, 2003, no. 1, 103-153.

\bibitem{intersection} Serre, Jean-Pierre. {\it Local algebra.}
Springer-Verlag, 2000.

\bibitem{shipley} Shipley, B. {\it A Convenient Model Category for Commutative Ring Spectra.} Homotopy theory: relations with algebraic geometry, group cohomology, and algebraic $K$-theory. Contemp. Math. volume 346 pp. 473-483, American Mathematical Society, Providence, RI, 2004. 

\bibitem{spivak} Spivak, D. {\it Quasi-smooth Derived Manifolds.} PhD dissertation.

\bibitem{srinivas} Srinivas, V. {\it Algebraic K-Theory.} Birkhauser, Boston, 1993.

\bibitem{toen} To\"{e}n, B. {\it Champs affines.} Available for
download: math.AG/0012219.

\bibitem{toenchar} To\"{e}n, B. {\it Vers une axiomatisation de la th\'{e}orie des cat\'{e}gories sup\'{e}riures.} K-theory 34 (2005), no. 3, 233-263.

\bibitem{toen2} To\"{e}n, B. and G. Vezzosi. {\it From HAG to DAG:
derived moduli stacks.} Available for download: math.AG/0210407.

\bibitem{toen3} To\"{e}n, B. and G. Vezzosi. {\it Algebraic
geometry over model categories.} Available for download:
math.AG/0110109.

\bibitem{toen4} To\"{e}n, B. and G. Vezzosi. {\it ``Brave New''
Algebraic Geometry and global derived moduli spaces of ring
spectra.} Available for download: math.AT/0309145.

\bibitem{toen5} To\"{e}n, B. and G. Vezzosi. {\it Segal topoi and
stacks over Segal categories.} Available for download:
math.AG/0212330.

\bibitem{toenK} To\"{e}n, B. and G. Vezzosi. {\it A remark on K-theory and S-categories.} Topology 43, No. 4 (2004), 765-791


\bibitem{verity} Verity, D. {\it Weak complicial sets, a simplicial weak omega-category
theory. Part I: basic homotopy theory.} 

\bibitem{verity2} Verity, D. {\it Weak complicial sets, a simplicial weak omega-category
theory. Part II: nerves of complicial Gray-categories.}

\bibitem{weibel} Weibel, C. {\it An Introduction to Homological Algebra.} Cambridge University Press, 1995.

\end{thebibliography}
\end{document}